\documentclass{article}
\usepackage{amssymb, amsmath, amsthm, mathrsfs, ascmac, color}
\usepackage[pdftex]{graphicx} 
\usepackage[subrefformat=parens]{subcaption}
\usepackage[top=3cm, bottom=3cm, left=3.25cm,right=3.25cm]{geometry}
\usepackage{comment}

\theoremstyle{plain}
\newtheorem{thm}{Theorem}[section]

\newtheorem{prop}[thm]{Proposition}
\newtheorem{lem}{Lemma}
\newtheorem{cor}[thm]{Corollary}

\newtheorem{rmk}{Remark}
\makeatletter

\@addtoreset{equation}{section}
\makeatother
\newcommand{\mcl}{\mathcal}
\newcommand{\mrm}{\mathrm}
\newcommand{\mbb}{\mathbb}
\newcommand{\bs}{\boldsymbol}
\newcommand{\ind}{\bs 1}
\newcommand{\dd}{\mrm d}
\newcommand{\pd}{\partial}
\newcommand{\ee}{\mrm e}
\newcommand{\EE}{\mbb E}
\newcommand{\PP}{\mathbf{P}}
\newcommand{\VV}{\mrm{Var}}
\newcommand{\cov}{\mrm{Cov}}
\newcommand{\OO}{\mathcal{O}}
\newcommand{\oo}{{\scriptsize\text{$\mathcal O$}}}

\newcommand{\FF}{\mathscr F}

\newcommand{\pto}{\stackrel{\PP}{\to}}
\newcommand{\dto}{\stackrel{d}{\to}}

\newcommand{\lqv}{\langle \!\langle}
\newcommand{\rqv}{\rangle \!\rangle}

\allowdisplaybreaks
\title{\textbf{Small diffusivity asymptotics for a linear parabolic SPDE 
in two space dimensions 
}}
\date{}

\author{\textbf{Yozo Tonaki}\thanks{Graduate School of Engineering Science, Osaka University} 
\thanks{Center for Mathematical Modeling and Data Science (MMDS), Osaka University}
\thanks{e-mail: \texttt{y.tonaki.es@osaka-u.ac.jp}}
\and \textbf{Yusuke Kaino}\thanks{Graduate School of Maritime Sciences, Kobe University}
\and \textbf{Masayuki Uchida}$^{* \dag}$\thanks{CREST, Japan Science and Technology Agency}
}

\begin{document}
\maketitle

\begin{abstract}
We consider parameter estimation of the reaction term 
for a second order linear parabolic
stochastic partial differential equation 
in two space dimensions driven by a $Q$-Wiener process under small diffusivity. 
We first construct an estimator of the reaction parameter 
based on continuous spatio-temporal data, 
and then derive an estimator of the reaction parameter
based on high frequency spatio-temporal data
by discretizing the estimator based on the continuous data.
We show that the estimators have consistency and asymptotic normality.
Furthermore, we give simulation results of 
the estimator based on high frequency data.

\begin{center}
\textbf{Keywords and phrases}
\end{center}
High frequency data,
linear parabolic stochastic partial differential equations,
parametric estimation,
reaction parameter,
small diffusive parameter.
\end{abstract}

\section{Introduction}
Stochastic partial differential equations (SPDEs) combine 
partial differential equations 
with spatio-temporal noise to enable mathematical modeling of spatio-temporal phenomena
and are used widely in natural sciences, medicine and economics.
Second order parabolic SPDEs, which cover stochastic heat equations 
and stochastic reaction-diffusion equations, 
are important models and are applied in many fields such as 
geophysical fluid dynamics, mathematical finance, neurobiology and population genetics, 
see
Piterbarg and Ostrovskii \cite{Piterbarg_Ostrovskii1997}, 
Kusuoka \cite{Kusuoka2000}, 
Tuckwell \cite{Tuckwell2013},
Altmeyer et al.\,\cite{Altmeyer_etal2022},
and Dawson \cite{Dawson1975}.
Since SPDEs are thus used in various fields and can describe complex phenomena, 
statistical inference for SPDEs has recently been studied as well as theories of SPDEs.
Refer to Cialenco \cite{Cialenco2018} for existing theories 
of statistical inference for SPDE models.

Let $\mbb T = [0,1]$ and $D = (0,1)^2$.
We treat the following linear parabolic SPDE in two space dimensions
\begin{align}
\dd X_t(y,z) &= 
\biggl\{ \theta_2 \biggl(\frac{\pd^2}{\pd y^2} + \frac{\pd^2}{\pd z^2} \biggr)
+ \theta_1 \frac{\pd}{\pd y} + \eta_1 \frac{\pd}{\pd z}
+\theta_0 \biggr\} X_t(y,z) \dd t 
\nonumber
\\
&\qquad +\sigma \dd W_t^Q(y,z), 
\quad (t,y,z) \in \mbb T \times D
\label{spde0}
\end{align}
with an initial value $X_0$ and
the Dirichlet boundary condition $X_t(y,z) = 0$, 
$(t,y,z) \in \mbb T \times \pd D$, where 
$W_t^Q$ is a $Q$-Wiener process in a Sobolev space on $D$,
$X_0$ is an $L^2(D)$-valued random variable and independent of $W_t^Q$ and 
$(\theta_0, \theta_1, \eta_1, \theta_2, \sigma) \in \mbb R^3 \times (0,\infty)^2$.
We consider parametric estimation of $\theta_0$ in SPDE \eqref{spde0} 
under small diffusivity, which means that
$\theta_2 \to 0$, $\theta_1 \to 0$ and $\eta_1 \to 0$.
Thus, we reparameterize the coefficients of SPDE \eqref{spde0} as follows.
\begin{equation*}
\nu = \theta_2, 
\quad
\kappa = \theta_1/\theta_2, 
\quad
\eta = \eta_1/\theta_2.
\end{equation*}
Let $\nu \in (0,1)$ and $(\kappa, \eta) \in \mbb R^2$ be known parameters, 
and assume that $\nu \to 0$ and $(\kappa,\eta)$ is fixed.
We then study estimation of the reaction parameter $\theta_0$ 
in the parabolic SPDE 
\begin{equation}\label{spde}
\dd X_t(y,z) = (-\nu \mcl A +\theta_0) X_t(y,z) \dd t +\sigma \dd W_t^Q(y,z), 
\quad (t,y,z) \in \mbb T \times D
\end{equation}
under $\nu \to 0$, where the differential operator $\mcl A$ is given by
\begin{equation*}
-\mcl A = \frac{\pd^2}{\pd y^2} + \frac{\pd^2}{\pd z^2}
+ \kappa \frac{\pd}{\pd y} + \eta \frac{\pd}{\pd z}, 
\end{equation*}
$(\theta_0, \sigma) \in \mbb R \times (0,\infty)$ is 
an unknown parameter whose parameter space is a compact convex subset of $\mbb R \times (0,\infty)$,
and $(\theta_0^*, \sigma^*)$ denotes the true value of $(\theta_0, \sigma)$
and belongs to the interior of the parameter space.

Statistical inference for SPDEs has been studied by many researchers, see for example, 
H{\"u}bner et al.\,\cite{Hubner_etal1993},
H{\"u}bner and Rozovskii \cite{Hubner_Rozovskii1995},
Lototsky \cite{Lototsky2003}
and Cialenco and Glatt-Holtz \cite{Cialenco_Glatt-Holtz2011}.
As for discrete observations, 
see Markussen \cite{Markussen2003},
Bibinger and Trabs \cite{Bibinger_Trabs2020},
Cialenco et al.\,\cite{Cialenco_etal2020},
Chong  \cite{Chong2020},
Hildebrandt and Trabs \cite{Hildebrandt_Trabs2021, Hildebrandt_Trabs2023},
Tonaki et al.\,\cite{TKU2023a, TKU2023arXiv},
Bibinger and Bossert \cite{Bibinger_Bossert2023},
Bossert \cite{Bossert2023arXiv}
and references therein.
In the case of SPDE \eqref{spde0}, which has a linear reaction term,
the Fourier coefficients with respect to the eigenfunctions 
of its differential operator are Ornstein-Uhlenbeck dynamics.
Using statistical inference for diffusion processes based on discrete observations,
Tonaki et al.\ \cite{TKU2023a, TKU2023arXiv}
proposed an estimator of $\theta_0$ of SPDE  \eqref{spde0} 
when $W_t^Q$ is a driving process with $\theta_0$.
Tonaki et al.\,\cite{TKU2024} proposed 
an estimator of $\theta_0$ of SPDE  \eqref{spde0} with small noise,
which means that $\sigma$ is known and $\sigma \to 0$.
For parametric estimation for the linear reaction term 
of the second order parabolic SPDE with one space dimension 
driven by a cylindrical Brownian motion,
see Kaino and Uchida \cite{Kaino_Uchida2021a, Kaino_Uchida2021b}. 
For statistical inference for diffusion processes 
based on discrete observations, 
see Kessler \cite{Kessler1997},
Yoshida \cite{Yoshida2011}, 
Uchida and Yoshida \cite{Uchida_Yoshida2012},
S{\o}rensen and Uchida \cite{Sorensen_Uchida2003}
and Gloter and S{\o}rensen \cite{Gloter_Sorensen2009}.

Recently, Gaudlitz and Reiss \cite{Gaudlitz_Reiss2023} 
treated estimation for the reaction term of SPDEs with $d$-space dimensions
under small diffusivity
in the case that the driving process is a cylindrical Brownian motion and the reaction term is nonlinear. 
However, their model does not encompass SPDE \eqref{spde}.
Note that the $L^2$-estimate of the second order parabolic SPDE with two space dimensions 
whose driving process is a cylindrical Brownian motion 
is divergent,  see Walsh \cite{Walsh1986}.
For this reason, we study parametric estimation for the reaction term 
of SPDE \eqref{spde} driven by a $Q$-Wiener process $W_t^Q$
in line with the approach of \cite{Gaudlitz_Reiss2023}.
The main purpose of this paper is to show that the estimator of $\theta_0$ 
with the weight parameter $\beta$ based on high frequency spatio-temporal data
is bounded in probability or asymptotically normal with the rate 
$\mcl R_{\beta,\nu}^{-1}$ given in \eqref{rate-R} below under $\nu \to 0$. 
We also show that the function $\beta \mapsto \mcl R_{\beta,\nu}$ is decreasing.

This paper is organized as follows.
We state main results in Section \ref{sec2}.
We first propose an estimator of $\theta_0$ 
with the weight parameter $\beta$ using continuous spatio-temporal data.
We then derive an estimator of $\theta_0$ 
based on high frequency spatio-temporal data 
by discretizing the estimator based on continuous spatio-temporal data.
We discuss the optimal choice of the weight parameter $\beta$, 
and also address the estimation of the volatility parameter $\sigma^2$.
Section \ref{sec3} gives the simulation results of the estimator 
based on high frequency spatio-temporal data.
Section \ref{sec4} provides the proofs of our results.

\section{Main results}\label{sec2}
\subsection{Setting and notation}
Let $(\Omega, \FF, \{ \FF_t \}_{t \ge 0}, \PP)$
be a stochastic basis with usual conditions,
and let $\{ w_{l_1,l_2} \}_{l_1,l_2 \in \mbb N}$ be independent $\mbb R$-valued 
standard Brownian motions on this basis. 

Let $\{ \lambda_{l_1,l_2}, e_{l_1,l_2} \}_{l_1,l_2 \in \mbb N}$ 
be the eigenpairs of the differential operator $\mcl A$, 
where the eigenfunctions $e_{l_1,l_2}$ and 
the corresponding eigenvalues $\lambda_{l_1,l_2}$ are given by
\begin{equation*}
\lambda_{l_1,l_2} = \pi^2(l_1^2+l_2^2) + \frac{\kappa^2+\eta^2}{4},
\quad
e_{l_1,l_2}(y,z) = 2\sin(\pi l_1 y) \sin(\pi l_2 z)\ee^{-(\kappa y +\eta z)/2}
\end{equation*}
for $l_1,l_2 \in \mbb N$ and $(y,z) \in \overline D$.
Let $\mcl H = L^2(D)$ with the inner product
\begin{equation*}
\langle u,v \rangle = \langle u,v \rangle_{\mcl H} =
\iint_D u(x,y) v(x,y) \ee^{\kappa x + \eta y} \dd x \dd y,
\quad \| u \| = \| u \|_{\mcl H} = \sqrt{\langle u,u \rangle_\mcl H}.
\end{equation*}
Note that $\{ e_{l_1,l_2} \}_{l_1,l_2 \in \mbb N}$ is 
the complete orthonormal system of $\mcl H$ 
with the inner product $\langle \cdot , \cdot \rangle$.

Fix $\alpha >0$, and define a trace class linear operator $Q$ 
on a Hilbert space $\mcl U \supset \mcl H$ with 
a norm $\| \cdot \|_{\mcl U}$
and the complete orthonormal system $\{ v_{l_1,l_2} \}_{l_1,l_2 \in \mbb N}$, 
$v_{l_1,l_2} = \| e_{l_1,l_2} \|_{\mcl U}^{-1} e_{l_1,l_2}$ such that
$Q v_{l_1, l_2} = \mu_{l_1,l_2}^{-\alpha} \| e_{l_1,l_2} \|_{\mcl U}^2 v_{l_1,l_2}$,
where $\mu_{l_1,l_2} = \pi^2(l_1^2+l_2^2)$. 
We consider the $Q$-Wiener process given by 
\begin{equation*}
W_t^Q = \sum_{l_1,l_2 \ge1} \mu_{l_1,l_2}^{-\alpha/2} \| e_{l_1,l_2} \|_{\mcl U} 
w_{l_1,l_2}(t) v_{l_1,l_2}
= \sum_{l_1,l_2\ge1} \mu_{l_1,l_2}^{-\alpha/2} w_{l_1,l_2}(t) e_{l_1,l_2}.
\end{equation*}
In this paper, we assume that $\alpha >0$ is known. 
See \cite{Bossert2023arXiv} for the estimation of $\alpha$.
Let $\{ S_t \}_{t \ge 0}$ be the semigroup generated 
by the differential operator $- \mcl A$, 
that is, $S_{t} = \ee^{-t \mcl A}$, where
\begin{equation*}
\ee^{-t \mcl A} u = \sum_{l_1,l_2 \ge 1} \ee^{-\lambda_{l_1,l_2} t} 
\langle u, e_{l_1,l_2} \rangle e_{l_1,l_2}, 
\quad u \in \mcl H.
\end{equation*}
Notice that
\begin{equation*}
G_t(\bs x,\bs y) = \sum_{l_1,l_2 \ge 1} \ee^{-\lambda_{l_1,l_2} t} 
e_{l_1,l_2}(\bs x) e_{l_1,l_2}(\bs y), 
\quad \bs x, \bs y \in D
\end{equation*}
is the kernel of the semigroup $S_t$ and 
$S_t u(\bs x) = \langle G_t(\bs x,\cdot), u \rangle$, $u \in \mcl H$.
Let $\widetilde Q$ be a linear operator such that
$\widetilde Q e_{l_1,l_2} = \mu_{l_1,l_2}^{-\alpha} e_{l_1,l_2}$, 
and consider the fractional powers $\widetilde Q^\delta$ for $\delta \in \mbb R$:
\begin{equation}\label{frac-Q}
\widetilde Q^\delta u = \sum_{l_1,l_2 \ge 1} \mu_{l_1,l_2}^{-\alpha \delta} 
u_{l_1,l_2} e_{l_1,l_2}
\end{equation}
for $u = \sum_{l_1,l_2 \ge 1} u_{l_1,l_2} e_{l_1,l_2}$, $u_{l_1,l_2} \in \mbb R$
such that $\widetilde Q^\delta u \in \mcl H$.
Note that $\widetilde Q^\delta u \in \mcl H$ for $u \in \mcl H$ if $\delta \ge 0$.
Define $\mcl U_0 
= \{ \widetilde Q^{1/2} v | v \in \mcl H \}$ 
with the induced norm $\| \widetilde Q^{-1/2} u \|$, $u \in \mcl U_0$. 
For a separable Hilbert space $\mcl K$, 
$\mrm{HS}(\mcl K)$ denotes the space of all 
Hilbert-Schmidt operators from $\mcl K$ to $\mcl H$.
Since it holds from Proposition \ref{propA1} that 
$S_t \in \mrm{HS}(\mcl U_0)$ for $t >0$
and $\int_0^1 \| S_{\nu t} \|_{\mrm{HS}(\mcl U_0)}^2 \dd t < \infty$ 
for $\nu >0$,
there exists a unique mild solution $\{ X_t \}_{t \in \mbb T}$ 
of SPDE \eqref{spde} such that
the variation of constants formula
\begin{equation}\label{vcf}
X_t = S_{\nu t} X_0 + \theta_0 \int_0^t S_{\nu(t-s)} X_s \dd s
+\sigma \int_0^t S_{\nu(t-s)} \dd W_s^Q,
\quad t \in \mbb T
\end{equation}
on $L^2(\Omega;\mcl H)$, and it follows from \eqref{eq-lemA5-1} that 
$t \mapsto X_t$ is continuous $\PP$-a.s., 
see Theorem 7.5 in \cite{DaPrato_Zabczyk2014}.
$\PP_{\theta_0,\sigma}$ denotes the law of the solution 
$\{ X_t \}_{t \in \mbb T}$ of SPDE \eqref{spde} on $C(\mbb T;\mcl H)$.

For $\beta \in \mbb R$, let 
$\mcl L_{\beta}^2 
= \{ u = \sum_{l_1,l_2 \ge 1} u_{l_1,l_2} e_{l_1,l_2} | 
\| u \|_{\mcl L_{\beta}^2} < \infty \}$
with the weighted inner product 
\begin{equation*}
\langle u,v \rangle_{\mcl L_{\beta}^2} 
= \langle \widetilde Q^{\beta/2} u, \widetilde Q^{\beta/2} v \rangle,
\quad \| u \|_{\mcl L_{\beta}^2} 
= \sqrt{\langle u,u \rangle_{\mcl L_{\beta}^2}}.
\end{equation*}
Note that 
$\mcl L_{\beta_1}^2 \subset \mcl L_{\beta_2}^2$ for $\beta_1 \le \beta_2$ and 
$\mcl L_{\beta_{-}}^2 \subset \mcl H \subset \mcl L_{\beta_{+}}^2$ 
for $\beta_{-} \le 0 \le \beta_{+}$.

For families $\{ a_\lambda \}_{\lambda}, \{ b_\lambda \}_{\lambda} \subset \mbb R$,
we write $a_\lambda \lesssim b_\lambda$
if $|a_\lambda| \leq C |b_\lambda|$
for some universal constant $C >0$ and any $\lambda$,
and we write $a_\lambda \sim b_\lambda$ 
if $a_\lambda \lesssim b_\lambda$ and $b_\lambda \lesssim a_\lambda$.
For a family $\{ a_\lambda \}_{\lambda} \subset \mbb R$ and $a \in \mbb R$, 
we write $a_\lambda \equiv a$ if $a_\lambda = a$ for all $\lambda$.

\subsection{Estimation of $\theta_0$ based on continuous observations}
Gaudlitz and Reiss \cite{Gaudlitz_Reiss2023} 
proposed the minimum likelihood estimator of the reaction term 
under small diffusive level.
Since $\sup_{t \in \mbb T} \EE[\| X_t \|_{\mcl L_{-1}^2}^2] \gtrsim 
\sum_{l_1,l_2 \ge 1} (l_1^2 +l_2^2)^{-1} = \infty$ 
and the Novikov condition (2.3) 
in Moleriu \cite{Moleriu2009} for 
$f(t) = \widetilde Q^{-1} X_t$ is not satisfied in our model,
we cannot construct a maximum likelihood estimator of $\theta_0$ using 
the Girsanov theorem (Theorem 2.3 in \cite{Moleriu2009}). 
Therefore, we propose the minimum contrast function based on 
the likelihood of \cite{Gaudlitz_Reiss2023} and provide the minimum contrast estimator
of $\theta_0$. 
Let $\mbb X^{(\mrm{cont})} 
= \{ X_{t}(y,z) \}_{t \in \mbb T, (y,z) \in \overline D}$
be continuous observations. 
Because it holds from Lemma \ref{lemB3} below that 
$\sup_{t \in \mbb T}\EE[ \| X_t \|_{\mcl L_{\beta}^2}^2 ] < \infty$ 
for $\beta > -1$, we set the contrast function 
\begin{equation}\label{contrast}
U_\beta (\theta_0 : \mbb X^{(\mrm{cont})}) = 
\theta_0 \int_0^1 \langle X_t, \dd X_t +\nu \mcl A X_t \dd t \rangle_{\mcl L_{\beta}^2}
+ \frac{\theta_0^2}{2} \int_0^1 \| X_t \|_{\mcl L_{\beta}^2}^2 \dd t,
\end{equation}
where $\beta > -1$ can be chosen arbitrarily.
Define the estimator $\hat \theta_{0,\beta}^{(\mrm{cont})}$
as $\theta_0$ which minimizes \eqref{contrast}. 
That is, 
\begin{equation}\label{est_cont}
\hat \theta_{0,\beta}^{(\mrm{cont})} = 
\hat \theta_{0,\beta,\nu}^{(\mrm{cont})} =
\frac{\displaystyle 
\int_0^1 \langle X_t, \dd X_t + \nu \mcl A X_t \dd t \rangle_{\mcl L_{\beta}^2} 
}{\displaystyle \int_0^1 \| X_t \|_{\mcl L_{\beta}^2}^2 \dd t}.
\end{equation}
The numerator of \eqref{est_cont} 
is decomposed as follows:
\begin{align*}
&\int_0^1 \langle X_t, \dd X_t + \nu \mcl A X_t \dd t \rangle_{\mcl L_{\beta}^2}
\\
&= \sum_{l_1,l_2 \ge 1}
\biggl\{
\int_0^1 
\mu_{l_1,l_2}^{-\alpha \beta}
\langle X_t, e_{l_1,l_2} \rangle \dd \langle X_t, e_{l_1,l_2} \rangle
+ \nu \int_0^1 \lambda_{l_1,l_2} \mu_{l_1,l_2}^{-\alpha \beta} 
\langle X_t, e_{l_1,l_2} \rangle^2 \dd t
\biggr\}.
\end{align*}
\begin{rmk}\label{rmk1}
The estimator of $\theta_0$ proposed by Gaudlitz and Reiss \cite{Gaudlitz_Reiss2023} 
in the SPDE driven by a cylindrical Brownian motion 
corresponds to our proposed estimator \eqref{est_cont} for $\beta = 0$
in the SPDE driven by the $Q$-Wiener process. 
If we choose $\beta$ independent of the value of the damping parameter $\alpha$ and 
spatio-temporal observation data,
it may not be possible to construct a consistent estimator of $\theta_0$
(see Theorem \ref{th1}, \ref{th2} or Corollary \ref{cor1} below). 
In fact, we observe that in our setting of Section \ref{sec3}, 
the consistent estimator of $\theta_0$ based on discrete observations
cannot be constructed unless $\beta > 0.2$.
In order to get a consistent estimator of $\theta_0$, 
we choose a suitable $\beta$ depending on the value of $\alpha$ and spatio-temporal observation data
and propose the estimator with the parameter $\beta$.
\end{rmk}
For $\beta \ge -1$, 
we make the following initial condition.
\begin{description}
\item[\textbf{[A1]$_{\beta,2}$}]
$\EE[\| X_0 \|_{\mcl L_{\beta}^2}^2] < \infty$. 
\end{description}
For $\beta >-1$, define
\begin{equation}\label{rate-R}
\mcl R_{\beta,\nu} =
\frac{\displaystyle
\int_0^1 \EE \Bigl[ \| X_t \|_{\mcl L_{\beta}^2}^2 \Bigr] \dd t}
{\displaystyle 
\sqrt{\int_0^1 \EE \Bigl[ \| X_t \|_{\mcl L_{2\beta +1}^2}^2 \Bigr] \dd t}}.
\end{equation}
We then obtain the following theorem.
\begin{thm}\label{th1}
Let $\alpha >0$ and $\beta > -1$. 
Assume that [A1]$_{\beta,2}$ holds. 
\begin{enumerate}
\item[(1)]
It follows that as $\nu \to 0$, 
\begin{equation}\label{th1-rate-R}
\mcl R_{\beta,\nu} \sim
\begin{cases}
\nu^{-1/2}, & -1 < \beta < \frac{1}{2\alpha} -1,
\\
(-\nu \log \nu)^{-1/2}, & \beta = \frac{1}{2\alpha} -1,
\\
\nu^{\alpha(\beta +1) -1}, & \frac{1}{2\alpha} -1 < \beta < \frac{1}{\alpha} -1,
\\
-\log \nu, & \beta = \frac{1}{\alpha} -1,
\\
1, & \frac{1}{\alpha} -1 < \beta.
\end{cases}
\end{equation}

\item[(2)]
If $-1 < \beta \le \frac{1}{2\alpha} -1$, then as $\nu \to 0$,
\begin{equation*}
\mcl R_{\beta,\nu} (\hat \theta_{0,\beta}^{(\mrm{cont})} -\theta_0^*) 
\dto N(0,(\sigma^*)^2).
\end{equation*}

\item[(3)]
If $\frac{1}{2\alpha} -1 < \beta \le \frac{1}{\alpha} -1$, then as $\nu \to 0$,
\begin{equation*}
\mcl R_{\beta,\nu} (\hat \theta_{0,\beta}^{(\mrm{cont})} -\theta_0^*) = \OO_\PP(1).
\end{equation*}
\end{enumerate}
\end{thm}
\begin{rmk}
When one estimates $\theta_0$ based on continuous observations 
under a regular initial condition, 
it is sufficient to use the result of Theorem \ref{th1}-(2)
since one can choose $\beta < \frac{1}{2\alpha} -1$. 
As mentioned in Remark \ref{rmk1}, when one constructs the estimator of $\theta_0$ 
based on discrete observations, one needs the result of Theorem \ref{th1}-(3) 
since one chooses $\beta$ depending on $\alpha$ and 
the number of spatio-temporal observations, see 
Theorem \ref{th2}-(3), Corollary \ref{cor1} or Section \ref{sec3}.
\end{rmk}
The properties of $\mcl R_{\beta,\nu}$ 
and the optimal choice of the weight parameter $\beta$ will be addressed 
in Subsection \ref{sec2.4}.

\subsection{Estimation of $\theta_0$ based on discrete observations}
Our main aim is to estimate $\theta_0$ based on 
high frequency spatio-temporal data.
To this end, we use discrete data and construct an estimator of $\theta_0$ 
by discretizing the estimator 
given in \eqref{est_cont}.

Let $\mbb X_{N,M} 
= \{ X_{t_i}(y_j,z_k)\}_{0 \le i \le N,, 0 \le j \le M_1, 0 \le k \le M_2}$
be discrete observations, where $t_i = i/N$, $y_j = j/M_1$ and $z_k = k/M_2$. 
We define the estimator of $\theta_0$ based on $\mbb X_{N,M}$ by
\begin{align}
\hat \theta_{0,\beta} =
\hat \theta_{0,\beta,\nu,N,M,L} & = 
\sum_{i=1}^N \sum_{l_1=1}^L \sum_{l_2=1}^L \mu_{l_1, l_2}^{-\alpha \beta}
[ X_{t_{i-1}} ]_{M,l_1,l_2}
\Bigl(
[ X_{t_i} ]_{M,l_1,l_2} - \ee^{-\nu \lambda_{l_1,l_2}/N} [ X_{t_{i-1}} ]_{M,l_1,l_2}
\Bigr)
\nonumber
\\
&\qquad \times
\Biggl\{\frac{1}{N}
\sum_{i=1}^N \sum_{l_1=1}^L \sum_{l_2=1}^L \mu_{l_1, l_2}^{-\alpha \beta}
[ X_{t_{i-1}} ]_{M,l_1,l_2}^2
\Biggr\}^{-1},
\label{est}
\end{align}
where
\begin{align*}
h_l(x:a) &= 
\frac{\sqrt{2}\ee^{ax/2}}{(a/2)^2+(\pi l)^2}
\biggl( \frac{a}{2} \sin(\pi l x) - \pi l \cos(\pi l x) \biggr),
\quad a, x \in \mbb R, \ l \in \mbb N,
\\
\delta_j^{[y]} h_l(a) &= h_l(y_{j}:a) - h_l(y_{j-1}:a),
\quad
\delta_k^{[z]} h_l(a) = h_l(z_{k}:a) - h_l(z_{k-1}:a), 
\\
[ X_t ]_{M,l_1,l_2} &=
\sum_{j=1}^{M_1} \sum_{k=1}^{M_2}
X_t(y_{j-1},z_{k-1}) 
\delta_j^{[y]} h_{l_1}(\kappa) \delta_k^{[z]} h_{l_2}(\eta),
\quad t \in \mbb T.
\end{align*}
For simplicity, we assume that 
there exist $b_M > 1/2$ and $b_L > (\alpha -1)/4$ such that
$M_1 \land M_2 = N^{b_M}$, $L = N^{b_L}$.
The derivation of this estimator can be found in Subsection \ref{sec4.2}.

For $\beta \in \mbb R$ and $q >2$, define
$\mcl L_{\beta}^q 
= \{ u \in \mcl L_{\beta}^2 | \| u \|_{\mcl L_{\beta}^q} < \infty \}$ with
\begin{equation*}
\| u \|_{\mcl L_{\beta}^q}
= \biggl( \iint_D |\widetilde Q^{\beta/2} u(x,y)|^q \ee^{\kappa x + \eta y} \dd x \dd y
\biggr)^{1/q}.
\end{equation*}
For $\beta \in \mbb R$ and $q \ge 2$, let $\mcl L^q = \mcl L_{0}^q$ and
$\mcl L_{\beta}^{\infty} = \bigcap_{q \ge 2} \mcl L_{\beta}^q$.
For $\alpha >0$, $x \in (0,1]$ and $p > 4/x$, we set
\begin{equation*}
\rho_{x}^{(\mrm{time})} = 
\begin{cases}
2x, & x \in (0,1/2],
\\
\frac{1}{2(1 -x)}, & x \in (1/2,1),
\\
\infty, & x = 1,
\end{cases}
\end{equation*}
\begin{equation*}
\tau_{1,x,p}^{(\mrm{space})} = 1 -\frac{2}{p x -2},
\quad
\tau_{2,x,p}^{(\mrm{space})} = 
\begin{cases}
1 -\frac{1}{p x -1}, & x \in (0,1/2],
\\
1 -\frac{2}{p}, & x \in (1/2,1],
\end{cases}
\end{equation*}
\begin{equation*}
\rho_{1,x,p}^{(\mrm{space})} = 1 -\frac{p +4}{2p x +p -4},
\quad
\rho_{2,x,p}^{(\mrm{space})} = 
\begin{cases}
1 -\frac{p +2}{2p x +p -2}, & x \in (0,1/2],
\\
\frac{1}{2} -\frac{1}{p}, & x \in (1/2,1],
\end{cases}
\end{equation*}
\begin{equation*}
\tau_{x,p}^{(\mrm{trunc})} = 2\biggl(1 -\frac{2}{p x -2} \biggr),
\quad
\rho_{x,p}^{(\mrm{trunc})} = 2\biggl(1 -\frac{p +2}{2p x +p -4}\biggr),
\end{equation*}
\begin{equation*}
\phi_{1,x,p}^{(\mrm{space})} = \frac{p -1}{p x -2}, 
\quad
\phi_{2,x,p}^{(\mrm{space})} = 
\begin{cases}
\frac{p}{p x -1}, & x \in (0,1/2],
\\
3 -2x +\frac{2}{p}, & x \in (1/2,1],
\end{cases}
\end{equation*}
\begin{equation*}
\psi_{1,x,p}^{(\mrm{space})} = \frac{2(p -1)}{2p x +p -4}, 
\quad
\psi_{2,x,p}^{(\mrm{space})} = 
\begin{cases}
\frac{2p}{2p x +p -2}, & x \in (0,1/2],
\\
\frac{3}{2} -x +\frac{1}{p}, & x \in (1/2,1],
\end{cases}
\end{equation*}
\begin{equation*}
\phi_{x,\alpha,p}^{(\mrm{trunc})}
= \frac{(\alpha-1)p x +2(1 -2\alpha +p)}{2(p x -2)},
\quad
\psi_{x,\alpha,p}^{(\mrm{trunc})} 
= \frac{(\alpha-1)p x +2(1 -2\alpha +p)}{2p x +p -4},
\end{equation*}
and introduce the following conditions.

\begin{description}
\item[\textbf{[A1]$_{\beta,p}$}]
$\EE[\| X_0 \|_{\mcl L_{\beta}^p}^2] < \infty$.

\item[\textbf{[B1]$_{\beta,p}$}]
$\nu (M_1 \land M_2)^{2r_1} N^{-\phi_{1,\alpha(\beta +1),p}^{(\mrm{space})}} 
\to \infty$ and
$\nu (M_1 \land M_2)^{2r_2} N^{-\phi_{2,\alpha(\beta +1),p}^{(\mrm{space})}} 
\to \infty$ 
as $\nu \to 0$, $N \to \infty$ and $M_1 \land M_2 \to \infty$ 
for some $r_1 < \tau_{1,\alpha(\beta+1),p}^{(\mrm{space})}$ and
$r_2 < \tau_{2,\alpha(\beta+1),p}^{(\mrm{space})}$.

\item[\textbf{[B2]$_{\beta,p}$}]
$\nu L^s N^{-\phi_{\alpha(\beta +1),\alpha,p}^{(\mrm{trunc})}} \to \infty$ 
as $\nu \to 0$, $N \to \infty$ and $L \to \infty$ 
for some $s < \tau_{\alpha(\beta +1),p}^{(\mrm{trunc})}$.

\item[\textbf{[C1]$_{\beta}$}]
$\nu N^q \to \infty$ as $\nu \to 0$ and $N \to \infty$ 
for some $q < \rho_{\alpha(\beta +1)}^{(\mrm{time})}$.

\item[\textbf{[C2]$_{\beta,p}$}]
$\nu (M_1 \land M_2)^{2r_1} N^{-\psi_{1,\alpha(\beta +1),p}^{(\mrm{space})}} 
\to \infty$ and
$\nu (M_1 \land M_2)^{2r_2} N^{-\psi_{2,\alpha(\beta +1),p}^{(\mrm{space})}} 
\to \infty$ 
as $\nu \to 0$, $N \to \infty$ and $M_1 \land M_2 \to \infty$ 
for some $r_1 < \rho_{1,\alpha(\beta+1),p}^{(\mrm{space})}$ and
$r_2 < \rho_{2,\alpha(\beta+1),p}^{(\mrm{space})}$.

\item[\textbf{[C3]$_{\beta,p}$}]
$\nu L^s N^{-\psi_{\alpha(\beta +1),\alpha,p}^{(\mrm{trunc})}} \to \infty$ 
as $\nu \to 0$, $N \to \infty$ and $L \to \infty$ 
for some $s < \rho_{\alpha(\beta +1),p}^{(\mrm{trunc})}$.
\end{description}

\begin{rmk}\label{rmk3}
\begin{enumerate}
\item[(1)]
For $\beta \ge -1$ and $p > 2$, 
[A1]$_{-1,p}$ is a stricter condition than [A1]$_{\beta,2}$.

\item[(2)]
Since $\| u \|_{\mcl L_{-1}^p} \sim \| \mcl A^{\alpha/2} u \|_{\mcl L^p}$ 
and $\mcl A$ is the second order elliptic operator,
for example, if $X_0$ is non-random and belongs to $C^{\lceil \alpha \rceil}(D)$, 
then [A1]$_{-1,p}$ and especially [A1]$_{-1,\infty}$ are satisfied.

\item[(3)]
For $\beta > -1$ and $p > 4$, 
[C2]$_{\beta,p}$ and [C3]$_{\beta,p}$ are stricter conditions 
than [B1]$_{\beta,p}$ and [B2]$_{\beta,p}$, respectively.
\end{enumerate}
\end{rmk}

Propositions \ref{propA1}-\ref{propB1} below show that 
the estimators $\hat \theta_{0,\beta}^{(\mrm{cont})}$ and $\hat \theta_{0,\beta}$
are asymptotically equivalent as $\nu \to 0$,
$N \to \infty$, $M_1 \land M_2 \to \infty$ and $L \to \infty$.
We obtain the following main theorem.

\begin{thm}\label{th2}
Let $\alpha >0$, $-1 < \beta \le \frac{1}{\alpha} -1$ 
and $p > \frac{4}{\alpha(\beta +1)}$. Assume that [A1]$_{-1,p}$ holds.
\begin{enumerate}
\item[(1)]

Under [B1]$_{\beta,p}$ and [B2]$_{\beta,p}$, it holds that 
as $\nu \to 0$, $N \to \infty$, $M_1 \land M_2 \to \infty$ and $L \to \infty$,
\begin{equation*}
\hat \theta_{0,\beta} \pto \theta_0^*.
\end{equation*}

\item[(2)]
If $-1 < \beta \le \frac{1}{2\alpha}-1$, then it holds that 
under [C1]$_\beta$, [C2]$_{\beta,p}$ and [C3]$_{\beta,p}$, 
\begin{equation}\label{th2_eq1}
\mcl R_{\beta,\nu} (\hat \theta_{0,\beta} - \theta_0^*) \dto N(0,(\sigma^*)^2)
\end{equation}
as $\nu \to 0$, $N \to \infty$, $M_1 \land M_2 \to \infty$ and $L \to \infty$.

\item[(3)]
If $\frac{1}{2\alpha}-1 < \beta \le \frac{1}{\alpha}-1$, then
it holds that under [C1]$_\beta$, [C2]$_{\beta,p}$ and [C3]$_{\beta,p}$, 
\begin{equation}\label{th2_eq2}
\mcl R_{\beta,\nu} (\hat \theta_{0,\beta} - \theta_0^*) = \OO_\PP(1)
\end{equation}
as $\nu \to 0$, $N \to \infty$, $M_1 \land M_2 \to \infty$ and $L \to \infty$.
\end{enumerate}
\end{thm}

See Subsection \ref{sec2.4} for sufficient conditions for consistency 
and the asymptotic properties \eqref{th2_eq1}, \eqref{th2_eq2} 
of the estimator $\hat \theta_{0,\beta}$ 
under the given discrete data $\mbb X_{N,M}$.

Finally, we give the sufficient conditions when $p=\infty$. 
Notice that
\begin{equation*}
\rho_{2,x,p}^{(\mrm{space})} \le \rho_{1,x,p}^{(\mrm{space})},
\quad
\phi_{1,x,p}^{(\mrm{space})} \le \phi_{2,x,p}^{(\mrm{space})},
\quad
\psi_{2,x,p}^{(\mrm{space})} \le \psi_{1,x,p}^{(\mrm{space})}
\end{equation*}
for sufficiently large $p$.
For $\alpha >0$ and $x \in (0,1]$, define
\begin{equation*}
\rho_{x}^{(\mrm{space})} = 1 -\frac{1}{2 x +1},
\quad
\rho_{x}^{(\mrm{trunc})} = 2\biggl(1 -\frac{1}{2 x +1}\biggr),
\end{equation*}
\begin{equation*}
\phi_{x}^{(\mrm{space})} = \frac{1}{x} \lor (3 -2x),
\quad
\psi_{1,x}^{(\mrm{space})} = \frac{2}{2 x +1},
\quad
\psi_{2,x}^{(\mrm{space})} = \frac{3}{2} -x, 
\end{equation*}
\begin{equation*}
\phi_{x,\alpha}^{(\mrm{trunc})} = \frac{(\alpha-1)x +2}{2x},
\quad
\psi_{x,\alpha}^{(\mrm{trunc})} = \frac{(\alpha-1) x +2}{2x +1}.
\end{equation*}
If $p = \infty$, 
the conditions [B1]$_{\beta,p}$, [B2]$_{\beta,p}$, [C2]$_{\beta,p}$ and 
[C3]$_{\beta,p}$ are respectively replaced by
\begin{description}
\item[\textbf{[B1]$_{\beta, \infty}$}]
$\nu (M_1 \land M_2)^{2r} N^{-\phi_{\alpha(\beta +1)}^{(\mrm{space})}} 
\to \infty$ 
as $\nu \to 0$, $N \to \infty$ and $M_1 \land M_2 \to \infty$ for some $r < 1$,

\item[\textbf{[B2]$_{\beta,\infty}$}]
$\nu L^s N^{-\phi_{\alpha(\beta +1),\alpha}^{(\mrm{trunc})}} \to \infty$ 
as $\nu \to 0$, $N \to \infty$ and $L \to \infty$ for some $s < 2$,

\item[\textbf{[C2]$_{\beta,\infty}$}]
$\nu (M_1 \land M_2)^{2r_1} N^{-\psi_{1,\alpha(\beta +1)}^{(\mrm{space})}} \to \infty$ 
and $\nu (M_1 \land M_2)^{2r_2} N^{-\psi_{2,\alpha(\beta +1)}^{(\mrm{space})}} 
\to \infty$ as $\nu \to 0$, $N \to \infty$ and $M_1 \land M_2 \to \infty$ 
for some $r_1 < \rho_{\alpha(\beta+1)}^{(\mrm{space})}$ and $r_2 < 1/2$,

\item[\textbf{[C3]$_{\beta,\infty}$}]
$\nu L^s N^{-\psi_{\alpha(\beta +1),\alpha}^{(\mrm{trunc})}} \to \infty$ 
as $\nu \to 0$, $N \to \infty$ and $L \to \infty$ 
for some $s < \rho_{\alpha(\beta +1)}^{(\mrm{trunc})}$.
\end{description}
We now obtain the sufficient conditions required 
for consistency or the asymptotic properties \eqref{th2_eq1}, \eqref{th2_eq2}
of the estimator $\hat \theta_{0,\beta}$ when [A1]$_{-1,\infty}$ is satisfied.

\subsection{Optimal choice of $\beta$}\label{sec2.4}
In this subsection, we discuss the optimal choice of $\beta$ for $\alpha$.
More precisely, we fix $\alpha >0$ and investigate the superiority of 
the estimators $\hat \theta_{0,\beta_1}$ and $\hat \theta_{0,\beta_2}$ 
for $-1 <\beta_1 < \beta_2 \le \frac{1}{\alpha} -1$.

For fixed $\alpha >0$, let $\frac{1}{2\alpha} -1 < \beta' < \frac{1}{\alpha} -1$. 
Because of \eqref{th1-rate-R} and the relation 
\begin{equation*}
1 \lesssim -\log \nu \lesssim \nu^{\alpha(\beta' +1)-1}
\lesssim (-\nu \log \nu)^{-1/2} \lesssim \nu^{-1/2},
\quad \nu \to 0,
\end{equation*}
the convergence rate $\mcl R_{\beta,\nu}^{-1}$ 
of the estimator $\hat \theta_{0,\beta}^{(\mrm{cont})}$
is faster as $\beta$ gets smaller.
However, the convergence rates of the estimators 
$\hat \theta_{0,\beta_1}^{(\mrm{cont})}$ 
and $\hat \theta_{0,\beta_2}^{(\mrm{cont})}$
are both $\nu^{1/2}$ for $-1 < \beta_1 < \beta_2 < \frac{1}{2\alpha} -1$, 
and we cannot determine the superiority of 
the estimators by only using \eqref{th1-rate-R}.
For this reason, we need to examine the properties 
of $\mcl R_{\beta,\nu}$ in more detail. 

The following proposition shows the superiority of 
the estimators $\hat \theta_{0,\beta_1}^{(\mrm{cont})}$ and 
$\hat \theta_{0,\beta_2}^{(\mrm{cont})}$ 
for $-1 < \beta_1 < \beta_2 < \frac{1}{2\alpha} -1$.
\begin{prop}\label{prop1}
Assume that [A1]$_{-1,2}$ holds. 
Then, for $\alpha >0$ and  $\nu \in (0,1)$, 
\begin{equation*}
\mcl R_{\beta,\nu} = 
\frac{\displaystyle\sum_{l_1,l_2 \ge 1} f_{l_1,l_2}(2(\nu \lambda_{l_1,l_2} -\theta_0): 
\alpha, \sigma^2, X_0) \mu_{l_1,l_2}^{-\alpha(\beta +1)}}
{\sqrt{\displaystyle\sum_{l_1,l_2 \ge 1} f_{l_1,l_2}(2(\nu \lambda_{l_1,l_2} -\theta_0): 
\alpha, \sigma^2, X_0) \mu_{l_1,l_2}^{-2\alpha(\beta +1)}}},
\end{equation*}
where $f_{l_1,l_2}(x:\alpha, \sigma^2, X_0) = \mu_{l_1,l_2}^\alpha 
\EE[\langle X_0, e_{l_1,l_2} \rangle^2] f_1(x)+ \sigma^2 f_2(x)$,
\begin{equation*}
f_1(x) = 
\begin{cases}
\frac{1 -\ee^{-x}}{x}, & x \neq 0,
\\
1, & x = 0, 
\end{cases}
\quad
f_2(x) = 
\begin{cases}
\frac{x -1 +\ee^{-x}}{x^2}, & x \neq 0,
\\
\frac{1}{2}, & x = 0
\end{cases}
\end{equation*}
and the function $\beta \mapsto \mcl R_{\beta,\nu}$ is decreasing 
on $(-1,\frac{1}{\alpha} -1]$. 
\end{prop}

Proposition \ref{prop1} shows that for $-1 <\beta_1 < \beta_2 < \frac{1}{2\alpha} -1$, 
$\hat \theta_{0,\beta_1}^{(\mrm{cont})}$ is superior to 
$\hat \theta_{0,\beta_2}^{(\mrm{cont})}$ in the sense that 
the asymptotic variance is smaller.
Indeed, by setting
$\mcl I_\beta = \lim_{\nu \to 0} (\sigma^*)^{-2} \nu \mcl R_{\beta,\nu}^2$
for $-1 < \beta < \frac{1}{2\alpha} -1$,
we see from Theorem \ref{th1} and Proposition \ref{prop1} that
$\nu^{-1/2}(\hat \theta_{0,\beta}^{(\mrm{cont})} -\theta_0^*)
\dto N(0,\mcl I_{\beta}^{-1})$
and $\mcl I_{\beta_1}^{-1} < \mcl I_{\beta_2}^{-1}$.

Theorem \ref{th2} indicates that the choice of $\beta$ in
the estimator $\hat \theta_{0,\beta}$ depends on 
the high frequency spatio-temporal data $\mbb X_{N,M}$.
We thus consider how to choose $\beta$ from the given data $\mbb X_{N,M}$.

Because $x \mapsto \rho_{x}^{(\mrm{time})}$, 
$x \mapsto \tau_{j,x,p}^{(\mrm{space})}$
and $x \mapsto \rho_{j,x,p}^{(\mrm{space})}$ are increasing and
$x \mapsto \phi_{j,x,p}^{(\mrm{space})}$
and $x \mapsto \psi_{j,x,p}^{(\mrm{space})}$ are decreasing for $p \ge 2$ and $j=1,2$,
we see that the conditions required for consistency or 
the asymptotic properties \eqref{th2_eq1}, \eqref{th2_eq2}
of the estimator $\hat \theta_{0,\beta}$ get more relaxed as $\beta$ increases
under a fixed $\alpha$. Therefore, we show the lower bound of $\beta$ 
for the given data $\mbb X_{N,M}$.
For simplicity, let $p = \infty$.
For $\alpha,{\rm n}, {\rm m} >0$, define
\begin{equation*}
\beta_{\alpha,{\rm n},{\rm m}}^{(\mrm{cons})} =
\begin{cases}
\frac{{\rm n}}{\alpha({\rm m} -1)} -1, & {\rm m} \ge 2 {\rm n} +1,
\\
\frac{1}{2\alpha}(3 -\frac{{\rm m}-1}{{\rm n}}) -1, & {\rm m} < 2 {\rm n} +1,
\end{cases}
\end{equation*}
and $\beta_{\alpha,{\rm n},{\rm m}}^{(\mrm{asym})} 
= \beta_{\alpha,{\rm n}}^{(\mrm{as}\text{-}\mrm{t})} 
\lor \beta_{\alpha,{\rm n},{\rm m}}^{(\mrm{as}\text{-}\mrm{sp})}$, 
where
\begin{equation*}
\beta_{\alpha,{\rm n}}^{(\mrm{as}\text{-}\mrm{t})} =
\begin{cases}
\frac{1}{2 \alpha {\rm n}} -1, & {\rm n} \ge 1,
\\
\frac{1}{\alpha}(1 -\frac{{\rm n}}{2}) -1, & {\rm n} < 1,
\end{cases}
\quad
\beta_{\alpha,{\rm n},{\rm m}}^{(\mrm{as}\text{-}\mrm{sp})} =
\frac{1}{2\alpha}
\biggl( \frac{1+ 2{\rm n}}{{\rm m} -1} 
\lor \frac{3{\rm n} -{\rm m} +2}{{\rm n}} \biggr) -1, \quad {\rm m} >1.
\end{equation*}

We then have the following corollary of Theorem \ref{th2} with $p = \infty$.
Note that as mentioned in Remark \ref{rmk3}-(2), 
the condition [A1]$_{-1,\infty}$ is satisfied 
if the initial value $X_0$ is deterministic and smooth.
\begin{cor}\label{cor1}
Let $\alpha >0$ and $-1 < \beta \le \frac{1}{\alpha} -1$. 
Assume that [A1]$_{-1,\infty}$ holds.
For ${\rm m} > {\rm n} >0$ and $\ell > \frac{{\rm n}(\alpha -1)}{4} \lor 0$, let
$N = \nu^{-\rm{n}}$, $(M_1 \land M_2)^2 = \nu^{-\rm{m}}$ and $L = \nu^{-\ell}$. 
\begin{itemize}
\item[(1)]
The estimator $\hat \theta_{0,\beta}$ is consistent as $\nu \to 0$
if $\ell > \frac{1 +{\rm n}\phi_{\alpha(\beta +1),\alpha}^{(\mrm{trunc})}}{2}$
and $\beta > \beta_{\alpha,{\rm n},{\rm m}}^{(\mrm{cons})}$ 
for ${\rm m} > {\rm n}+1$.

\item[(2)]
The estimator $\hat \theta_{0,\beta}$ has the asymptotic property 
\eqref{th2_eq1} or \eqref{th2_eq2} as $\nu \to 0$ if 
$\ell > \frac{1 +{\rm n}\psi_{\alpha(\beta +1),\alpha}^{(\mrm{trunc})}}
{\rho_{\alpha(\beta +1)}^{(\mrm{trunc})}}$
and $\beta > \beta_{\alpha,{\rm n},{\rm m}}^{(\mrm{asym})}$ 
for ${\rm m} > {\rm n}+2$.
\end{itemize}
\end{cor}
\begin{rmk}
\begin{enumerate}
\item[(1)]
We find from Corollary \ref{cor1} that
the lower bounds of $\beta$ required for consistency and 
the asymptotic properties \eqref{th2_eq1}, \eqref{th2_eq2}
are given by $\beta_{\alpha,{\rm n},{\rm m}}^{(\mrm{cons})}$ and
$\beta_{\alpha,{\rm n},{\rm m}}^{(\mrm{asym})}$, respectively, 
and from Proposition \ref{prop1} that 
$\beta$ close to the lower bound gives a better estimator 
if $\nu$ is sufficiently small.

\item[(2)]
Corollary \ref{cor1} shows that 
we need to increase the number of spatial observations 
relatively more than the number of temporal observations  
in order to obtain the estimator $\hat \theta_{0,\beta}$ with 
consistency or the asymptotic properties \eqref{th2_eq1}, \eqref{th2_eq2} 
as well as Tonaki et al.\,\cite{TKU2024} does.

\item[(3)]
The lower bounds $\beta_{\alpha,{\rm n},{\rm m}}^{(\mrm{cons})}$ and
$\beta_{\alpha,{\rm n},{\rm m}}^{(\mrm{asym})}$
are both decreasing for ${\rm m}$.
This fact means that the lower bounds get small 
as the number of spatial observations increases, 
and we can obtain a better estimator.

\end{enumerate}
\end{rmk}

\subsection{Estimation of $\sigma^2$}\label{sec2.5}
In this subsection, we give an estimator of the volatility parameter $\sigma^2$ 
under $\nu \to 0$ based on high frequency spatio-temporal data.
We consider parametric estimation of $\sigma^2$ in line with the approach of 
Tonaki et al.\,\cite{TKU2023arXiv}. 

For a fixed $b \in (0,1/2)$, let
$\mbb X_{N,m}^{(\mrm{space})} = \{ X_{t_i}(\widetilde y_j, \widetilde z_k) \}_{
0 \le i \le N, 0 \le j \le m_1, 0 \le k \le m_2}$ 
be spatially thinned data of $\mbb X_{N,M}$ such that
$\widetilde y_0, \widetilde z_0 \ge b$, 
$\widetilde y_{m_1}, \widetilde z_{m_2} \le 1-b$,
$\widetilde y_j = \widetilde y_0 +j \delta$ and 
$\widetilde z_k = \widetilde z_0 +k \delta$,
where $\delta = \frac{\widetilde y_{m_1} -\widetilde y_0}{m_1} 
= \frac{\widetilde z_{m_2} -\widetilde z_0}{m_2}$, 
$m_1 m_2 = \OO(\nu^{-1} N)$ and $N = \OO(\nu m_1 m_2)$.
For the thinned data $\mbb X_{N,m}^{(\mrm{space})}$, define the triple increment
\begin{align*}
T_{i,j,k} X &= 
X_{t_i}(\widetilde y_j, \widetilde z_k) 
-X_{t_i}(\widetilde y_{j-1}, \widetilde z_k)
-X_{t_i}(\widetilde y_j, \widetilde z_{k-1})
+X_{t_i}(\widetilde y_{j-1}, \widetilde z_{k-1})
\\
&\qquad -X_{t_{i-1}}(\widetilde y_j, \widetilde z_k) 
+X_{t_{i-1}}(\widetilde y_{j-1}, \widetilde z_k)
+X_{t_{i-1}}(\widetilde y_j, \widetilde z_{k-1})
-X_{t_{i-1}}(\widetilde y_{j-1}, \widetilde z_{k-1}).
\end{align*}
We denote the Bessel function of the first kind of order $0$ by
\begin{equation*}
J_0(x) = 1 +\sum_{k \ge 1} \frac{(-1)^k}{(k!)^2} \Bigl(\frac{x}{2}\Bigr)^{2k}.
\end{equation*}
We assume $\theta_0 \le 0$ so that 
the operator $\mcl A_{\theta_0,\nu} = \nu \mcl A -\theta_0$ is 
positive definite and self-adjoint.
For simplicity, we assume the following initial condition as in \cite{TKU2023arXiv}.
\begin{description}
\item[\textbf{[A2]}]
The initial value $X_0$ is non-random and 
$\| \mcl A^{(1 +\alpha)/2} X_0 \| <\infty$.
\end{description}
Note that we can consider the general initial condition as in \cite{TKU2023a}.

Let $\Delta = 1/N$.
For $\alpha \in (0,2)$ and $\delta/\sqrt{\nu \Delta} \equiv r \in(0,\infty)$, define
\begin{equation*}
\psi_{r,\alpha} 
= \frac{2}{\pi} \int_0^\infty \frac{1 -\ee^{-x^2}}{x^{1 +2\alpha}}
\bigl( J_0(\sqrt{2} r x) -2J_0(r x) +1 \bigr) \dd x.
\end{equation*}

In the same way as Proposition 2.1 in \cite{TKU2023arXiv}, 
we obtain the following result.
\begin{prop}\label{prop3}
Let $\alpha \in (0,2)$ and $\delta/\sqrt{\nu \Delta} \equiv r \in(0,\infty)$. 
Under [A2], it follows that
\begin{equation*}
\EE[(T_{i,j,k}X)^2]
=\nu^{\alpha -1} \Delta^\alpha \sigma^2
\ee^{-\kappa(\widetilde y_{j-1}+\widetilde y_j)/2}
\ee^{-\eta (\widetilde z_{k-1}+\widetilde z_k)/2}
\psi_{r,\alpha}
+ R_{i,j,k} + \OO(\nu^{\alpha -1} \Delta^{1+\alpha}),
\end{equation*}
where $\sum_{i=1}^N R_{i,j,k} = \OO(\nu^{\alpha -1} \Delta^{\alpha})$
uniformly in $j,k$. 
\end{prop}
The estimator of $\sigma^2$ is defined as 
\begin{align*}
\hat \sigma^2 =
\hat \sigma_{\nu,N,m}^2 &=
\frac{1}{N \nu^{\alpha -1} \Delta^\alpha} 
\sum_{k=1}^{m_2} \sum_{j=1}^{m_1} \sum_{i=1}^N
(T_{i,j,k} X)^2 
\ee^{(-\kappa(\widetilde y_{j-1} +\widetilde y_j)
-\eta(\widetilde z_{k-1} +\widetilde z_k))/2}
\\
&\qquad \times
\Biggl(
\psi_{r,\alpha}
\sum_{k=1}^{m_2} \sum_{j=1}^{m_1}
\ee^{-\kappa(\widetilde y_{j-1} +\widetilde y_j)
- \eta(\widetilde z_{k-1} +\widetilde z_k)} \Biggr)^{-1}.
\end{align*}
Then, we have the following asymptotic property of the estimator $\hat \sigma^2$.
\begin{prop}\label{prop4}
Under [A2], 
as $N \to \infty$, $m_1 \land m_2 \to \infty$ and $\nu \to 0$,
\begin{equation*}
N (\hat \sigma^2 -(\sigma^*)^2) = \OO_\PP(1).
\end{equation*}
\end{prop}

\begin{rmk}
\begin{enumerate}
\item[(1)]
The estimator $\hat \sigma^2$ is defined as $\sigma^2$ 
which minimizes the contrast function
\begin{equation*}
K_{\nu,N,m}(\sigma^2) =
\sum_{k=1}^{m_2} \sum_{j=1}^{m_1} 
\Biggl\{
\frac{1}{N \nu^{\alpha -1} \Delta^\alpha} \sum_{i=1}^N (T_{i,j,k} X)^2 
-\sigma^2 \psi_{r,\alpha}
\ee^{(-\kappa(\widetilde y_{j-1} +\widetilde y_j)
- \eta(\widetilde z_{k-1} +\widetilde z_k))/2}
\Biggr\}^2.
\end{equation*}
This is because $\sigma^2$ can be identified 
without using two rates $ r = \delta/\sqrt{\nu \Delta}$ and 
$r' = \delta/\sqrt{2\nu \Delta}$ 
if the diffusive parameter $\nu$ is known, see \cite{TKU2023arXiv}. 

\item[(2)]
The estimator $\hat \sigma^2$ has the same rate 
as that in \cite{TKU2023arXiv}, even under $\nu \to 0$.
This is due to the fact that 
the main term of $\EE[(T_{i,j,k}X)^2]$ is independent of $\theta_0$,
see the proof of Lemma 4.1 in \cite{TKU2023arXiv}.
\end{enumerate}
\end{rmk}

\section{Simulations}\label{sec3}
The numerical solution of SPDE \eqref{spde} is generated by
\begin{equation*}
\tilde X_{t_{i}}(y_j, z_k)
= \sum_{l_1=1}^{K_1} \sum_{l_2=1}^{K_2} 
x_{l_1,l_2}(t_{i}) e_{l_1,l_2}(y_j, z_k), 
\quad i = 1,\ldots, N, j = 1,\ldots, M_1, k = 1,\ldots, M_2
\end{equation*}
with
\begin{equation*}
\dd x_{l_1,l_2}(t) = -\lambda_{l_1,l_2}^{(\theta_0)} x_{l_1,l_2}(t) \dd t
+\sigma \mu_{l_1,l_2}^{-\alpha/2} \dd w_{l_1,l_2}(t), 
\quad
x_{l_1,l_2}(0) = \langle X_0, e_{l_1,l_2} \rangle,
\end{equation*}
where $\lambda_{l_1,l_2}^{(\theta_0)} = \nu \lambda_{l_1,l_2} -\theta_0$.
In this simulation, the true values of parameters
$(\theta_0^*, \kappa, \eta, \sigma^*) = (0,1,1,1)$.
We set $N = 10^2$, $M_1 = M_2 = 200$, $K_1 = K_2 = 10^4$, 
$X_0 = 0$, $\alpha = 0.5$, $\nu =0.1$.
We used R language to compute the estimators of Theorem \ref{th2} and Proposition \ref{prop4}.
The number of iterations is $200$.

First, we estimated $\theta_0$.
Table \ref{tab1} is the simulation results of 
the mean and the standard deviation (s.d.) of $\hat \theta_{0,\beta}$.
We set that  $(N, M_1,M_2,\alpha,\nu,\beta,L) = (10^2, 200,200,0.5,0.1,0.6,32)$.
In this case, ${\rm n}=2$, ${\rm m} \approx 4.6$ and $\ell \approx 1.505$
in Corollary \ref{cor1}, thus
$\ell > \frac{1 +{\rm n}\phi_{\alpha(\beta +1),\alpha}^{(\mrm{trunc})}}{2} = 1.5$
and $\beta > \beta_{\alpha,{\rm n},{\rm m}}^{(\mrm{cons})} \approx 0.20 $. 
\begin{table}[h]
\caption{Simulation results of $\hat{\theta}_{0,\beta}$ \label{tab1}}
\begin{center}
\begin{tabular}{c|c} \hline
		&$\hat{\theta}_{0,\beta}$
\\ \hline
true value & 0
\\ \hline
mean & -0.029
 \\
s.d. & (0.274)
 \\   \hline
\end{tabular}
\end{center}
\end{table}
It seems from Table \ref{tab1} that the bias of $\hat{\theta}_{0,\beta}$ is very small
and
$\hat{\theta}_{0,\beta}$ has a good performance.

Next, we estimated $\sigma^2$.
Table \ref{tab2} is the simulation results of 
the mean and the standard deviation (s.d.) of $\hat{\sigma}^2$
with $(N, m_1,m_2,\alpha,\nu)$ 
$ = (10^2, 30,30, 0.5,0.1)$.
\begin{table}[h]
\caption{Simulation results of $\hat{\sigma}^2$ \label{tab2}}
\begin{center}
\begin{tabular}{c|c} \hline
		&$\hat{\sigma}^2$
\\ \hline
true value &1  
\\ \hline
mean & 0.987
 \\
s.d. & (0.008)
 \\   \hline
\end{tabular}
\end{center}
\end{table}
We see from Table \ref{tab2} that 
the bias of $\hat{\sigma}^2$ is very small.

\section{Proofs}\label{sec4}
We set the following notation.
\begin{enumerate}
\item 
$\| \cdot \|_\infty$ denotes the uniform norm on $D$.

\item 
For a continuous square integrable martingale $\{ M_t \}_{t \ge 0}$ 
with $M_0=0$, $\lqv M \rqv_t$ stands for the quadratic variation of 
$\{ M_t \}_{t \ge 0}$.

\item 
For $\gamma \in [-1, \infty)$, define 
$\mcl U_\gamma = \{ \widetilde Q^{(\gamma+1)/2} v | v \in \mcl H \}$ 
with the induced norm $\| \widetilde Q^{-(\gamma +1)/2} u \|$, $u \in \mcl U_\gamma$.

\item
\textit{Hilbert-Schmidt space}.
For separable Hilbert spaces $\mcl K_1$ and $\mcl K_2$, 
let $\mrm{HS}(\mcl K_1; \mcl K_2) = \{ L:\mcl K_1 \to \mcl K_2 | 
L \text{ is bounded linear, } 
\| L \|_{\mrm{HS}(\mcl K_1; \mcl K_2)} < \infty \}$, where
\begin{equation*}
\| L \|_{\mrm{HS}(\mcl K_1; \mcl K_2)} = 
\sqrt{\sum_{j \ge 1} \| L \phi_j \|_{\mcl K_2}^2}
\end{equation*}
and $\{ \phi_j \}_{j \in \mbb N}$ is a complete orthonormal system of $\mcl K_1$.
We write $\mrm{HS}(\mcl K_1; \mcl K_2)$ as $\mrm{HS}(\mcl K_1)$ if $\mcl K_2 = \mcl H$.

\item \textit{H\"{o}lder space}.
For $\gamma \in (0,1]$, define
$\mcl C^\gamma = \{ u:D \to \mbb R | \| u \|_{\mcl C^\gamma} < \infty \}$,
where
\begin{equation*}
\| u \|_{\mcl C^\gamma} = \| u \|_\infty 
+ \sup_{\bs x \neq \bs y \in D} 
\frac{|u(\bs x) -u(\bs y)|}{|\bs x - \bs y|^\gamma}.
\end{equation*}

\item \textit{Sobolev space}.
For $\gamma \in (0,1)$ and $q \ge 2$, let
$\mcl W^{\gamma,q} 
= \{ u \in \mcl L^q | \| u \|_{\mcl W^{\gamma,q}} < \infty \}$, where
\begin{equation*}
\| u \|_{\mcl W^{\gamma,q}} = 
\Biggl(
\| u \|_{\mcl L^q}^q
+ \int_{D} \int_{D} 
\frac{|u(\bs x) -u(\bs y)|^q}
{|\bs x - \bs y|^{2+\gamma q}} \dd \bs x \dd \bs y
\Biggr)^{1/q}.
\end{equation*}

\item 
Let $\bar v(y,z) = \ee^{\kappa y + \eta z}$ and $v(y,z) = 1/\bar v(y,z)$.
Note that $\bar v$ is the weight function in the inner product 
$\langle \cdot,\cdot \rangle$.
\end{enumerate}

\subsection{Proof of Theorem \ref{th1}}
For $\beta > -1$ and $t >0$, define
\begin{equation*}
J_{\beta,\nu} = \int_0^1 \| X_t \|_{\mcl L_{\beta}^2}^2 \dd t,
\quad
K_{\beta,\nu} = 
\int_0^1 \langle X_t, \dd X_t + \nu \mcl A X_t \dd t \rangle_{\mcl L_{\beta}^2},
\end{equation*}
\begin{equation*}
M_{\beta,\nu,t} = 
\int_0^t \langle X_s, \dd W_s^Q \rangle_{\mcl L_{\beta}^2},
\quad
\mcl M_{\beta,\nu,t} = \frac{M_{\beta,\nu,t}}{\sqrt{\EE[\lqv M_{\beta,\nu} \rqv_t]}}.
\end{equation*}

From $\mcl R_{\beta,\nu} = \EE[J_{\beta,\nu}]/\sqrt{\EE[J_{2\beta+1,\nu}]}$
and Proposition \ref{propA1} below, 
one can easily show the first claim of Theorem \ref{th1}.

We next show (2) of Theorem \ref{th1}. Note that 
$\{ M_{\beta,\nu,t} \}_{t \in \mbb T}$ is continuous square integrable martingale,
$\hat \theta_{0,\beta}^{(\mrm{cont})}= K_{\beta,\nu}/J_{\beta,\nu}$ and 
\begin{equation*}
K_{\beta,\nu} = 
\int_0^1 \langle X_t, 
\theta_0^* X_t \dd t + \sigma^* \dd W_t^Q \rangle_{\mcl L_{\beta}^2}
= \theta_0^* J_{\beta,\nu} +\sigma^* M_{\beta,\nu,1}
\end{equation*}
under $\PP_{\theta_0^*,\sigma^*}$. 
Since $\lqv M_{\beta,\nu} \rqv_1 = J_{2\beta +1,\nu}$ and 
\begin{equation*}
\hat \theta_{0,\beta}^{(\mrm{cont})} - \theta_0^* = 
\frac{\sigma^* M_{\beta,\nu,1}}{J_{\beta,\nu}}=
\sigma^* \times \frac{\EE[J_{\beta,\nu}]}{J_{\beta,\nu}}
\times \frac{\sqrt{\EE[\lqv M_{\beta,\nu} \rqv_1]}}{\EE[J_{\beta,\nu}]}
\times \mcl M_{\beta,\nu,1},
\end{equation*}
it holds that
\begin{equation}\label{th2-CLT}
\mcl R_{\beta,\nu}(\hat \theta_{0,\beta}^{(\mrm{cont})} - \theta_0^*) 
= \sigma^* \times \frac{\EE[J_{\beta,\nu}]}{J_{\beta,\nu}}
\times \mcl M_{\beta,\nu,1}.
\end{equation}
Since it follows from Proposition \ref{propA2} below that 
for $-1 < 2\beta +1 \le \frac{1}{\alpha} -1$, 
\begin{equation*}
\lqv \mcl M_{\beta,\nu} \rqv_1 = 
\frac{\lqv M_{\beta,\nu} \rqv_1}{\EE[\lqv M_{\beta,\nu} \rqv_1]}
= \frac{J_{2\beta +1, \nu}}{\EE[J_{2\beta+1,\nu}]} \pto 1
\end{equation*}
as $\nu \to 0$, we find from Theorem A.1 in \cite{Altmeyer_etal2022arXiv} that
$\mcl M_{\beta,\nu,1} \dto N(0,1)$
as $\nu \to 0$ for $-1 < \beta \le \frac{1}{2\alpha} -1$.
With Slutsky's theorem and Proposition \ref{propA2}, we complete the proof of (2).

Finally, we verify (3) of Theorem \ref{th1}. 
Using Proposition \ref{propA1}, we have 
$\EE[M_{\beta,\nu,1}^2] = \EE[\lqv M_{\beta,\nu} \rqv_1] 
= \EE[J_{2\beta+1,\nu}]< \infty$ 
for any $\nu >0$ and $-1 < \beta \le \frac{1}{\alpha} -1$, 
and we obtain $\EE[ \mcl M_{\beta,\nu,1}^2] = 1$ 
for $\nu >0$ and $-1 < \beta \le \frac{1}{\alpha} -1$.
For any $\epsilon >0$, we can choose a constant $C > \epsilon^{-1/2}$ such that
\begin{equation*}
\sup_{\nu >0} \PP(|\mcl M_{\beta,\nu,1}| >C)
\le C^{-2} \sup_{\nu >0} \EE[\mcl M_{\beta,\nu,1}^2] = C^{-2} < \epsilon,
\end{equation*}
which implies $\mcl M_{\beta,\nu,1} = \OO_\PP(1)$ for 
$-1 < \beta \le \frac{1}{\alpha} -1$.
By \eqref{th2-CLT} and Proposition \ref{propA2}, we get the desired result.


For $\beta > -1$ and $0 \le u \le t \le 1$, let
\begin{equation*}
\overline X_{t,u} = \sigma^* \int_u^t S_{\nu(t-s)} \dd W_s^Q,
\quad 
\overline X_t = \overline X_{t,0},
\quad 
\widetilde X_t = X_t - \overline X_t,
\end{equation*}
\begin{equation*}
I_{\beta,\nu,t} 
= \int_0^t \| \overline X_s \|_{\mcl L_{\beta}^2}^2 \dd s,
\quad
I_{\beta,\nu} = I_{\beta,\nu,1},
\end{equation*}
\begin{equation*}
\varphi_\beta(\nu,t) 
= \int_0^t \| S_{\nu s} \|_{\mrm{HS}(\mcl U_\beta)}^2 \dd s,
\quad 
\varphi_\beta(\nu) = \varphi_\beta(\nu,1).
\end{equation*}
\begin{prop}\label{propA1}
Let $\alpha >0$ and $\beta >-1$. Then, 
$S_t \in \mrm{HS}(\mcl U_\beta)$ for $t >0$ and
\begin{equation}\label{prop5-varphi}
\varphi_\beta(\nu) \sim
\begin{cases}
\nu^{\alpha(\beta+1)-1}, & \alpha(\beta +1) < 1,
\\
-\log \nu, & \alpha(\beta +1) = 1,
\\
1, & \alpha(\beta +1) > 1
\end{cases}
\end{equation}
as $\nu \to 0$. Furthermore, if $-1 < \beta \le \frac{1}{\alpha} -1$, 
then under [A1]$_{\beta,2}$, 
\begin{equation*}
\EE[J_{\beta,\nu}] \sim  \EE[I_{\beta,\nu}] \sim \varphi_\beta(\nu)
\end{equation*}
as $\nu \to 0$.
\end{prop}

\begin{prop}\label{propA2}
Let $\alpha >0$. If $-1 < \beta \le \frac{1}{\alpha} -1$, then under [A1]$_{\beta,2}$,
\begin{equation*}
\frac{\EE[J_{\beta,\nu}]}{J_{\beta,\nu}} \pto 1
\end{equation*}
as $\nu \to 0$.
\end{prop}
The rest of this subsection is devoted to the proofs of these two propositions.

\subsubsection{Proof of Proposition \ref{propA1}}
The proof consists of four steps.
Specifically, we will show $S_t \in \mrm{HS}(\mcl U_\beta)$ and 
\eqref{prop5-varphi} in Step 1, 
prove $\EE[I_{\beta,\nu}] \sim \varphi_\beta(\nu)$ in Step 2, 
verify $\EE[J_{\beta,\nu}] \lesssim \varphi_\beta(\nu)$ in Step 3 
and show $\EE[J_{\beta,\nu}] \gtrsim \EE[I_{\beta,\nu}]$ under $\nu \to 0$ in Step 4.

\textbf{Step 1.}
We can check immediately that $S_t \in \mrm{HS}(\mcl U_\beta)$ for $t >0$
since $\| S_{t} \|_{\mrm{HS}(\mcl U_\beta)} 
= \| S_{t} \widetilde Q^{(\beta+1)/2} \|_{\mrm{HS}(\mcl H)}$ and 
\begin{align*}
\| S_{t} \widetilde Q^\gamma \|_{\mrm{HS}(\mcl H)}^2
= \sum_{l_1,l_2 \ge 1} 
\ee^{-2t \lambda_{l_1,l_2}} \mu_{l_1,l_2}^{-2\alpha \gamma}
\le \sum_{l_1,l_2 \ge 1} 
\frac{1}{2t \lambda_{l_1,l_2} \mu_{l_1,l_2}^{2\alpha \gamma}}
< \infty
\end{align*}
for $\gamma > 0$ and $t >0$. We next show \eqref{prop5-varphi}.
Note that $\nu \in (0,1)$.
It follows from $\frac{1-\ee^{-ax}}{x} \sim \frac{1}{x} \land a$ ($a >0$) that 
for $t \in (0,1]$ and $\gamma >0$,
\begin{align*}
\int_0^t 
\| S_{\nu s} \widetilde Q^{\gamma} \|_{\mrm{HS}(\mcl H)}^2 \dd s
&= \int_0^t \sum_{l_1,l_2 \ge 1} 
\ee^{-2\nu s \lambda_{l_1,l_2}} \mu_{l_1,l_2}^{-2\alpha \gamma} \dd s
\\
&= \sum_{l_1,l_2 \ge 1} 
\frac{1-\ee^{-2\nu t \lambda_{l_1,l_2}}}{2 \nu \lambda_{l_1,l_2}} 
\mu_{l_1,l_2}^{-2\alpha \gamma}
\\
&\sim \sum_{l_1,l_2 \ge 1} 
\biggl(
\frac{1}{\nu(l_1^2+l_2^2)} \land t
\biggr)
\frac{1}{(l_1^2+l_2^2)^{2\alpha \gamma}}
\\
&= \sum_{l_1^2+l_2^2 > 1/\nu t} 
\frac{1}{\nu(l_1^2+l_2^2)^{1+2\alpha \gamma}}
+ \sum_{1 < l_1^2+l_2^2 \le 1/\nu t} 
\frac{t}{(l_1^2+l_2^2)^{2\alpha \gamma}}
\\
&\sim
\nu^{2\alpha \gamma-1} t^{2\alpha \gamma} + 
\begin{cases}
\nu^{2\alpha \gamma-1} t^{2\alpha \gamma}, & 2\alpha \gamma < 1,
\\
-t \log (\nu t), &  2\alpha \gamma = 1,
\\
t, & 2\alpha \gamma > 1
\end{cases}
\\
&\sim
\begin{cases}
\nu^{2\alpha \gamma-1} t^{2\alpha \gamma}, & 2\alpha \gamma < 1,
\\
-t \log (\nu t), &  2\alpha \gamma = 1,
\\
t, & 2\alpha \gamma > 1.
\end{cases}
\end{align*}
We therefore get the desired result. 

\textbf{Step 2.}
We verify that 
\begin{equation}\label{Pf_prop5-01}
\EE[I_{\beta,\nu}] \sim \varphi_\beta(\nu),
\end{equation}
and that for $\tau \in (0,1]$, 
\begin{equation}\label{Pf_prop5-02}
\EE[I_{\beta,\nu}] \lesssim \int_0^\tau \varphi_\beta(\nu,s) \dd s.
\end{equation}
It is shown that 
\begin{equation}\label{Pf_prop5-03}
\varphi_\beta(\nu,t_1) \le \varphi_\beta(\nu,t_2) 
\le \frac{t_2}{t_1}\varphi_\beta(\nu,t_1),
\quad 0< t_1 \le t_2 \le 1
\end{equation}
by using the facts that (i) for $\gamma >0$, 
the function 
$t \mapsto \| S_t \widetilde Q^\gamma \|_{\mrm{HS}(\mcl H)}$ 
is decreasing in $t \ge 0$,
and (ii) for a non-negative decreasing function $f:[0,\infty) \to \mbb R$,
\begin{equation*}
\frac{1}{t_2} \int_0^{t_2} f(s) \dd s \le \frac{1}{t_1} \int_0^{t_1} f(s) \dd s,
\quad 0 < t_1 \le t_2.
\end{equation*}
Since it follows from the It\^{o} isometry \eqref{stoch-conv-isometry} that
\begin{equation}\label{Pf_prop5-04}
\EE\Bigl[ \| \overline X_t \|_{\mcl L_{\beta}^2}^2 \Bigr]
= (\sigma^*)^2 \int_0^t 
\| S_{\nu(t-s)} \|_{\mrm{HS}(\mcl U_\beta)}^2 \dd s
= (\sigma^*)^2 \varphi_\beta(\nu,t),
\end{equation}
and from \eqref{Pf_prop5-03} that for $\tau \in (0,1]$,
\begin{equation*}
\frac{\tau \varphi_\beta(\nu,\tau)}{2}
\le
\int_0^\tau \varphi_\beta(\nu,s) \dd s
\le 
\tau \varphi_\beta(\nu,\tau),
\end{equation*}
we find from \eqref{Pf_prop5-03} again that
\begin{equation*}
\EE[I_{\beta,\nu,\tau}]
= \int_0^\tau \EE\Bigl[ \| \overline X_s \|_{\mcl L_{\beta}^2}^2 \Bigr] \dd s
\sim \tau \varphi_\beta(\nu,\tau),
\quad
\int_0^\tau \varphi_\beta(\nu,s) \dd s
\ge \frac{\tau^2 \varphi_\beta(\nu,1)}{2}
\sim \tau^2 \EE[I_{\beta,\nu}],
\end{equation*}
which imply that \eqref{Pf_prop5-01} and \eqref{Pf_prop5-02} hold.

\textbf{Step 3.}
We show that $\EE[J_{\beta,\nu}] \lesssim \varphi_\beta(\nu)$.
Since
\begin{equation*}
\widetilde X_t 
= S_{\nu t}X_0 + \theta_0^* \int_0^t S_{\nu(t-s)}(\widetilde X_s + \overline X_s) \dd s,
\end{equation*}
it holds that for $t \in \mbb T$,
\begin{align*}
\EE\Bigl[ \| \widetilde X_t \|_{\mcl L_{\beta}^2}^2 \Bigr] 
&\le 3 
\Biggl\{
\EE\Bigl[\| S_{\nu t}X_0 \|_{\mcl L_{\beta}^2}^2\Bigr] 
+ (\theta_0^*)^2 
\biggl(
\int_0^t \EE\Bigl[\| S_{\nu(t-s)} \widetilde X_s \|_{\mcl L_{\beta}^2}^2\Bigr] \dd s
+\int_0^t \EE\Bigl[\| S_{\nu(t-s)} \overline X_s \|_{\mcl L_{\beta}^2}^2\Bigr] \dd s
\biggr)
\Biggr\}
\\
&\le 3 
\Biggl\{
\EE\Bigl[\| X_0 \|_{\mcl L_{\beta}^2}^2\Bigr] 
+ (\theta_0^*)^2 
\biggl(
\int_0^t \EE\Bigl[\| \widetilde X_s \|_{\mcl L_{\beta}^2}^2\Bigr] \dd s
+\int_0^t \EE\Bigr[\| \overline X_s \|_{\mcl L_{\beta}^2}^2\Bigr] \dd s
\biggr)
\Biggr\}.
\end{align*}
Using Gronwall's inequality, we see from [A1]$_{\beta,2}$, \eqref{Pf_prop5-03}
and \eqref{Pf_prop5-04} that
\begin{align}
\EE\Bigl[\| \widetilde X_t \|_{\mcl L_{\beta}^2}^2 \Bigr] 
&\le
3 \ee^{3 (\theta_0^*)^2 t}
\biggl(
\EE\Bigl[\| X_0 \|_{\mcl L_{\beta}^2}^2 \Bigr] 
+ (\theta_0^*)^2 
\int_0^t \EE\Bigl[\| \overline X_s\|_{\mcl L_{\beta}^2}^2 \Bigr] \dd s 
\biggr)
\nonumber
\\
&\lesssim
1 + \varphi_\beta(\nu,t)
\lesssim
\varphi_\beta(\nu,t).
\label{Pf_prop5-05}
\end{align}
Therefore, we find from \eqref{Pf_prop5-03}-\eqref{Pf_prop5-05} that
\begin{align*}
\EE[J_{\beta,\nu}]
&= \int_0^1 \EE\Bigl[\| X_t \|_{\mcl L_{\beta}^2}^2\Bigr] \dd t
\lesssim
\int_0^1 \EE\Bigl[\| \overline X_t \|_{\mcl L_{\beta}^2}^2\Bigr] \dd t
+ \int_0^1 \EE\Bigl[\| \widetilde X_t \|_{\mcl L_{\beta}^2}^2\Bigr] \dd t
\\
&\lesssim \int_0^1 \varphi_\beta(\nu,t) \dd t
\le \varphi_\beta(\nu).
\end{align*}

\textbf{Step 4.}
We prove that $\EE[J_{\beta,\nu}] \gtrsim \EE[I_{\beta,\nu}]$.
For any $t \in \mbb T$, it follows from \eqref{Pf_prop5-03}-\eqref{Pf_prop5-05} that
\begin{align*}
\EE\Bigl[\| X_t \|_{\mcl L_{\beta}^2}^2 \Bigr] 
&\ge \frac{\EE\Bigl[\| \overline X_t \|_{\mcl L_{\beta}^2}^2\Bigr]}{2} 
- \EE\Bigl[\| \widetilde X_t \|_{\mcl L_{\beta}^2}^2\Bigr]  
\\
&\ge
\frac{\EE\Bigl[\| \overline X_t \|_{\mcl L_{\beta}^2}^2\Bigr]}{2}
-3 \ee^{3 (\theta_0^*)^2 t}
\biggl(
\EE\Bigl[\| X_0 \|_{\mcl L_{\beta}^2}^2\Bigr] 
+ (\theta_0^*)^2 \int_0^t \EE\Bigl[\| \overline X_s \|_{\mcl L_{\beta}^2}^2\Bigr] \dd s
\biggr)
\\
&=
\frac{(\sigma^*)^2 \varphi_\beta(\nu,t)}{2}
-3 \ee^{3 (\theta_0^*)^2 t}
\biggl(
\EE\Bigl[\| X_0 \|_{\mcl L_{\beta}^2}^2\Bigr] 
+ (\theta_0^* \sigma^*)^2 \int_0^t \varphi_\beta(\nu,s) \dd s
\biggr)
\\
&\ge
\varphi_\beta(\nu,t)
\biggl( \frac{(\sigma^*)^2}{2}
-3 \ee^{3 (\theta_0^*)^2 t} (\theta_0^* \sigma^*)^2 t
\biggr)
-3\ee^{3 (\theta_0^*)^2 t}\EE\Bigl[\| X_0 \|_{\mcl L_{\beta}^2}^2\Bigr]
\\
&=:
\varphi_\beta(\nu,t) C_1(t) -C_2(t).
\end{align*}
Fix $\tau \in (0,1]$ such that $C_1(\tau) > 0$. 
Thanks to \eqref{prop5-varphi}, we can choose a sufficiently small $\nu$ to fulfill
\begin{equation*}
\frac{C_1(\tau)}{2} \int_0^\tau \varphi_{\beta}(\nu,t) \dd t \ge \tau C_2(\tau)
\end{equation*}
for $-1 < \beta \le \frac{1}{\alpha} -1$. 
Since $C_1(t)$ (resp. $C_2(t)$) is positive decreasing (resp. increasing) 
on $[0,\tau]$, we find from \eqref{Pf_prop5-02} that
\begin{align*} 
\EE[J_{\beta,\nu}] 
&\ge \int_0^\tau \EE\Bigl[\| X_t \|_{\mcl L_{\beta}^2}^2\Bigr] \dd t
\ge 
\int_0^\tau \bigl( \varphi_\beta(\nu,t) C_1(t) - C_2(t) \bigr) \dd t
\\
&\ge 
C_1(\tau) \int_0^\tau \varphi_\beta(\nu,t) \dd t - \tau C_2(\tau)
\\
&\ge 
\frac{C_1(\tau)}{2} \int_0^\tau \varphi_\beta(\nu,t) \dd t
\gtrsim \EE[I_{\beta,\nu}].
\end{align*}
This completes the proof.

\subsubsection{Proof of Proposition \ref{propA2}}
We use basic properties of the Malliavin calculus to obtain 
Proposition \ref{propA2}. To this end, we shall first give some properties 
of the Malliavin calculus. 

For $\beta > -1$, we define $X_{\beta,t} = \widetilde Q^{\beta/2} X_t$.
By operating $\widetilde Q^{\beta/2}$ to the both sides of \eqref{vcf},  
it holds from Lemmas \ref{lemA4}, \ref{lemA5} and \ref{lemB3} that 
under [A1]$_{\beta,2}$,
\begin{equation}\label{QX1}
X_{\beta,t} = S_{\nu t} X_{\beta,0} 
+ \theta_0 \int_0^t S_{\nu(t-s)} X_{\beta,s} \dd s
+\sigma \int_0^t \widetilde Q^{\beta/2} S_{\nu(t-s)} \dd W_s^Q,
\quad
t \in \mbb T
\end{equation}
on $L^2(\Omega; \mcl H)$.
Let $\widehat Q_\beta = Q \widetilde Q^{\beta}$.
Since $\widehat Q_\beta$ is a trace class operator 
on a Hilbert space and the $\widehat Q_\beta$-Wiener process $W_t^{\widehat Q_\beta}$ 
has the representation
\begin{equation*}
W_t^{\widehat Q_\beta} = 
\sum_{l_1,l_2\ge1} \mu_{l_1,l_2}^{-\alpha(\beta +1)/2} w_{l_1,l_2}(t) e_{l_1,l_2},
\end{equation*}
we can interpret \eqref{QX1} as  
\begin{equation}\label{QX2}
X_{\beta,t} = S_{\nu t} X_{\beta,0} 
+ \theta_0 \int_0^t S_{\nu(t-s)} X_{\beta,s} \dd s
+\sigma \int_0^t S_{\nu(t-s)} \dd W_s^{\widehat Q_\beta},
\quad t \in \mbb T
\end{equation}
by \eqref{eq-Pf-lem***}. For any $(t,\bs y) \in \mbb T \times D$,
the Malliavin derivative $\mcl D_{\beta} X_{\beta,t}(\bs y)$ 
of $X_{\beta,t}(\bs y)$ given in \eqref{QX2} based on $W_t^{\widehat Q_\beta}$ 
belongs to $L^2(\Omega; L^2(\mbb T; \mcl L_{\beta+1}^2))$ and satisfies
\begin{equation}\label{Mall-Der}
\mcl D_{\beta,\tau} X_{\beta,t}(\bs y) = 
\sigma G_{\nu(t-\tau)}(\bs y,\cdot) 
+ \theta_0 \int_0^t \int_D G_{\nu(t-s)}(\bs y,\bs z) 
\mcl D_{\beta,\tau} X_{\beta,s}(\bs z)
\dd \bs z \dd s
\end{equation}
for $\tau \in [0,t)$ and
$\mcl D_{\beta,\tau} X_{\beta,t}(\bs y) = 0$ for $\tau \in [t,1]$
(see Theorem 7.1 in \cite{Sanz-Solo2005}).
Then, we get
\begin{equation}\label{eq-S-S}
\sup_{(t,\bs y) \in \mbb T \times D} 
\EE \Biggl[ \biggl( 
\int_0^1 
\| \mcl D_{\beta,\tau} X_{\beta,t}(\bs y) \|_{\mcl L_{\beta+1}^2}^2 \dd \tau 
\biggr)^2 \Biggr]
< \infty.
\end{equation}

The proof of Proposition \ref{propA2} is complete 
if we can show that for $-1 < \beta \le \frac{1}{\alpha} -1$,
\begin{equation}\label{Pf_prop6-01}
\EE \Biggl[ \biggl( 
\int_0^1 \| \mcl D_{\beta,\tau} J_{\beta,\nu} 
\|_{\mcl L_{\beta +1}^2}^2 \dd \tau 
\biggr)^2 \Biggr]
\lesssim \sqrt{\EE[J_{\beta,\nu}^2]}.
\end{equation}
Indeed, 
it follows from the Poincar\'{e} inequality 
(Proposition 3.1 in \cite{Nourdin_etal2009}) that
\begin{equation*}
\VV[J_{\beta,\nu}] \le 
\EE \Biggl[ \biggl( 
\int_0^1 \| \mcl D_{\beta,\tau} J_{\beta,\nu} 
\|_{\mcl L_{\beta +1}^2}^2 \dd \tau 
\biggr)^2 \Biggr],
\end{equation*}
which together with \eqref{Pf_prop6-01} and
$\sqrt{\EE [J_{\beta,\nu}^2]} \le \sqrt{\VV[J_{\beta,\nu}]} + \EE[J_{\beta,\nu}]$
yields
\begin{equation*}
\VV[J_{\beta,\nu}] \lesssim \sqrt{\VV[J_{\beta,\nu}]} + \EE[J_{\beta,\nu}].
\end{equation*}
Solving this quadratic inequality, one has
\begin{equation*}
\VV[J_{\beta,\nu}] \lesssim 1+ \EE[J_{\beta,\nu}]
\lesssim \EE[J_{\beta,\nu}].
\end{equation*}
Therefore, we see from the Chebyshev inequality and Proposition \ref{propA1} that
for any $\epsilon > 0$, 
\begin{equation*}
\varlimsup_{\nu \to 0}
\PP \biggl( \biggl| \frac{J_{\beta,\nu}}{\EE[J_{\beta,\nu}]}-1 \biggr| 
> \epsilon \biggr)
\le \epsilon^{-2} \varlimsup_{\nu \to 0}
\VV \Biggl[\frac{J_{\beta,\nu}}{\EE[J_{\beta,\nu}]} \Biggr]
\lesssim 
\epsilon^{-2} \varlimsup_{\nu \to 0}
\frac{1}{\EE[J_{\beta,\nu}]}
= 0
\end{equation*}
if $\alpha >0$ and $-1 < \beta \le \frac{1}{\alpha} -1$.

We prove \eqref{Pf_prop6-01}.
Note that the chain rule of Malliavin calculus 
(Proposition 1.2.3 in \cite{Nualart2006}) yields  
\begin{equation*}
\mcl D_{\beta,\tau} J_{\beta,\nu} 
= 2\int_0^1 \int_D X_{\beta,t}(\bs y)
\mcl D_{\beta,\tau} X_{\beta,t} (\bs y) \bar v(\bs y) \dd \bs y \dd t.
\end{equation*}
By setting
\begin{equation*}
\xi_1(X_{\beta,t}) = \| X_{\beta,t} \bar v \| \ (>0),
\quad
\xi_2(X_{\beta,t},\tau) = 
\Biggl\| \int_D  
\frac{X_{\beta,t}(\bs y) \bar v(\bs y)}{\| X_{\beta,t} \bar v \|}
\mcl D_{\beta,\tau} X_{\beta,t}(\bs y) \dd \bs y
\Biggr\|_{\mcl L_{\beta+1}^2},
\end{equation*}
it holds from the Schwarz inequality and \eqref{Bochner-ineq} that
\begin{align*}
&\EE \Biggl[ \biggl( 
\int_0^1 \| \mcl D_{\beta,\tau} J_{\beta,\nu} 
\|_{\mcl L_{\beta +1}^2}^2 \dd \tau 
\biggr)^2 \Biggr]
\\
&\le 4 \EE\Biggl[
\int_0^1 \biggl( \int_\tau^1 
\biggl\| \int_D  
X_{\beta,t}(\bs y)
\mcl D_{\beta,\tau} X_{\beta,t}(\bs y) \bar v(\bs y) \dd \bs y 
\biggr\|_{\mcl L_{\beta+1}^2} \dd t \biggr)^2 \dd \tau
\Biggr]
\\
&\le 4 \EE\Biggl[ \biggl(\int_0^1 \xi_1(X_{\beta,t})^2 \dd t \biggr)^2 \Biggr]^{1/2}
\EE\Biggl[
\biggl(
\int_0^1 \int_\tau^1 \xi_2(X_{\beta,t},\tau)^2 \dd t \dd \tau
\biggr)^2
\Biggr]^{1/2}
\\
&\le 4 
\| \bar v \|_{\infty}^2 \EE[J_{\beta,\nu}^2]^{1/2}
\sup_{t \in \mbb T}
\EE\Biggl[ 
\biggl( \int_0^1 \xi_2(X_{\beta,t},\tau)^2 \dd \tau \biggr)^2 \Biggr]^{1/2}.
\end{align*}
If $\varphi(\bs y)=\frac{X_{\beta,t}(\bs y) \bar v(\bs y)}{\| X_{\beta,t} \bar v \|}$,
then it follows from Lemma \ref{lemA1} below that
\begin{align*}
\sup_{t \in \mbb T}
\EE\Biggl[ \biggl( \int_0^1 \xi_2(X_{\beta,t},\tau)^2 \dd \tau \biggr)^2 \Biggr]
&\le 
\sup_{t \in \mbb T, \| \varphi \| = 1}
\EE\Biggl[ \biggl( \int_0^1
\biggl\| \int_D  
\varphi(\bs y) \mcl D_{\beta,\tau} X_{\beta,t}(\bs y) 
\dd \bs y
\biggr\|_{\mcl L_{\beta+1}^2}^2
\dd \tau \biggr)^2 \Biggr]
\\
&\lesssim
(\sigma^* \ee^{|\kappa|+|\eta|} \lor 1)^4
\exp \Bigl(
4|\theta_0^*| \ee^{|\kappa|+|\eta|} (\sigma^* \ee^{|\kappa|+|\eta|} \lor 1)
\Bigr),
\end{align*}
which yields the desired result.

\subsubsection{Auxiliary results on Malliavin derivative}
Here, we provide auxiliary results on Malliavin derivative 
$\mcl D_{\beta,\tau} X_{\beta,t}$ in order to show Proposition \ref{propA2}.

\begin{lem}\label{lemA1}
For any $\varphi \in \mcl H$ with $\| \varphi \| = 1$ and $t, \tau \in \mbb T$, 
it holds that
\begin{align*}
&\EE \Biggl[
\biggl( \int_0^1
\biggl\| 
\int_D  \varphi(\bs y) 
\mcl D_{\beta,\tau} X_{\beta,t}(\bs y) \dd \bs y
\biggr\|_{\mcl L_{\beta +1}^2}^2
\dd \tau \biggr)^2
\Biggr]
\lesssim 
(\sigma^* \ee^{|\kappa|+|\eta|} \lor 1)^4
\exp \Bigl(
4|\theta_0^*| \ee^{|\kappa|+|\eta|} (\sigma^* \ee^{|\kappa|+|\eta|} \lor 1)
\Bigr).
\end{align*}
\end{lem}

We approximate the Malliavin derivative $\mcl D_{\beta,\tau} X_{\beta,t}(\bs y)$ 
by using the Picard iteration and show Lemma \ref{lemA1}.
Let $\tau \in \mbb T$. For $t \in (\tau,1]$ and $\bs y \in D$, we define
\begin{equation}\label{seq-u}
\begin{cases}
u_{t,\tau}^{(0)}(\bs y) = \sigma^* G_{\nu(t-\tau)}(\bs y,\cdot),
\\
\displaystyle
u_{t,\tau}^{(n)}(\bs y) = u_{t,\tau}^{(0)}(\bs y) 
+ \theta_0^* \int_\tau^t 
\int_D G_{\nu(t-s)}(\bs y,\bs z)
u_{s,\tau}^{(n-1)}(\bs z) \dd \bs z \dd s,
\quad n \in \mbb N.
\end{cases}
\end{equation}
Let $u_{t,\tau}^{(n)}(\bs y) \equiv 0$
for $t \in [0,\tau]$, $\bs y \in D$ and $n \in \mbb N \cup \{0\}$.

It is easily shown that Lemma \ref{lemA1} holds from the following two lemmas.
\begin{lem}\label{lemA2}
For $\varphi \in \mcl H$ with $\| \varphi \|=1$, 
$\tau \in \mbb T$, $t \in (\tau,1]$, $\beta >-1$
and $\{u_{t,\tau}^{(n)}\}_{n \in \mbb N}$ given in \eqref{seq-u}, it follows that
\begin{equation}\label{lem2-01}
\biggl\|
\int_D \varphi(\bs y) u_{t,\tau}^{(n)}(\bs y) \dd \bs y
\biggr\|_{\mcl L_{\beta +1}^2} 
\le \sigma^* \| v \|_\infty 
+ \sum_{k=1}^n \frac{(|\theta_0^*| \| v \|_\infty (\sigma^* \| v \|_\infty \lor 1))^k}
{k!} (t-\tau)^k.
\end{equation}
In particular, 
\begin{equation*}
\sup_{\substack{t, \tau \in \mbb T, \\ \| \varphi \| = 1, n \in \mbb N}}
\biggl\|
\int_D \varphi(\bs y) u_{t,\tau}^{(n)}(\bs y) \dd \bs y
\biggr\|_{\mcl L_{\beta +1}^2}
\le (\sigma^* \ee^{|\kappa|+|\eta|} \lor 1) 
\exp \Bigl(
{|\theta_0^*| \ee^{|\kappa|+|\eta|} (\sigma^* \ee^{|\kappa|+|\eta|} \lor 1)}
\Bigr).
\end{equation*}
\end{lem}
\begin{lem}\label{lemA3}
For any $t \in \mbb T$, it holds that 
\begin{equation*}
\lim_{n \to \infty}
\sup_{\| \varphi \|=1}
\EE\Biggl[
\biggl(
\int_0^1
\biggl\|
\int_D  \varphi(\bs y) 
\mcl D_{\beta,\tau} X_{\beta,t}(\bs y) \dd \bs y
-\int_D  \varphi(\bs y) u_{t,\tau}^{(n)}(\bs y) \dd \bs y
\biggr\|_{\mcl L_{\beta+1}^2}^2
\dd \tau
\biggr)^2
\Biggr]
=0.
\end{equation*}
\end{lem}

\begin{proof}[\bf{Proof of Lemma \ref{lemA2}}]
Let
\begin{equation*}
F_n(t,\tau,\varphi) = 
\biggl\|
\int_D \varphi(\bs y) u_{t,\tau}^{(n)}(\bs y) \dd \bs y
\biggr\|_{\mcl L_{\beta +1}^2}.
\end{equation*}
We prove \eqref{lem2-01} by induction.

(i) For $n=0$, we have
\begin{equation*}
F_0(t,\tau,\varphi) = 
\sigma^*
\biggl\|
\int_D \varphi(\bs y)
G_{\nu(t-\tau)}(\bs y,\cdot) \dd \bs y
\biggr\|_{\mcl L_{\beta+1}^2}
= \sigma^* \| S_{\nu(t-\tau)} \varphi v \|_{\mcl L_{\beta+1}^2}
\le \sigma^* \| v \|_\infty.
\end{equation*}

(ii) Let \eqref{lem2-01} be true for $n$. Then, we see from \eqref{Bochner-ineq} that
\begin{align*}
F_{n+1}(t,\tau,\varphi) 
&\le
\biggl\|
\int_D  \varphi(\bs y) u_{t,\tau}^{(0)}(\bs y) \dd \bs y
\biggr\|_{\mcl L_{\beta+1}^2}
\nonumber
\\
&\qquad
+ |\theta_0^*| \biggl\| \int_D  \varphi(\bs y)
\int_\tau^t 
\int_D  G_{\nu(t-s)}(\bs y,\bs z)u_{s,\tau}^{(n)}(\bs z) 
\dd \bs z \dd s \dd \bs y
\biggr\|_{\mcl L_{\beta+1}^2}
\nonumber
\\
&\le
\sigma^* \| v \|_\infty
+ |\theta_0^*| \int_\tau^t 
\biggl\|
\int_D u_{s,\tau}^{(n)}(\bs z) 
\int_D G_{\nu(t-s)}(\bs y,\bs z) \varphi(\bs y) 
\dd \bs y \dd \bs z
\biggr\|_{\mcl L_{\beta+1}^2} \dd s
\nonumber
\\
&= \sigma^* \| v \|_\infty
+ |\theta_0^*| \int_\tau^t 
\biggl\|
\int_D u_{s,\tau}^{(n)}(\bs z) 
\widetilde \varphi_{s,t}(\bs z) \dd \bs z  
\biggr\|_{\mcl L_{\beta+1}^2} \dd s
\nonumber
\\
&= \sigma^* \| v \|_\infty
+ |\theta_0^*| \int_\tau^t 
\| \widetilde \varphi_{s,t} \|
\biggl\|
\int_D u_{s,\tau}^{(n)}(\bs z) 
\widehat \varphi_{s,t}(\bs z) \dd \bs z  
\biggr\|_{\mcl L_{\beta+1}^2} \dd s,
\end{align*}
where
\begin{equation*}
\widetilde \varphi_{s,t}(\cdot) = 
\int_D  G_{\nu(t-s)}(\bs y,\cdot) \varphi(\bs y) \dd \bs y
= S_{\nu(t-s)} \varphi v(\cdot),
\quad
\widehat \varphi_{s,t} = 
\begin{cases}
\widetilde \varphi_{s,t} /\| \widetilde \varphi_{s,t} \|, 
& \| \widetilde \varphi_{s,t} \| \neq 0,
\\
0, & \| \widetilde \varphi_{s,t} \| = 0.
\end{cases}
\end{equation*}
Since $\| \widetilde \varphi_{s,t} \| \le \| v \|_\infty$ 
and $\| \widehat \varphi_{s,t} \| =1$ if $\| \widetilde \varphi_{s,t} \| \neq 0$,  
it follows from the induction hypothesis that 
\begin{align*}
F_{n+1}(t,\tau,\varphi) 
&\le
\sigma^* \| v \|_\infty
+ |\theta_0^*| \| v \|_\infty
\int_\tau^t 
\biggl\{
\sigma^* \| v \|_\infty 
+ \sum_{k=1}^n \frac{(|\theta_0| \| v \|_\infty (\sigma^* \| v \|_\infty \lor 1))^k}
{k!} (s-\tau)^k
\biggr\} \dd s
\\
&\le
\sigma^* \| v \|_\infty
+ |\theta_0^*| \| v \|_\infty (\sigma^* \| v \|_\infty \lor 1) (t-\tau) 
+ \sum_{k=1}^n \frac{(|\theta_0^*| \| v \|_\infty 
(\sigma^* \| v \|_\infty \lor 1))^{k+1}}
{(k+1)!} (t-\tau)^{k+1}
\\
&=
\sigma^* \| v \|_\infty
+ \sum_{k=1}^{n+1} 
\frac{(|\theta_0^*| \| v \|_\infty (\sigma^* \| v \|_\infty \lor 1))^k}{k!} (t-\tau)^k.
\end{align*}
Therefore, \eqref{lem2-01} is true for $n+1$.

The last claim follows immediately 
from \eqref{lem2-01} and $\| v \|_\infty \le \ee^{|\kappa|+|\eta|}$.
\end{proof}

\begin{proof}[\bf{Proof of Lemma \ref{lemA3}}]
For $\varphi \in \mcl H$ with $\| \varphi \|=1$, let
\begin{equation*}
\widetilde F_n(t,\tau,\varphi) =
\biggl\|
\int_D  \varphi(\bs y) 
\mcl D_{\beta,\tau} X_{\beta,t}(\bs y) \dd \bs y
-\int_D  \varphi(\bs y) u_{t,\tau}^{(n)}(\bs y) \dd \bs y
\biggr\|_{\mcl L_{\beta+1}^2}^2.
\end{equation*}
It follows from \eqref{Mall-Der}, \eqref{seq-u} and \eqref{Bochner-ineq} that
\begin{align*}
\widetilde F_n(t,\tau,\varphi) &=
(\theta_0^*)^2
\biggl\|
\int_D  \varphi(\bs y) 
\int_\tau^t \int_D  G_{\nu(t-s)}(\bs y,\bs z)
(\mcl D_{\beta,\tau} X_{\beta,s}(\bs z) 
- u_{s,\tau}^{(n-1)}(\bs z)) \dd \bs z 
\dd s \dd \bs y
\biggr\|_{\mcl L_{\beta+1}^2}^2
\nonumber
\\
&\le
(\theta_0^*)^2
\Biggl(
\int_\tau^t 
\biggl\|
\int_D 
(\mcl D_{\beta,\tau} X_{\beta,s}(\bs z)
-u_{s,\tau}^{(n-1)}(\bs z)) 
\widetilde \varphi_{s,t}(\bs z) \dd \bs z 
\biggr\|_{\mcl L_{\beta+1}^2} \dd s
\Biggr)^2
\nonumber
\\
&\le
(\theta_0^*)^2 \| v \|_\infty^2 (t-\tau)
\int_\tau^t 
\biggl\|
\int_D (\mcl D_{\beta,\tau} X_{\beta,s}(\bs z) 
-u_{s,\tau}^{(n-1)}(\bs z)) \widehat \varphi_{s,t}(\bs z) \dd \bs z
\biggr\|_{\mcl L_{\beta+1}^2}^2 \dd s
\nonumber
\\
&\le (\theta_0^*)^2 \| v \|_\infty^2
\int_\tau^t \widetilde F_{n-1}(s, \tau, \widehat \varphi_{s,t}) \dd s.
\end{align*}
When $\| \widetilde \varphi_{s,t} \| =0$, 
it is obvious that Lemma \ref{lemA3} holds 
because $\widetilde F_n(s,\tau,\widehat \varphi_{s,t}) =0$. 
Hereafter, let $\| \widetilde \varphi_{s,t} \| \neq 0$,
that is, $\| \widehat \varphi_{s,t} \| = 1$. 
Let 
\begin{equation*}
\widehat{F}_n(t) = 
\sup_{\| \varphi \| = 1}
\EE\Biggl[ \biggl( \int_0^1 \widetilde F_n(t,\tau,\varphi) \dd \tau \biggr)^2 \Biggr].
\end{equation*}
In order to get the desired result, we prove that 
for any $t \in \mbb T$, $\widehat{F}_n(t) \to 0$ as $n \to \infty$. Since 
\begin{align*}
\widehat{F}_n(t) &= 
\sup_{\| \varphi \| = 1}
\EE\Biggl[ \biggl( \int_0^t \widetilde F_n(t,\tau,\varphi) \dd \tau \biggr)^2 \Biggr]
\\
&\le
(\theta_0^*)^2 \| v \|_\infty^2
\sup_{\| \varphi \| = 1}
\EE\Biggl[ \biggl( \int_0^t 
\int_\tau^t \widetilde F_{n-1}(s, \tau, \varphi) \dd s
\dd \tau \biggr)^2 \Biggr]
\\
&=
(\theta_0^*)^2 \| v \|_\infty^2
\sup_{\| \varphi \| = 1}
\EE\Biggl[ \biggl( 
\int_0^t \int_0^s \widetilde F_{n-1}(s, \tau, \varphi) \dd \tau \dd s  
\biggr)^2 \Biggr]
\\
&\le
(\theta_0^*)^2 \| v \|_\infty^2
\int_0^t 
\sup_{\| \varphi \| = 1}
\EE\Biggl[ \biggl( 
\int_0^s \widetilde F_{n-1}(s, \tau, \varphi) \dd \tau 
\biggr)^2 \Biggr]\dd s  
\\
&=
(\theta_0^*)^2 \| v \|_\infty^2
\int_0^t \widehat F_{n-1}(s) \dd s, 
\end{align*}
it follows from Gronwall's lemma (Lemma 6.2 in \cite{Sanz-Solo2005}) that
there exists a non-negative sequence $\{ a_n \}_{n \in \mbb N}$ 
such that $a_n \to 0$ and 
$\widehat{F}_n(t) \le a_n \sup_{t \in \mbb T} \widehat{F}_0(t)$. 
Noting that
\begin{align*}
\widehat{F}_0(t) 
&\lesssim
\sup_{\| \varphi \| = 1}
\EE \Biggl[
\biggl( \int_0^t
\biggl\|
\int_D \varphi(\bs y) 
\mcl D_{\beta,\tau} X_{\beta,t}(\bs y) \dd \bs y
\biggr\|_{\mcl L_{\beta+1}^2}^2 \dd \tau \biggr)^2
\Biggr]
\nonumber
\\
&\qquad+
\sup_{\| \varphi \| = 1}
\biggl(
\int_0^t
\biggl\|
\int_D  \varphi(\bs y) G_{\nu(t-\tau)}(\bs y,\cdot) \dd \bs y
\biggr\|_{\mcl L_{\beta+1}^2}^2
\dd \tau
\biggr)^2
\nonumber
\\
&\le
\sup_{\| \varphi \| = 1}
\EE\Biggl[
\biggl(
\int_0^t
\biggl(
\int_D | \varphi(\bs y) | \| \mcl D_{\beta,\tau} X_{\beta,t}(\bs y) 
\|_{\mcl L_{\beta+1}^2} \dd \bs y
\biggr)^2
\dd \tau
\biggr)^2
\Biggr]
+(\sigma^* \| v \|_\infty)^4
\nonumber
\\
&\le
\sup_{\| \varphi \| = 1} 
\biggl( \int_D  \varphi(\bs y)^2 \dd \bs y \biggr)^2
\EE \Biggl[
\biggl(
\int_0^t \int_D  \| \mcl D_{\beta,\tau} X_{\beta,t}(\bs y) 
\|_{\mcl L_{\beta+1}^2}^2 
\dd \bs y \dd \tau
\biggr)^2
\Biggr]
+(\sigma^* \| v \|_\infty)^4
\nonumber
\\
&\le \| v \|_\infty^2 
\int_D 
\EE \Biggl[ 
\biggl(
\int_0^t \| \mcl D_{\beta,\tau} X_{\beta,t}(\bs y) 
\|_{\mcl L_{\beta+1}^2}^2 \dd \tau
\biggr)^2
\Biggr] \dd \bs y +(\sigma^* \| v \|_\infty)^4
\nonumber
\\
&\le
\| v \|_\infty^2 
\sup_{\bs y \in D}
\EE \Biggl[ 
\biggl(
\int_0^t \| \mcl D_{\beta,\tau} X_{\beta,t}(\bs y) 
\|_{\mcl L_{\beta+1}^2}^2 \dd \tau
\biggr)^2
\Biggr] +(\sigma^* \| v \|_\infty)^4,
\end{align*}
we find from \eqref{eq-S-S} that
\begin{equation*}
\sup_{t \in \mbb T} \widehat{F}_0(t) 
\lesssim 
\| v \|_\infty^2 
\sup_{(t,\bs y) \in \mbb T \times D} 
\EE \Biggl[ 
\biggl( 
\int_0^1 \| \mcl D_{\beta,\tau} X_{\beta,t}(\bs y) 
\|_{\mcl L_{\beta+1}^2}^2 \dd \tau 
\biggr)^2
\Biggr] + (\sigma^* \| v \|_\infty)^4
< \infty,
\end{equation*}
which yields $\widehat{F}_n(t) \to 0$ as $n \to \infty$. 
This concludes the proof.
\end{proof}

\subsubsection{Basic properties of closed linear operators}
Here, we present some useful properties of closed operators.
Note that the operator $\widetilde Q^\delta$, $\delta \in \mbb R$
given in \eqref{frac-Q} is a closed linear operator.

Let $(\mcl X, \mathscr A, \mu)$ be a measure space,
and $\mcl B_1$, $\mcl B_2$ be Banach spaces.
For a linear operator $L$, $\mathscr D(L)$ denotes the domain of $L$.
\begin{lem}\label{lemA4}
Let $u :\mcl X \to \mcl B_1$ be a $\mcl B_1$-valued Bochner 
$\mu$-integrable function.
\begin{enumerate}
\item[(1)]
If $L : \mcl B_1 \to \mcl B_2$ is a bounded linear operator, 
then $L u$ is $\mcl B_2$-valued Bochner $\mu$-integrable function 
on $\mcl X$, and
\begin{equation}\label{closed-integral}
L \int_{\mcl X} u(x) \dd \mu(x) = \int_{\mcl X} L u(x) \dd \mu(x).
\end{equation}

\item[(2)]
If $L: \mcl B_1 \supset \mathscr D(L) \to \mcl B_2$ is a closed linear operator 
and $L u$ is $\mcl B_2$-valued Bochner $\mu$-integrable function on $\mcl X$,
then $\int_{\mcl X} u(x) \dd \mu(x) \in \mathscr D(L)$ and
\eqref{closed-integral} holds.
\end{enumerate}
\end{lem}
For the proof of this lemma, see Corollary V.5.2 in \cite{Yoshida1980} and
Theorem 3.10.16 in \cite{Denkowski_etal2003}.

Let $(\mcl X, \mathscr A, \mu)$ be a finite measure space and $p \ge 1$.
If $u :\mcl X \to \mcl B_1$ is a $\mcl B_1$-valued Bochner 
$\mu$-integrable function and $L$ is the linear operator of Lemma \ref{lemA4} such that
$\| L u(\cdot) \|_{\mcl B_2}^p$ is Lebesgue $\mu$-integrable, then 
\begin{equation}\label{Bochner-ineq}
\biggl \| L \int_{\mcl X} u(x) \dd \mu(x) \biggr \|_{\mcl B_2}^p
\lesssim \int_{\mcl X} \| L u(x) \|_{\mcl B_2}^p \dd \mu(x).
\end{equation}
Indeed, one can find from \eqref{closed-integral}, 
Corollary V.5.1 in \cite{Yoshida1980} and the H\"{o}lder inequality that
\begin{align*}
\biggl \| L \int_{\mcl X} u(x) \dd \mu(x) \biggr \|_{\mcl B_2}^p
&= \biggl \| \int_{\mcl X} L u(x) \dd \mu(x) \biggr \|_{\mcl B_2}^p
\\
&\le 
\biggl( \int_{\mcl X} \| L u(x) \|_{\mcl B_2} \dd \mu(x) \biggr)^p
\\
&\le
\mu(\mcl X)^{p-1} \int_{\mcl X} \| L u(x) \|_{\mcl B_2}^p \dd \mu(x).
\end{align*}

The interchange of an integral and a closed operator 
such as \eqref{closed-integral} also holds 
for a stochastic integral and a closed operator 
(see Proposition 4.30 in \cite{DaPrato_Zabczyk2014}), 
but we here consider the interchange of 
the stochastic integral $\int_0^t (t-s)^{-\delta} S_{\nu(t-s)} \dd W_s^Q$ and 
the operator $\widetilde Q^{\beta/2}$.
\begin{lem}\label{lemA5}
Let $t \in (0,1]$, $\nu \in (0,1)$, $\alpha >0$ and $\beta > -1$. 
If $\delta < \alpha(\beta+1)/2$, then
$\int_0^t (t-u)^{-\delta} S_{\nu(t-u)} \dd W_u^Q \in L^2(\Omega;\mcl L_\beta^2)$,
$\int_0^t (t-u)^{-\delta} \widetilde Q^{\beta/2} S_{\nu(t -u)} \dd W_u^Q
\in L^2(\Omega;\mcl H)$
and
\begin{equation}\label{closed-stoch-integral}
\widetilde Q^{\beta/2} \int_0^t (t-u)^{-\delta} S_{\nu(t -u)} \dd W_u^Q 
= \int_0^t (t-u)^{-\delta} \widetilde Q^{\beta/2} S_{\nu(t -u)} \dd W_u^Q
\quad \text{on } L^2(\Omega; \mcl H).
\end{equation}
In particular, we obtain the It\^{o} isometry
\begin{equation}\label{stoch-conv-isometry}
\EE \Biggl[ \biggl \| 
\int_0^t S_{\nu(t-s)} \dd W_s^Q 
\biggr \|_{\mcl L_{\beta}^2}^2 \Biggr] 
= \int_0^t \| S_{\nu s} \|_{\mrm{HS}(\mcl U_\beta)}^2
\dd s.
\end{equation}
\end{lem}
\begin{proof}
In the same way as the proof of Proposition \ref{propA1},
it follows that for $\beta >-1$ and $\delta < \alpha(\beta+1)/2$, 
\begin{align}
\int_0^1 \| u^{-\delta} S_u \|_{\mrm{HS}(\mcl U_\beta)}^2 \dd u
&= \sum_{l_1,l_2 \ge 1} \mu_{l_1,l_2}^{-\alpha (\beta +1)} 
\int_0^1 u^{-2\delta} \ee^{-2u \lambda_{l_1,l_2}} \dd u
\nonumber
\\
&\le \sum_{l_1,l_2 \ge 1} \mu_{l_1,l_2}^{-\alpha (\beta +1)} 
\int_0^\infty u^{-2\delta} \ee^{-2u \lambda_{l_1,l_2}} \dd u
\nonumber
\\
&= \frac{\Gamma(1-2\delta)}{2^{1-2\delta}} 
\sum_{l_1,l_2 \ge 1} \frac{1}{\lambda_{l_1,l_2}^{1-2\delta}
\mu_{l_1,l_2}^{\alpha (\beta +1)}}
< \infty.
\label{eq-lemA5-1}
\end{align}
Since it follows that for $\beta > -1$ and $\delta < \alpha(\beta +1)/2$,
\begin{align}
\| u^{-\delta} S_u \|_{\mrm{HS}(\mcl U_0; \mcl L_\beta^2)}^2
&= \sum_{l_1,l_2 \ge 1} 
\| u^{-\delta} S_u \widetilde Q^{1/2} e_{l_1,l_2} \|_{\mcl L_\beta^2}^2
\nonumber
\\
&= \sum_{l_1,l_2 \ge 1} 
\| \widetilde Q^{\beta/2} u^{-\delta} S_u \widetilde Q^{1/2} e_{l_1,l_2} \|^2
\ (= \| \widetilde Q^{\beta/2} u^{-\delta} S_u \|_{\mrm{HS}(\mcl U_0)}^2)
\nonumber
\\
&= \sum_{l_1,l_2 \ge 1} 
\| u^{-\delta} S_u \widetilde Q^{(\beta+1)/2} e_{l_1,l_2} \|^2
\nonumber
\\
&
= \| u^{-\delta} S_u \|_{\mrm{HS}(\mcl U_\beta)}^2,
\label{eq-lemA5-2}
\end{align} 
the stochastic integrals in \eqref{closed-stoch-integral} can be defined and 
\begin{equation*}
\int_0^t (t-u)^{-\delta} S_{\nu(t-u)} \dd W_u^Q \in L^2(\Omega;\mcl L_\beta^2),
\quad
\int_0^t (t-u)^{-\delta} \widetilde Q^{\beta/2} S_{\nu(t-u)} \dd W_u^Q 
\in L^2(\Omega; \mcl H).
\end{equation*}
Therefore, we have the following representations
\begin{equation*}
\int_0^t (t-u)^{-\delta} S_{\nu(t-u)} \dd W_u^Q 
= \sum_{l_1,l_2 \ge 1} \mu_{l_1,l_2}^{-\alpha/2} e_{l_1,l_2}
\int_0^t (t-u)^{-\delta} \ee^{-\nu(t-u) \lambda_{l_1,l_2}} \dd w_{l_1,l_2}(u),
\end{equation*}
\begin{equation}\label{eq-Pf-lem***}
\int_0^t (t-u)^{-\delta} \widetilde Q^{\beta/2} S_{\nu(t-u)} \dd W_u^Q 
= \sum_{l_1,l_2 \ge 1} \mu_{l_1,l_2}^{-\alpha(\beta +1)/2} e_{l_1,l_2}
\int_0^t (t-u)^{-\delta} \ee^{-\nu(t-u) \lambda_{l_1,l_2}} \dd w_{l_1,l_2}(u)
\end{equation}
for $t \in \mbb T$ and obtain
\begin{align*}
\widetilde Q^{\beta/2} \int_0^t (t-u)^{-\delta} S_{\nu(t-u)} \dd W_u^Q 
&= \sum_{l_1,l_2 \ge 1} \mu_{l_1,l_2}^{-\alpha(\beta +1)/2} e_{l_1,l_2}
\int_0^t (t-u)^{-\delta} \ee^{-\nu(t-u) \lambda_{l_1,l_2}} \dd w_{l_1,l_2}(u) 
\\
&= \int_0^t (t-u)^{-\delta} \widetilde Q^{\beta/2} S_{\nu(t-u)} \dd W_u^Q 
\end{align*}
on $L^2(\Omega; \mcl H)$. By using \eqref{closed-stoch-integral}, 
the It\^{o} isometry (4.30) in \cite{DaPrato_Zabczyk2014} and \eqref{eq-lemA5-2}, 
\begin{align*}
\EE \Biggl[ \biggl \| 
\int_0^t S_{\nu(t-s)} \dd W_s^Q 
\biggr \|_{\mcl L_{\beta}^2}^2 \Biggr] 
&= \EE \Biggl[ \biggl\|
\widetilde Q^{\beta/2} \int_0^t S_{\nu (t-s)} \dd W_s^Q
\biggr\|^2 \Biggr]
\\
&= \EE \Biggl[ \biggl\|
\int_0^t \widetilde Q^{\beta/2} S_{\nu (t-s)} \dd W_s^Q
\biggr\|^2 \Biggr]
\\
&= \int_0^t \| \widetilde Q^{\beta/2} S_{\nu s} \|_{\mrm{HS}(\mcl U_0)}^2 \dd s.
\\
&= \int_0^t \| S_{\nu s} \|_{\mrm{HS}(\mcl U_\beta)}^2 \dd s.
\end{align*}
\end{proof}

\subsection{Proof of Theorem \ref{th2}}\label{sec4.2}
In this subsection, we show in three steps that the estimators 
$\hat \theta_{0,\beta}$ and $\hat \theta_{0,\beta}^{(\mrm{cont})}$ 
are asymptotically equivalent as $\nu \to 0$, $N \to \infty$, 
$M_1 \land M_2 \to \infty$ and $L \to \infty$.
Auxiliary results on the discretization error can be found 
in the end of this subsection.

\subsubsection{Control of temporal discretization error}
Consider a temporal discretization 
of the estimator $\hat \theta_{0,\beta}^{(\mrm{cont})}$
given in \eqref{est_cont}. 
That is, we construct an estimator of $\theta_0$ 
using discrete temporal data $\mbb X_{N}^{(\mrm{disc}\text{-}\mrm{cont})} 
= \{X_{t_i}(y,z)\}_{0 \le i \le N, (y,z)\in \overline D}$.
We use the property
\begin{equation*}
-\mcl A u = \lim_{h \downarrow 0} \frac{S_h u -u}{h},
\quad u \in \mathscr D(\mcl A)
\end{equation*}
and discretize the terms $\dd X_t$ and $\nu \mcl A X_t \dd t$ 
in the numerator of the estimator $\hat \theta_{0,\beta}^{(\mrm{cont})}$
with $X_{t_i} -X_{t_{i-1}}$ and $X_{t_{i-1}} -S_{\nu (t_i -t_{i-1})}X_{t_{i-1}}$
for $t \in [t_{i-1}, t_i)$, respectively.
We then define the estimator of $\theta_0$ based on 
$\mbb X_{N}^{(\mrm{disc}\text{-}\mrm{cont})}$ by
\begin{equation}\label{est_disc_cont}
\hat \theta_{0,\beta}^{(\mrm{disc}\text{-}\mrm{cont})} = 
\hat \theta_{0,\beta,\nu,N}^{(\mrm{disc}\text{-}\mrm{cont})} = 
\frac{\displaystyle 
\sum_{i=1}^N \langle X_{t_{i-1}}, 
X_{t_i} - S_{\nu (t_i-t_{i-1})} X_{t_{i-1}} \rangle_{\mcl L_{\beta}^2} 
}{\displaystyle \frac{1}{N}
\sum_{i=1}^N \| X_{t_{i-1}} \|_{\mcl L_{\beta}^2}^2}.
\end{equation}
Let $Y_{t,s} = X_t -S_{\nu(t-s)}X_s$, $s \le t$ 
and $Y_i = Y_{t_i,t_{i-1}}$, $i=1,\ldots, N$.
For $\beta >-1$, define
\begin{align*}
\widehat J_{\beta,\nu} 
= \widehat J_{\beta,\nu,N} 
&= \frac{1}{N}\sum_{i=1}^N \| X_{t_{i-1}} \|_{\mcl L_{\beta}^2}^2,
\quad
\widehat K_{\beta,\nu} 
= \widehat K_{\beta,\nu,N} 
= \sum_{i=1}^N \langle X_{t_{i-1}}, Y_i \rangle_{\mcl L_{\beta}^2},
\\
\widehat M_{\beta,\nu} 
= \widehat M_{\beta,\nu,N}
&= \frac{\theta_0^*}{\sigma^*} \sum_{i=1}^N \int_{t_{i-1}}^{t_i}
\langle X_{t_{i-1}}, 
S_{\nu(t_i-s)} X_s - X_{t_{i-1}} \rangle_{\mcl L_{\beta}^2} \dd s
+ \sum_{i=1}^N \int_{t_{i-1}}^{t_i} 
\langle X_{t_{i-1}}, S_{\nu(t_i-s)}  \dd W_s^Q \rangle_{\mcl L_{\beta}^2}.
\end{align*}

The difference between the estimators $\hat \theta_{0,\beta}^{(\mrm{cont})}$ 
and $\hat \theta_{0,\beta}^{(\mrm{disc}\text{-}\mrm{cont})}$ 
is controlled as follows.
\begin{prop}\label{propB1}
Let $\alpha >0$ and $-1 < \beta \le \frac{1}{\alpha} -1$.
Assume that [A1]$_{-1,2}$ holds.

\begin{enumerate}
\item[(1)]
It holds that $\hat \theta_{0,\beta}^{(\mrm{cont})} 
- \hat \theta_{0,\beta}^{(\mrm{disc}\text{-}\mrm{cont})} = \oo_{\PP}(1)$
as $\nu \to 0$ and $N \to \infty$. 

\item[(2)]
Under [C1]$_\beta$, it holds that 
$\mcl R_{\beta,\nu} 
(\hat \theta_{0,\beta}^{(\mrm{cont})} 
- \hat \theta_{0,\beta}^{(\mrm{disc}\text{-}\mrm{cont})})
=\oo_{\PP}(1)$ as $\nu \to 0$ and $N \to \infty$.
\end{enumerate}
\end{prop}

\begin{proof}
Since it follows that under $\PP_{\theta_0^*,\sigma^*}$,  
\begin{equation}\label{Pf_prop7-01}
X_t = S_{\nu(t-t_{i-1})}X_{t_{i-1}}
+ \theta_0^* \int_{t_{i-1}}^t S_{\nu(t-s)} X_s \dd s
+ \sigma^* \int_{t_{i-1}}^t S_{\nu(t-s)} \dd W_s^Q,
\end{equation}
we have  
\begin{equation*}
Y_i = \theta_0^* X_{t_{i-1}}(t_i-t_{i-1})
+ \theta_0^* \int_{t_{i-1}}^{t_i} (S_{\nu(t_i-s)} X_s - X_{t_{i-1}}) \dd s
+ \sigma^* \int_{t_{i-1}}^{t_i} S_{\nu(t_i-s)} \dd W_s^Q.
\end{equation*}
Therefore, we find that
$\widehat K_{\beta,\nu}
= \theta_0^* \widehat J_{\beta,\nu} + \sigma^* \widehat M_{\beta,\nu}$,
which together with $\hat \theta_{0,\beta}^{(\mrm{disc}\text{-}\mrm{cont})} 
= \widehat K_{\beta,\nu}/ \widehat J_{\beta,\nu}$ yields
\begin{equation*}
\hat \theta_{0,\beta}^{(\mrm{disc}\text{-}\mrm{cont})} - \theta_0^* 
= \sigma^* \frac{\widehat M_{\beta,\nu}}{\widehat J_{\beta,\nu}}.
\end{equation*}
Let
$\widehat \Delta_{\beta,\nu}^{(1)} = |\widehat J_{\beta,\nu} - J_{\beta,\nu}|$
and $\widehat \Delta_{\beta,\nu}^{(2)} = |\widehat M_{\beta,\nu} - M_{\beta,\nu}|$.
Notice that
\begin{align*}
|\hat \theta_{0,\beta}^{(\mrm{cont})} 
- \hat \theta_{0,\beta}^{(\mrm{disc}\text{-}\mrm{cont})}|
&= \sigma^*
\biggl| \frac{M_{\beta,\nu}}{J_{\beta,\nu}}
-\frac{\widehat M_{\beta,\nu}}{\widehat J_{\beta,\nu}} \biggr|
\le \sigma^* 
\Biggl(
\biggl| \frac{M_{\beta,\nu}}{J_{\beta,\nu}} \biggr|
\biggl|1-\frac{J_{\beta,\nu}}{\widehat J_{\beta,\nu}}\biggr|
+ \frac{\widehat \Delta_{\beta,\nu}^{(2)}}{\widehat J_{\beta,\nu}} 
\Biggr)
\\
&\le 
\frac{|\hat \theta_{0,\beta}^{(\mrm{cont})} - \theta_0^*| 
\widehat \Delta_{\beta,\nu}^{(1)} + \sigma^* \widehat \Delta_{\beta,\nu}^{(2)}
}{J_{\beta,\nu} - \widehat \Delta_{\beta,\nu}^{(1)}}.
\end{align*}
Since Theorem \ref{th1}, Propositions \ref{propA1} and \ref{propA2} show
\begin{equation*}
\hat \theta_{0,\beta}^{(\mrm{cont})} - \theta_0^* 
= \OO_{\PP}(\mcl R_{\beta,\nu}^{-1}),
\quad
\frac{\varphi_{\beta}(\nu)}{J_{\beta,\nu}} = \OO_\PP(1),
\end{equation*}
it is sufficient to verify
\begin{equation}\label{Pf_prop7-02}
\widehat \Delta_{\beta,\nu}^{(1)} = \oo_{\PP}(\varphi_{\beta}(\nu)),
\end{equation}
and show
\begin{equation}\label{Pf_prop7-03}
\widehat \Delta_{\beta,\nu}^{(2)} = \oo_{\PP}(\varphi_{\beta}(\nu))
\end{equation}
for the proof of (1) and
\begin{equation}\label{Pf_prop7-04}
\widehat \Delta_{\beta,\nu}^{(2)} = \oo_{\PP}(\varphi_{2\beta+1}(\nu)^{1/2})
\end{equation}
for (2), respectively.
Indeed, we can deduce from \eqref{Pf_prop7-02} that
\begin{equation*}
\frac{1}{J_{\beta,\nu} - \widehat \Delta_{\beta,\nu}^{(1)}} 
= \frac{\varphi_{\beta}(\nu)^{-1}}{J_{\beta,\nu}/\varphi_{\beta}(\nu) 
- \widehat \Delta_{\beta,\nu}^{(1)}/\varphi_{\beta}(\nu)}
= \OO_{\PP}(\varphi_{\beta}(\nu)^{-1}),
\end{equation*}
from \eqref{Pf_prop7-02} and \eqref{Pf_prop7-03} that
\begin{equation*}
|\hat \theta_{0,\beta}^{(\mrm{cont})} 
- \hat \theta_{0,\beta}^{(\mrm{disc}\text{-}\mrm{cont})}|
= \OO_{\PP}(\varphi_{\beta}(\nu)^{-1})
\Bigl( \OO_{\PP}(\mcl R_{\beta,\nu}^{-1})
\oo_{\PP}(\varphi_{\beta}(\nu))
+\oo_{\PP}(\varphi_{\beta}(\nu)) \Bigr)
= \oo_{\PP}(1),
\end{equation*}
and from \eqref{Pf_prop7-02}, \eqref{Pf_prop7-04} and 
$\mcl R_{\beta,\nu} \varphi_{\beta}(\nu)^{-1} \sim \varphi_{2\beta+1}(\nu)^{-1/2}$
that
\begin{align*}
\mcl R_{\beta,\nu} 
|\hat \theta_{0,\beta}^{(\mrm{cont})} 
- \hat \theta_{0,\beta}^{(\mrm{disc}\text{-}\mrm{cont})}|
&= \mcl R_{\beta,\nu} \OO_{\PP}(\varphi_{\beta}(\nu)^{-1})
\Bigl( \OO_{\PP}(\mcl R_{\beta,\nu}^{-1})
\oo_{\PP}(\varphi_{\beta}(\nu)) 
+\oo_{\PP}(\varphi_{2\beta +1}(\nu)^{1/2}) \Bigr)
\\
&= \oo_{\PP}(1).
\end{align*}

\textbf{Step 1:} We first show \eqref{Pf_prop7-02}.
For $0 < \gamma_{\beta} < \frac{\alpha(\beta+1)}{2}$, 
it follows from Lemmas \ref{lemB3} and \ref{lemB4} that
\begin{align*}
\EE\Bigl[\widehat \Delta_{\beta,\nu}^{(1)} \Bigr]
&\le
\sum_{i=1}^N \int_{t_{i-1}}^{t_i} 
\EE \Bigl[ \bigl| \| X_{t_{i-1}} \|_{\mcl L_{\beta}^2}^2 
- \| X_s \|_{\mcl L_{\beta}^2}^2 \bigr| \Bigr] \dd s
\\
&\le
\sum_{i=1}^N \int_{t_{i-1}}^{t_i}
\EE \Bigl[ 
\bigl( \| X_{t_{i-1}} \|_{\mcl L_{\beta}^2} 
+ \| X_s \|_{\mcl L_{\beta}^2} \bigr) 
\| X_{t_{i-1}} - X_s \|_{\mcl L_{\beta}^2}
\Bigr] \dd s
\\
&\lesssim
\sum_{i=1}^N \int_{t_{i-1}}^{t_i}
\EE \Bigl[ \| X_{t_{i-1}} \|_{\mcl L_{\beta}^2}^2 
+ \| X_s \|_{\mcl L_{\beta}^2}^2 \Bigr]^{1/2} 
\EE \Bigl[\| X_{t_{i-1}} - X_s \|_{\mcl L_{\beta}^2}^2 \Bigr]^{1/2} \dd s
\\
&\lesssim 
\nu^{(\alpha(\beta+1)-1)/2}
\biggl(
\frac{\nu^{\alpha(\beta+1)-1}}{N^{2\gamma_\beta \land 1}} + 
\Bigl(\frac{\nu}{N}\Bigr)^{2\gamma_\beta} + \frac{1}{N}
\biggr)^{1/2}
\\
&\lesssim 
\nu^{(\alpha(\beta+1)-1)/2}
\biggl(
\frac{\nu^{(\alpha(\beta+1)-1)/2}}{N^{\gamma_\beta \land 1/2}}
+ \frac{1}{\sqrt{N}}
\biggr).
\end{align*}
Since Proposition \ref{propA1} implies that for $0 < \alpha(\beta +1) \le 1$,
\begin{equation*}
\varphi_{\beta}(\nu)^{-1}
\nu^{(\alpha(\beta+1)-1)/2}
\biggl(
\frac{\nu^{(\alpha(\beta+1)-1)/2}}{N^{\gamma_\beta \land 1/2}}
+ \frac{1}{\sqrt{N}}
\biggr)
\lesssim 
\frac{1}{N^{\gamma_\beta \land 1/2}} +\frac{\nu^{(1 -\alpha(\beta+1))/2}}{\sqrt{N}}
\to 0,
\end{equation*}
we get $\EE[\widehat \Delta_{\beta,\nu}^{(1)}] 
= \oo(\varphi_\beta(\nu))$.

\textbf{Step 2:} We next estimate $\widehat \Delta_{\beta,\nu}^{(2)}$.
We have 
\begin{align*}
\widehat \Delta_{\beta,\nu}^{(2)}
&\lesssim 
\Biggl|
\sum_{i=1}^N \int_{t_{i-1}}^{t_i}
\langle X_{t_{i-1}}, 
S_{\nu(t_i-s)} X_s - X_{t_{i-1}} \rangle_{\mcl L_{\beta}^2} \dd s
\Biggr|
\\
&\qquad 
+ \Biggl|
\sum_{i=1}^N \int_{t_{i-1}}^{t_i} 
\langle S_{\nu(t_i-s)} X_{t_{i-1}} - X_s, \dd W_s^Q \rangle_{\mcl L_{\beta}^2} 
\Biggr|
\\
&=:
\widehat \Delta_{\beta,\nu}^{(2,1)} + \widehat \Delta_{\beta,\nu}^{(2,2)}.
\end{align*}
Note that the weighted Schwarz inequality: 
\begin{equation*}
| \langle u, v \rangle_{\mcl L_{\beta}^2} | 
\le \| u \|_{\mcl L_{\beta +\delta'}^2} \| v \|_{\mcl L_{\beta -\delta'}^2},
\quad
u \in \mcl L_{\beta +\delta'}^2 \cap \mcl L_{\beta}^2, 
\ v \in \mcl L_{\beta -\delta'}^2 \cap \mcl L_{\beta}^2
\ (\delta' \in \mbb R).
\end{equation*}
For any $0 \le \delta < \beta+1$ and 
$0< \gamma_\delta < \frac{\alpha(\beta+\delta+1)}{2}$,
it holds from \eqref{Pf_lem7-01}, Lemmas \ref{lemB3} and \ref{lemB4} that
\begin{align*}
\EE\Bigl[\bigl(\widehat \Delta_{\beta,\nu}^{(2,1)}\bigr)^2\Bigr]
&\le
\Biggl(
\sum_{i=1}^N \int_{t_{i-1}}^{t_i}
\EE\Bigl[ \| X_{t_{i-1}} \|_{\mcl L_{\beta-\delta}^2} 
\| S_{\nu(t_i-s)} X_s - X_{t_{i-1}} \|_{\mcl L_{\beta+\delta}^2} \Bigr] \dd s
\Biggr)^2
\\
&\lesssim 
\sum_{i=1}^N \int_{t_{i-1}}^{t_i}
\EE\Bigl[\| X_{t_{i-1}} \|_{\mcl L_{\beta-\delta}^2}^2\Bigr] \dd s
\\
&\qquad \times
\sum_{i=1}^N \int_{t_{i-1}}^{t_i}
\biggl(
\EE\Bigl[
\| S_{\nu(t_i-s)} X_s - S_{\nu(t_i-s)} X_{t_{i-1}} \|_{\mcl L_{\beta+\delta}^2}^2 
\Bigr] 
\\
&\qquad\qquad
+ \EE\Bigl[
\| S_{\nu(t_i-s)} X_{t_{i-1}} - X_{t_{i-1}} \|_{\mcl L_{\beta+\delta}^2}^2 
\Bigr]
\biggr) \dd s
\\
&\lesssim 
\sum_{i=1}^N \int_{t_{i-1}}^{t_i}
\EE\Bigl[\| X_{t_{i-1}} \|_{\mcl L_{\beta-\delta}^2}^2\Bigr] \dd s
\\
&\qquad \times
\sum_{i=1}^N \int_{t_{i-1}}^{t_i}
\biggl(\EE\Bigl[\| X_s - X_{t_{i-1}} \|_{\mcl L_{\beta+\delta}^2}^2 \Bigr] 
+ \Bigl( \frac{\nu}{N} \Bigr)^{2\gamma_\delta} 
\EE\Bigl[\| X_{t_{i-1}} \|_{\mcl L_{\beta+\delta-2\gamma_\delta/\alpha}^2}^2 \Bigr]
\biggr) \dd s
\\
&\lesssim 
\nu^{\alpha(\beta -\delta +1)-1}
\biggl(
\frac{\nu^{\alpha(\beta+\delta+1)-1}}{N^{2\gamma_\delta \land 1}} + 
\Bigl(\frac{\nu}{N}\Bigr)^{2\gamma_\delta} + \frac{1}{N}
+ \Bigl(\frac{\nu}{N}\Bigr)^{2\gamma_\delta}
(1+\nu^{\alpha(\beta+\delta+1)-2\gamma_\delta-1})
\biggr)
\\
&\lesssim 
\nu^{\alpha(\beta -\delta +1)-1}
\biggl(
\frac{\nu^{\alpha(\beta+\delta+1)-1}}{N^{2\gamma_\delta \land 1}} + 
\Bigl(\frac{\nu}{N}\Bigr)^{2\gamma_\delta} + \frac{1}{N}
\biggr)
=: \mcl S_{\beta,\delta,\gamma_\delta}(\nu,N).
\end{align*}
Similarly, it follows that for any $0 < \gamma_1 < \frac{\alpha(\beta+2)}{2}$,
\begin{align*}
\EE\Bigl[\bigl(\widehat \Delta_{\beta,\nu}^{(2,2)}\bigr)^2 \Bigr]
&=
\sum_{i=1}^N \int_{t_{i-1}}^{t_i}
\EE\Bigl[\| S_{\nu(t_i-s)} X_{t_{i-1}} - X_s \|_{\mcl L_{\beta+1}^2}^2\Bigr] \dd s
\\
&\lesssim 
\sum_{i=1}^N \int_{t_{i-1}}^{t_i}
\biggl( \EE\Bigl[
\| S_{\nu(t_i-s)} X_{t_{i-1}} - S_{\nu(t_i-s)} X_s \|_{\mcl L_{\beta+1}^2}^2
\Bigr] 
+ \EE\Bigl[\| S_{\nu(t_i-s)} X_s - X_s \|_{\mcl L_{\beta+1}^2}^2\Bigr]
\biggr)
\dd s
\\
&\lesssim 
\sum_{i=1}^N \int_{t_{i-1}}^{t_i}
\biggl(
\EE\Bigl[\| X_{t_{i-1}} - X_s \|_{\mcl L_{\beta+1}^2}^2\Bigr] 
+ \Bigl(\frac{\nu}{N}\Bigr)^{2\gamma_1} 
\EE\Bigl[\| X_s \|_{\mcl L_{\beta+1-2\gamma_1/\alpha}^2}^2\Bigr]
\biggr) \dd s
\\
&\lesssim 
\biggl(
\frac{\nu^{\alpha(\beta+2)-1}}{N^{2\gamma_1 \land 1}} + 
\Bigl(\frac{\nu}{N}\Bigr)^{2\gamma_1} + \frac{1}{N}
\biggr)
+ \Bigl(\frac{\nu}{N}\Bigr)^{2\gamma_1}
(1 +\nu^{\alpha(\beta+2)-2\gamma_1-1})
\\
&\lesssim 
\frac{\nu^{\alpha(\beta+2)-1}}{N^{2\gamma_1 \land 1}} + 
\Bigl(\frac{\nu}{N}\Bigr)^{2\gamma_1} + \frac{1}{N}
=: \mcl T_{\beta,\gamma_1}(\nu,N).
\end{align*}
Therefore, we see that
for any $0 \le \delta < \beta+1$, 
$0< \gamma_\delta < \frac{\alpha(\beta+\delta+1)}{2}$
and $0 < \gamma_1 < \frac{\alpha(\beta+2)}{2}$,
\begin{equation*}
\EE\Bigl[\bigl(\widehat \Delta_{\beta,\nu}^{(2)}\bigr)^2\Bigr] 
\lesssim \mcl S_{\beta,\delta,\gamma_\delta}(\nu,N) 
+\mcl T_{\beta,\gamma_1}(\nu,N).
\end{equation*}

\textbf{Step 3:} 
In order to show \eqref{Pf_prop7-03}, 
we prove that there exist $(\delta, \gamma_\delta, \gamma_1)$ such that
\begin{equation*}
\varphi_{\beta}(\nu)^{-2} 
(\mcl S_{\beta,\delta,\gamma_\delta}(\nu,N) 
+ \mcl T_{\beta,\gamma_1}(\nu,N))
\to 0.
\end{equation*}
By choosing $\delta = 0$, $\gamma_\delta > 0$ and $\gamma_1 >0$, 
it follows from Proposition \ref{propA1}, $1 -\alpha(\beta+1) \ge 0$ 
and $1-\alpha \beta >0$ that
\begin{align*}
\varphi_{\beta}(\nu)^{-2} \mcl S_{\beta,\delta,\gamma_\delta}(\nu,N)
&\lesssim
\frac{1}{N^{2\gamma_\delta \land 1}} 
+ \frac{\nu^{1-\alpha(\beta +1) +2\gamma_\delta}}
{N^{2\gamma_\delta}}
+ \frac{\nu^{1-\alpha(\beta +1)}}{N}
\to 0,
\\
\varphi_{\beta}(\nu)^{-2}
\mcl T_{\beta,\gamma_1}(\nu,N)
&\lesssim 
\frac{\nu^{1-\alpha \beta}}{N^{2\gamma_1 \land 1}} + 
\frac{\nu^{2(1-\alpha(\beta+1)+\gamma_1)}}{N^{2\gamma_1}} 
+ \frac{\nu^{2(1-\alpha(\beta+1))}}{N}
\to 0.
\end{align*}

\textbf{Step 4:} 
To complete the proof of \eqref{Pf_prop7-04}, we verify that under [C1]$_\beta$, 
\begin{equation*}
\varphi_{2\beta+1}(\nu)^{-1} 
(\mcl S_{\beta,\delta,\gamma_\delta}(\nu,N) 
+ \mcl T_{\beta,\gamma_1}(\nu,N)) \to 0 
\end{equation*}
for some $(\delta,\gamma_\delta,\gamma_1)$.

(i) Consider the case that $\alpha(\beta +1) \le 1/2$. 
By noting that $2\gamma_\delta < \alpha(\beta +\delta+1) < 1$ 
and $\gamma_1 > 0 \ge \alpha(\beta+1) -1/2$,
it follows from Proposition \ref{propA1} that
\begin{align*}
\varphi_{2\beta +1}(\nu)^{-1} \mcl S_{\beta,\delta,\gamma_\delta}(\nu,N)
&\lesssim
\nu^{-\alpha(\beta +\delta +1)}
\biggl(
\frac{\nu^{\alpha(\beta+\delta+1)-1}}{N^{2\gamma_\delta \land 1}} + 
\Bigl(\frac{\nu}{N}\Bigr)^{2\gamma_\delta} + \frac{1}{N}
\biggr)
\\
&\lesssim
\frac{1}{\nu N^{2\gamma_\delta}} + \frac{1}{\nu^{\alpha(\beta +\delta +1)} N},
\\
\varphi_{2\beta+1}(\nu)^{-1}
\mcl T_{\beta,\gamma_1}(\nu,N)
&\lesssim 
\frac{\nu^{-\alpha \beta}}{N^{2\gamma_1 \land 1}} + 
\frac{\nu^{1-2\alpha(\beta+1)+2\gamma_1}}{N^{2\gamma_1}} 
+ \frac{\nu^{1-2\alpha(\beta+1)}}{N}
\\
&=
\begin{cases}
\oo(1), & \beta \le 0,
\\
\dfrac{1}{\nu^{\alpha \beta}N^{2\gamma_1 \land 1}} + \oo(1), & \beta > 0.
\end{cases}
\end{align*}
Hence, we prove that
under $\nu N^q \to \infty$ for some 
$q < \rho_{\alpha(\beta +1)}^{(\mrm{time})} = 2\alpha(\beta+1)$, 
there exist $0 \le \delta < \beta +1$ and  
$0 < \gamma_\delta < \frac{\alpha(\beta +\delta +1)}{2}$ 
such that
\begin{equation}\label{Pf_prop7-05}
\nu N^{2\gamma_\delta \land \frac{1}{\alpha(\beta +\delta +1)}} \to \infty
\end{equation}
and additionally show that there exists
$0< \gamma_1 < \frac{\alpha(\beta +2)}{2}$ such that
\begin{equation}\label{Pf_prop7-06}
\nu N^{\frac{2\gamma_1 \land 1}{\alpha \beta}} \to \infty
\end{equation}
if $\beta > 0$.

Note that for a function $f: \mbb R^d \to \mbb R$ and 
a domain $E \subset \mbb R^d$,
if $\sup_{x \in E} f(x)$ exists 
and $\nu N^q \to \infty$ for some $q < \sup_{x \in E} f(x)$, 
then $\nu N^{f(x)} \to \infty$ for some $x \in E$.

Because  $2\gamma_\delta < \alpha(\beta +\delta +1) < 1$, 
$x \land \frac{1}{x} = x$ ($0 < x <1$) and $\delta < \beta +1$, we find that
\begin{align*}
\sup_{\delta} \sup_{\gamma_\delta}
\biggl( 2\gamma_\delta \land \frac{1}{\alpha(\beta +\delta +1)} \biggr)
&= \sup_{\delta} 
\biggl( \alpha(\beta +\delta +1) \land \frac{1}{\alpha(\beta +\delta +1)} \biggr)
\\
&= \sup_{\delta} \alpha(\beta +\delta +1) 
= 2 \alpha(\beta +1),
\end{align*}
\begin{equation*}
\sup_{\gamma_1} \frac{2 \gamma_1 \land 1}{\alpha \beta}
= \biggl(1 +\frac{2}{\beta} \biggr) \land \frac{1}{\alpha \beta} 
> 1 \ge 2\alpha(\beta +1),
\end{equation*}
which imply that there exist $(\delta,\gamma_\delta, \gamma_1)$ 
such that \eqref{Pf_prop7-05} and \eqref{Pf_prop7-06}
under $\nu N^q \to \infty$ for some $q < 2\alpha(\beta+1)$.

(ii) Consider the case that $1/2 < \alpha(\beta +1) \le 1$. 
By choosing $\delta > \frac{1}{\alpha} -(\beta +1)$ and
$\gamma_\delta < \frac{\alpha(\beta +\delta +1)}{2}$, it follows 
from Proposition \ref{propA1} and $2\gamma_\delta > 1$ that
\begin{align*}
\varphi_{2\beta +1}(\nu)^{-1} \mcl S_{\beta,\delta,\gamma_\delta}(\nu,N)
&\lesssim
\nu^{\alpha(\beta -\delta +1)-1}
\biggl(
\frac{\nu^{\alpha(\beta+\delta+1)-1}}{N^{2\gamma_\delta \land 1}} + 
\Bigl(\frac{\nu}{N}\Bigr)^{2\gamma_\delta} + \frac{1}{N}
\biggr)
\\
&\lesssim
\frac{1}{\nu^{1-\alpha(\beta -\delta +1)} N}.
\end{align*}
Furthermore, it holds that for $\gamma_1 > 1/2$, 
\begin{equation*}
\varphi_{2\beta +1}(\nu)^{-1}
\mcl T_{\beta,\gamma_1}(\nu,N)
\lesssim \frac{\nu^{\alpha(\beta+2)-1}}{N^{2\gamma_1 \land 1}} + 
\Bigl(\frac{\nu}{N}\Bigr)^{2\gamma_1} + \frac{1}{N}
= \frac{1}{\nu^{1-\alpha(\beta+2)} N} + \oo(1).
\end{equation*}
Therefore, we see from
\begin{equation*}
\sup_{\delta} \frac{1}{1-\alpha(\beta -\delta +1)} 
= \frac{1}{2(1-\alpha(\beta +1))},
\end{equation*}
\begin{equation*}
\frac{1}{1-\alpha(\beta +2)} -\frac{1}{2(1-\alpha(\beta +1))}
= \frac{1 -\alpha \beta}{2(1-\alpha(\beta +1))(1-\alpha(\beta +2))} > 0
\end{equation*}
that there exists 
$\delta > \frac{1}{\alpha} -(\beta +1)$ such that
\begin{equation*}
\nu N^{\frac{1}{1-\alpha(\beta -\delta +1)} \land \frac{1}{1-\alpha(\beta +2)}} 
\to \infty
\end{equation*}
under $\nu N^q \to \infty$ for some 
$q < \rho_{\alpha(\beta +1)}^{(\mrm{time})} = \frac{1}{2(1-\alpha(\beta +1))}$.
\end{proof}

\subsubsection{Control of spatial discretization error}
Consider a spatial discretization of the estimator
$\hat \theta_{0,\beta}^{(\mrm{disc}\text{-}\mrm{cont})}$ given in \eqref{est_disc_cont}.
We set the  estimator of $\theta_0$ 
based on the discrete spatio-temporal data $\mbb X_{N,M}$ as follows. 
\begin{equation}\label{est_disc}
\hat \theta_{0,\beta}^{(\mrm{disc})} = 
\hat \theta_{0,\beta,\nu,N,M}^{(\mrm{disc})} = 
\frac{\displaystyle 
\sum_{i=1}^N \langle \Psi_M X_{t_{i-1}}, 
\Psi_M X_{t_i} - S_{\nu (t_i-t_{i-1})} \Psi_M X_{t_{i-1}} \rangle_{\mcl L_{\beta}^2} 
}{\displaystyle 
\frac{1}{N}\sum_{i=1}^N \| \Psi_M X_{t_{i-1}} \|_{\mcl L_{\beta}^2}^2},
\end{equation}
where $\Psi_M$ is the operator such that
\begin{equation*}
\Psi_M f(y,z) = f(y_{j-1},z_{k-1}), 
\quad (y,z) \in [y_{j-1},y_j) \times [z_{k-1},z_k)
\end{equation*}
for $j=1,\ldots,M_1$, $k=1,\ldots,M_2$ and a continuous function $f$ on $D$.

For $\beta >-1$, define
$Y_i^{\Psi_M} = \Psi_M X_{t_i} - S_{\nu (t_i-t_{i-1})} \Psi_M X_{t_{i-1}}$,
\begin{equation*}
\widetilde J_{\beta,\nu}
= \widetilde J_{\beta,\nu,N,M}
= \frac{1}{N}\sum_{i=1}^N \| \Psi_M X_{t_{i-1}} \|_{\mcl L_{\beta}^2}^2,
\quad
\widetilde K_{\beta,\nu}
= \widetilde K_{\beta,\nu,N,M}
= \sum_{i=1}^N \langle \Psi_M X_{t_{i-1}}, Y_i^{\Psi_M} \rangle_{\mcl L_{\beta}^2}.
\end{equation*}
Let $M = (M_1 \land M_2)^2$.
The following proposition provides 
the asymptotic equivalences of 
the estimators $\hat \theta_{0,\beta}^{(\mrm{disc}\text{-}\mrm{cont})}$ and 
$\hat \theta_{0,\beta}^{(\mrm{disc})}$.
\begin{prop}\label{propB2}
Let $\alpha >0$, $-1 < \beta \le \frac{1}{\alpha} -1$ 
and $p > \frac{4}{\alpha(\beta +1)}$. 
Assume that [A1]$_{-1,p}$ holds.
\begin{enumerate}
\item[(1)]
Under [B1]$_{\beta,p}$, it holds that 
$\hat \theta_{0,\beta}^{(\mrm{disc}\text{-}\mrm{cont})} 
- \hat \theta_{0,\beta}^{(\mrm{disc})}  = \oo_{\PP}(1)$
as $\nu \to 0$, $N \to \infty$ and $M \to \infty$.

\item[(2)]
Under [C2]$_{\beta,p}$, it holds that
$\mcl R_{\beta,\nu} 
(\hat \theta_{0,\beta}^{(\mrm{disc}\text{-}\mrm{cont})} 
- \hat \theta_{0,\beta}^{(\mrm{disc})}) = \oo_{\PP}(1)$
as $\nu \to 0$, $N \to \infty$ and $M \to \infty$. 
\end{enumerate}
\end{prop}

\begin{proof}
Recall
$\hat \theta_{0,\beta}^{(\mrm{disc}\text{-}\mrm{cont})} 
= \widehat K_{\beta,\nu}/\widehat J_{\beta,\nu}$.
Setting
$\widetilde \Delta_{\beta,\nu}^{(1)} 
= |\widetilde J_{\beta,\nu} - \widehat J_{\beta,\nu}|$
and
$\widetilde \Delta_{\beta,\nu}^{(2)} 
= |\widetilde K_{\beta,\nu} - \widehat K_{\beta,\nu}|$,
we have
\begin{align*}
|\hat \theta_{0,\beta}^{(\mrm{disc}\text{-}\mrm{cont})} - \hat \theta_{0,\beta}^{(\mrm{disc})}|
&=
\biggl|
\hat \theta_{0,\beta}^{(\mrm{disc}\text{-}\mrm{cont})} - 
\frac{\widehat K_{\beta,\nu}}{\widetilde J_{\beta,\nu}}
-\frac{\widetilde K_{\beta,\nu} - \widehat K_{\beta,\nu}}{\widetilde J_{\beta,\nu}}
\biggr|
\le 
|\hat \theta_{0,\beta}^{(\mrm{disc}\text{-}\mrm{cont})}|
\biggl| 1 - \frac{\widehat J_{\beta,\nu}}{\widetilde J_{\beta,\nu}} \biggr|
+\frac{\widetilde \Delta_{\beta,\nu}^{(2)}}{\widetilde J_{\beta,\nu}}
\\
&\le
\frac{|\hat \theta_{0,\beta}^{(\mrm{disc}\text{-}\mrm{cont})}| \widetilde \Delta_{\beta,\nu}^{(1)}
+\widetilde \Delta_{\beta,\nu}^{(2)}}
{\widehat J_{\beta,\nu} -\widetilde \Delta_{\beta,\nu}^{(1)}}.
\end{align*}
Because it follows from Proposition \ref{propB1} and \eqref{Pf_prop7-02} that
\begin{equation}\label{Pf_prop8-01}
\hat \theta_{0,\beta}^{(\mrm{disc}\text{-}\mrm{cont})} = \OO_{\PP}(1), 
\quad
\frac{\varphi_{\beta}(\nu)}{\widehat J_{\beta,\nu}} = \OO_\PP(1),
\end{equation}
we show that
\begin{equation}\label{Pf_prop8-02}
\widetilde \Delta_{\beta,\nu}^{(1)}, \widetilde \Delta_{\beta,\nu}^{(2)}
= \oo_{\PP}(\varphi_{\beta}(\nu))
\end{equation}
under [B1]$_{\beta,p}$ for the proof of (1), and
\begin{equation*}
\widetilde \Delta_{\beta,\nu}^{(1)}, \widetilde \Delta_{\beta,\nu}^{(2)}
= \oo_{\PP}(\varphi_{2\beta+1}(\nu)^{1/2}) 
\end{equation*}
under [C2]$_{\beta,p}$ for the proof of (2).

\textbf{Step 1:} 
We estimate $\widetilde \Delta_{\beta,\nu}^{(1)}$ and
$\widetilde \Delta_{\beta,\nu}^{(2)}$.
Note that for any $\delta_1, \delta_2, \delta_3 \in \mbb R$,
\begin{align*}
| \langle \Psi_M f, \Psi_M g \rangle_{\mcl L_{\beta}^2} 
-\langle f, g \rangle_{\mcl L_{\beta}^2} |
&\le 
\| \Psi_M f -f \|_{\mcl L_{\beta -\delta_1}^2} 
\| \Psi_M g -g \|_{\mcl L_{\beta +\delta_1}^2} 
+ \| f \|_{\mcl L_{\beta -\delta_2}^2} 
\| \Psi_M g -g \|_{\mcl L_{\beta +\delta_2}^2} 
\\
&\qquad
+ \| \Psi_M f -f \|_{\mcl L_{\beta -\delta_3}^2} 
\| g \|_{\mcl L_{\beta +\delta_3}^2}.
\end{align*}
It holds from Lemmas \ref{lemB3}, \ref{lemB5} and \ref{lemB6}
that for any $0 \le \delta_1 < \beta +1 -\frac{4}{\alpha p}$, 
$0 \le \delta_2, \delta_3 < \beta +1$, 
$\frac{1}{p} < \gamma_{\delta_j} < (\frac{\alpha(\beta +\delta_j +1)}{2} -\frac{1}{p}) 
\land \frac{1}{2}$ and 
$\frac{1}{p} < \gamma_{-\delta_1} < \frac{\alpha(\beta -\delta_1 +1)}{2} -\frac{1}{p}$,
\begin{align*}
&\EE\Bigl[ \bigl| \| X_{t_{i-1}} \|_{\mcl L_{\beta}^2}^2 
-\| \Psi_M X_{t_{i-1}} \|_{\mcl L_{\beta}^2}^2 \bigr| \Bigr]
\\
&\le
\EE\Bigl[ \| \Psi_M X_{t_{i-1}} -X_{t_{i-1}} 
\|_{\mcl L_{\beta -\delta_1}^2}^2 \Bigr]^{1/2}
\EE\Bigl[ \| \Psi_M X_{t_{i-1}} -X_{t_{i-1}} 
\|_{\mcl L_{\beta +\delta_1}^2}^2 \Bigr]^{1/2}
\\
&\qquad+ 2 \EE\Bigl[\| X_{t_{i-1}} \|_{\mcl L_{\beta -\delta_2}^2}^2 \Bigr]^{1/2} 
\EE\Bigl[\| \Psi_M X_{t_{i-1}} -X_{t_{i-1}} 
\|_{\mcl L_{\beta +\delta_2}^2}^2\Bigr]^{1/2}
\\
&\lesssim
\frac{\varphi_{\beta -\delta_1 -2\gamma_{-\delta_1}/\alpha}(\nu)^{1/2}
\varphi_{\beta +\delta_1 -2\gamma_{\delta_1}/\alpha}(\nu)^{1/2}}
{M^{\gamma_{\delta_1} +\gamma_{-\delta_1} -2/p}}
\\
&\qquad+ 
\frac{\nu^{(\alpha(\beta -\delta_2 +1)-1)/2}
\varphi_{\beta +\delta_2 -2\gamma_{\delta_2}/\alpha}(\nu)^{1/2}}
{M^{\gamma_{\delta_2} -1/p}}
\\
&=
\frac{\nu^{(\alpha(\beta -\delta_1 +1)-2\gamma_{-\delta_1}-1)/2}
\varphi_{\beta +\delta_1 -2\gamma_{\delta_1}/\alpha}(\nu)^{1/2}}
{M^{\gamma_{\delta_1} +\gamma_{-\delta_1} -2/p}}
\\
&\qquad+ 
\frac{\nu^{(\alpha(\beta -\delta_2 +1)-1)/2}
\varphi_{\beta +\delta_2 -2\gamma_{\delta_2}/\alpha}(\nu)^{1/2}}
{M^{\gamma_{\delta_2} -1/p}},
\end{align*}
\begin{align*}
&\EE\Bigl[ \bigl| \langle \Psi_M X_{t_{i-1}}, 
Y_i^{\Psi_M} \rangle_{\mcl L_{\beta}^2}
- \langle X_{t_{i-1}}, Y_i \rangle_{\mcl L_{\beta}^2} \bigr| \Bigr]
\\
&\le 
\EE \Bigl[\| \Psi_M X_{t_{i-1}} -X_{t_{i-1}} 
\|_{\mcl L_{\beta -\delta_1}^2}^2 \Bigr]^{1/2} 
\EE \Bigl[\| Y_i^{\Psi_M} -Y_i \|_{\mcl L_{\beta +\delta_1}^2}^2 \Bigr]^{1/2} 
\\
&\qquad+ \EE \Bigl[\| X_{t_{i-1}} \|_{\mcl L_{\beta-\delta_2}^2}^2 \Bigr]^{1/2} 
\EE \Bigl[\| Y_i^{\Psi_M} -Y_i \|_{\mcl L_{\beta +\delta_2}^2}^2 \Bigr]^{1/2} 
\\
&\qquad+ \EE \Bigl[\| \Psi_M X_{t_{i-1}} -X_{t_{i-1}} 
\|_{\mcl L_{\beta +\delta_3}^2}^2 \Bigr]^{1/2} 
\EE \Bigl[\| Y_i \|_{\mcl L_{\beta -\delta_3}^2}^2 \Bigr]^{1/2} 
\\
&\lesssim
\frac{\nu^{(\alpha(\beta -\delta_1 +1)-2\gamma_{-\delta_1}-1)/2}
\varphi_{\beta +\delta_1 -2\gamma_{\delta_1}/\alpha}(\nu)^{1/2}}
{M^{\gamma_{\delta_1} +\gamma_{-\delta_1} -2/p} 
N^{\alpha(\beta +\delta_1 +1)/2 -\gamma_{\delta_1}}}
\\
&\qquad+ 
\frac{\nu^{(\alpha(\beta -\delta_2 +1)-1)/2}
\varphi_{\beta +\delta_2 -2\gamma_{\delta_2}/\alpha}(\nu)^{1/2}}
{M^{\gamma_{\delta_2} -1/p} N^{\alpha(\beta +\delta_2 +1)/2 -\gamma_{\delta_2}}}
\\
&\qquad+ \frac{\varphi_{\beta +\delta_3 -2\gamma_{\delta_3}/\alpha}(\nu)^{1/2} 
\nu^{(\alpha(\beta -\delta_3 +1) -1)/2}}
{M^{\gamma_{\delta_3} -1/p} N^{\alpha(\beta -\delta_3 +1)/2}}.
\end{align*}
Therefore, for $i=1,2$, we have
\begin{align*}
\EE \Bigl[ \widetilde \Delta_{\beta,\nu}^{(i)} \Bigr]
&\lesssim 
\frac{\nu^{(\alpha(\beta -\delta_1 +1)-2\gamma_{-\delta_1}-1)/2}
\varphi_{\beta +\delta_1 -2\gamma_{\delta_1}/\alpha}(\nu)^{1/2}}
{M^{\gamma_{\delta_1} +\gamma_{-\delta_1} -2/p} 
N^{\alpha(\beta +\delta_1 +1)/2 -\gamma_{\delta_1} -1}}
\\
&\qquad+ 
\frac{\nu^{(\alpha(\beta -\delta_2 +1)-1)/2}
\varphi_{\beta +\delta_2 -2\gamma_{\delta_2}/\alpha}(\nu)^{1/2}}
{M^{\gamma_{\delta_2} -1/p} N^{\alpha(\beta +\delta_2 +1)/2 -\gamma_{\delta_2} -1}}
\\
&\qquad+ \frac{\varphi_{\beta +\delta_3 -2\gamma_{\delta_3}/\alpha}(\nu)^{1/2} 
\nu^{(\alpha(\beta -\delta_3 +1) -1)/2}}
{M^{\gamma_{\delta_3} -1/p} N^{\alpha(\beta -\delta_3 +1)/2 -1}}
\\
&=: 
\mcl U_{\beta,\delta_1,\delta_2,\delta_3,
\gamma_{\delta_1},\gamma_{-\delta_1},\gamma_{\delta_2}, \gamma_{\delta_3}}
(\nu,N,M).
\end{align*}

\textbf{Step 2:}
In order to conclude this proof, we show that 
for some $\delta_j$, $\gamma_{\delta_j}$ ($j=1,2,3$) 
and $\gamma_{-\delta_1}$,
\begin{equation*}
\varphi_\beta(\nu)^{-1}
\mcl U_{\beta,\delta_1,\delta_2,\delta_3,
\gamma_{\delta_1},\gamma_{-\delta_1},\gamma_{\delta_2}, \gamma_{\delta_3}}
(\nu,N,M)
\to 0
\end{equation*}
under [B1]$_{\beta,p}$ and 
\begin{equation*}
\varphi_{2\beta+1}(\nu)^{-1/2}
\mcl U_{\beta,\delta_1,\delta_2,\delta_3,
\gamma_{\delta_1},\gamma_{-\delta_1},\gamma_{\delta_2}, \gamma_{\delta_3}}
(\nu,N,M)
\to 0
\end{equation*}
under [C2]$_{\beta,p}$, respectively.

Recall 
\begin{equation*}
0 \le \delta_1 < \beta +1 -\frac{4}{\alpha p},
\quad 
0 \le \delta_2, \delta_3 < \beta +1,
\end{equation*}
\begin{equation*}
\frac{1}{p} < \gamma_{-\delta_1} 
< \frac{\alpha(\beta -\delta_1 +1)}{2} -\frac{1}{p}, 
\quad
\frac{1}{p} < \gamma_{\delta_j}
< \biggl( \frac{\alpha(\beta +\delta_j +1)}{2} -\frac{1}{p} \biggr) \land \frac{1}{2} 
\quad (j=1,2,3).
\end{equation*}

(i) Consider the case that $\alpha(\beta +1) \le 1/2$.
Since
$\varphi_{\beta +\delta_j -2\gamma_{\delta_j}/\alpha}(\nu) \sim 
\nu^{\alpha(\beta +\delta_j +1)-2\gamma_{\delta_j} -1}$
($j=1,2,3$) and 
\begin{align*}
&\mcl U_{\beta,\delta_1,\delta_2,\delta_3,
\gamma_{\delta_1},\gamma_{-\delta_1},\gamma_{\delta_2}, \gamma_{\delta_3}}
(\nu,N,M)
\\
&\sim
\nu^{\alpha(\beta +1)-1}
\biggl(
\frac{1}{\nu^{\gamma_{\delta_1} +\gamma_{-\delta_1}}
M^{\gamma_{\delta_1} +\gamma_{-\delta_1} -2/p} 
N^{\alpha(\beta +\delta_1 +1)/2 -\gamma_{\delta_1} -1}}
\\
&\qquad +\frac{1}
{\nu^{\gamma_{\delta_2}}
M^{\gamma_{\delta_2} -1/p} N^{\alpha(\beta +\delta_2 +1)/2 -\gamma_{\delta_2} -1}}
+\frac{1}
{\nu^{\gamma_{\delta_3}}M^{\gamma_{\delta_3} -1/p} N^{\alpha(\beta -\delta_3 +1)/2 -1}}
\biggr),
\end{align*}
it suffices to verify that for $c =0, 1/2$,
there exist $\delta_j$, $\gamma_{\delta_j}$ ($j=1,2,3$) 
and $\gamma_{-\delta_1}$ such that
\begin{equation}\label{Pf_prop8-04}
\begin{cases}
\nu^{\gamma_{\delta_1} +\gamma_{-\delta_1} +c}
M^{\gamma_{\delta_1} +\gamma_{-\delta_1} -2/p} 
N^{\alpha(\beta +\delta_1 +1)/2 -\gamma_{\delta_1} -1} 
\to \infty,
\\
\nu^{\gamma_{\delta_2} +c}
M^{\gamma_{\delta_2} -1/p} N^{\alpha(\beta +\delta_2 +1)/2 -\gamma_{\delta_2} -1}
\to \infty,
\\
\nu^{\gamma_{\delta_3} +c}
M^{\gamma_{\delta_3} -1/p} N^{\alpha(\beta -\delta_3 +1)/2 -1}
\to \infty
\end{cases}
\end{equation}
under [B1]$_{\beta,p}$ and [C2]$_{\beta,p}$, respectively.

Let $M = N^b$, $b \ge 1$.
Notice that 
$\frac{1}{p} < \gamma_{\delta_j} < \frac{\alpha(\beta +\delta_j +1)}{2} -\frac{1}{p}$ 
($j=1,2,3$).
Since
\begin{align*}
&\sup_{\delta_1} \sup_{\gamma_{\delta_1}} \sup_{\gamma_{-\delta_1}} 
\frac{b(\gamma_{\delta_1} +\gamma_{-\delta_1} -2/p)
+\alpha(\beta +\delta_1 +1)/2 -\gamma_{\delta_1} -1}
{\gamma_{\delta_1} +\gamma_{-\delta_1} +c}
\\
&=\sup_{\delta_1} \sup_{\gamma_{\delta_1}} 
\biggl(b +
\frac{-b(c +2/p)+\alpha(\beta +\delta_1 +1)/2 -\gamma_{\delta_1} -1}
{\gamma_{\delta_1} +\alpha(\beta -\delta_1 +1)/2 -1/p +c}
\biggr)
\\
&= b \biggl(1 -\frac{c +2/p}{\alpha(\beta +1) -2/p +c} \biggr)
-\frac{1 -1/p}{\alpha(\beta +1) -2/p +c}
=:\Gamma_{\alpha(\beta +1),p,c}^{(1)},
\end{align*}
\begin{align*}
&\sup_{\delta_2} \sup_{\gamma_{\delta_2}} 
\frac{b(\gamma_{\delta_2} -1/p)
+\alpha(\beta +\delta_2 +1)/2 -\gamma_{\delta_2} -1}
{\gamma_{\delta_2} +c}
\\
&=\sup_{\delta_2}
\biggl(b -1
+\frac{-b(c +1/p) +c +\alpha(\beta +\delta_2 +1)/2 -1}
{\alpha(\beta +\delta_2 +1)/2 -1/p +c}
\biggr)
\\
&= b \biggl(1 -\frac{c +1/p}{\alpha(\beta +1) -1/p +c} \biggr)
-\frac{1 -1/p}{\alpha(\beta +1) -1/p +c}
=: \widetilde \Gamma_{\alpha(\beta +1),p,c}^{(2)},
\end{align*}
and
\begin{align*}
&\sup_{\delta_3} \sup_{\gamma_{\delta_3}} 
\frac{b(\gamma_{\delta_3} -1/p)
+\alpha(\beta -\delta_3 +1)/2 -1}
{\gamma_{\delta_3} +c}
\\
&=\sup_{\delta_3}
\biggl(b +\frac{-b(c +1/p) +\alpha(\beta -\delta_3 +1)/2 -1}
{\alpha(\beta +\delta_3 +1)/2 -1/p +c}
\biggr)
\\
&= b \biggl(1 -\frac{c +1/p}{\alpha(\beta +1) -1/p +c} \biggr)
-\frac{1}{\alpha(\beta +1) -1/p +c}
=: \Gamma_{\alpha(\beta +1),p,c}^{(2)}
\le \widetilde \Gamma_{\alpha(\beta +1),p,c}^{(2)},
\end{align*}
we have \eqref{Pf_prop8-04} under $\nu N^{u} \to \infty$ for some 
\begin{equation*}
u < \Gamma_{x,p,c}^{(1)} \land \Gamma_{x,p,c}^{(2)} =
\begin{cases}
(b \tau_{1,x,p}^{(\mrm{space})} -\phi_{1,x,p}^{(\mrm{space})})
\land 
(b \tau_{2,x,p}^{(\mrm{space})} -\phi_{2,x,p}^{(\mrm{space})}),
& c=0,
\\
(b \rho_{1,x,p}^{(\mrm{space})} -\psi_{1,x,p}^{(\mrm{space})})
\land 
(b \rho_{2,x,p}^{(\mrm{space})} -\psi_{2,x,p}^{(\mrm{space})}),
& c=1/2,
\end{cases}
\end{equation*}
where $x =\alpha(\beta +1) \in (0,1/2]$.

(ii) Consider the case that $1/2 < \alpha(\beta +1) \le 1$. 
By choosing $\gamma_{\delta_j} > \frac{\alpha(\beta +\delta_j +1)-1}{2}$ ($j=1,2,3$), 
we have $\varphi_{\beta +\delta_j -2\gamma_{\delta_j}/\alpha}(\nu) \sim 
\nu^{\alpha(\beta +\delta_j +1)-2\gamma_{\delta_j} -1}$.
Hence, it is enough to show that for $c =0, 1/2$,
there exist $\delta_j$, $\gamma_{\delta_j}$ ($j=1,2,3$) 
and $\gamma_{-\delta_1}$ such that \eqref{Pf_prop8-04}
under [B1]$_{\beta,p}$ and [C2]$_{\beta,p}$, respectively.

Let $M = N^b$, $b \ge 1$ and 
$E = \{ \delta | \delta > \frac{1}{\alpha}(1 +\frac{2}{p}) -(\beta +1) \}$.
Note that
\begin{equation*}
\frac{1}{p} <\gamma_\delta < 
\begin{cases}
\frac{\alpha(\beta +\delta +1)}{2} -\frac{1}{p}, & \delta \in E^{\mrm c},
\\
\frac{1}{2}, & \delta \in E.
\end{cases}
\end{equation*}
It then follows 
from $\delta < \frac{2\gamma_\delta +1}{\alpha} -(\beta +1)$ ($\delta \in E$) that
\begin{align*}
&\sup_{\delta_1} \sup_{\gamma_{\delta_1}} \sup_{\gamma_{-\delta_1}} 
\frac{b(\gamma_{\delta_1} +\gamma_{-\delta_1} -2/p)
+\alpha(\beta +\delta_1 +1)/2 -\gamma_{\delta_1} -1}
{\gamma_{\delta_1} +\gamma_{-\delta_1} +c}
\\
&= \Gamma_{\alpha(\beta +1),p,c}^{(1)}
\lor \sup_{\delta_1 \in E}
\biggl(
b -1
+\frac{-b(c +2/p)+\alpha(\beta +1) -1/p +c -1}
{1/2 +\alpha(\beta -\delta_1 +1)/2 -1/p +c}
\biggr)
= \Gamma_{\alpha(\beta +1),p,c}^{(1)},
\end{align*}
\begin{align*}
&\sup_{\delta_2} \sup_{\gamma_{\delta_2}} 
\frac{b(\gamma_{\delta_2} -1/p)
+\alpha(\beta +\delta_2 +1)/2 -\gamma_{\delta_2} -1}
{\gamma_{\delta_2} +c}
\\
&= \sup_{\delta_2 \in E^{\mrm c}}
\biggl(b -1
+\frac{-b(c +1/p) +c +\alpha(\beta +\delta_2 +1)/2 -1}
{\alpha(\beta +\delta_2 +1)/2 -1/p +c}
\biggr)
\\
&\qquad \lor 
\sup_{\gamma_{\delta_2} \in (1/p,1/2)}
\biggl(b 
+\frac{-b(c +1/p) -1/2}{\gamma_{\delta_2} +c}
\biggr)
\\
&= b\biggl(1 -\frac{c +1/p}{1/2 +c} \biggr)
-\frac{(1 -1/p) \land 1/2}{1/2 +c}
\\
&= b\biggl(1 -\frac{c +1/p}{1/2 +c} \biggr)
-\frac{1/2}{1/2 +c}
=: \widetilde \Gamma_{\alpha(\beta +1),p,c}^{(3)},
\end{align*}
and
\begin{align*}
&\sup_{\delta_3} \sup_{\gamma_{\delta_3}} 
\frac{b(\gamma_{\delta_3} -1/p)
+\alpha(\beta -\delta_3 +1)/2 -1}
{\gamma_{\delta_3} +c}
\\
&=\sup_{\delta_3 \in E^{\mrm c}}
\biggl(b +\frac{-b(c +1/p) +\alpha(\beta -\delta_3 +1)/2 -1}
{\alpha(\beta +\delta_3 +1)/2 -1/p +c}
\biggr)
\\
&\qquad \lor
\sup_{\delta_3 \in E}
\biggl(b +\frac{-b(c +1/p) +\alpha(\beta -\delta_3 +1)/2 -1}{1/2 +c}
\biggr)
\\
&= b \biggl(1 -\frac{c +1/p}{1/2 +c} \biggr)
-\frac{3/2 -\alpha(\beta +1) +1/p}{1/2 +c}
=: \Gamma_{\alpha(\beta +1),p,c}^{(3)}
\le \widetilde \Gamma_{\alpha(\beta +1),p,c}^{(3)}.
\end{align*}
Therefore, we get \eqref{Pf_prop8-04} under $\nu N^u \to \infty$ for some
\begin{equation*}
u < \Gamma_{x,p,c}^{(1)} \land \Gamma_{x,p,c}^{(3)} =
\begin{cases}
(b \tau_{1,x,p}^{(\mrm{space})} -\phi_{1,x,p}^{(\mrm{space})})
\land 
(b \tau_{2,x,p}^{(\mrm{space})} -\phi_{2,x,p}^{(\mrm{space})}),
& c=0,
\\
(b \rho_{1,x,p}^{(\mrm{space})} -\psi_{1,x,p}^{(\mrm{space})})
\land 
(b \rho_{2,x,p}^{(\mrm{space})} -\psi_{2,x,p}^{(\mrm{space})}),
& c=1/2,
\end{cases}
\end{equation*}
where $x =\alpha(\beta +1) \in (1/2,1]$. 
\end{proof}

\subsubsection{Control of truncation error and derivation of $\hat \theta_{0,\beta}$}
\label{sec4.2.3}
Because 
\begin{equation*}
\langle u, v \rangle_{\mcl L_{\beta}^2} 
= \sum_{l_1,l_2 \ge 1} 
\mu_{l_1, l_2}^{-\alpha \beta} 
\langle u, e_{l_1,l_2} \rangle \langle v, e_{l_1,l_2} \rangle,
\quad u, v \in \mcl L_{\beta}^2,
\end{equation*}
the estimator $\hat \theta_{0,\beta}^{(\mrm{disc})}$ 
given in \eqref{est_disc} can be represented as
\begin{align*}
\hat \theta_{0,\beta}^{(\mrm{disc})} &= 
\frac{\displaystyle 
\sum_{i=1}^N \sum_{l_1,l_2 \ge 1} \mu_{l_1, l_2}^{-\alpha \beta}
\langle \Psi_M X_{t_{i-1}}, e_{l_1,l_2} \rangle 
\langle Y_i^{\Psi_M}, e_{l_1,l_2} \rangle 
}{\displaystyle \frac{1}{N}
\sum_{i=1}^N \sum_{l_1,l_2 \ge 1} \mu_{l_1, l_2}^{-\alpha \beta}
\langle \Psi_M X_{t_{i-1}}, e_{l_1,l_2} \rangle^2}.
\end{align*}
In practice, we need to truncate the infinite sums in the estimator 
$\hat \theta_{0,\beta}^{(\mrm{disc})}$.
We thus consider the estimator
\begin{equation}\label{est_ver2}
\hat \theta_{0,\beta} = 
\hat \theta_{0,\beta,\nu,N,M,L} = 
\frac{\displaystyle 
\sum_{i=1}^N \sum_{l_1=1}^L \sum_{l_2=1}^L \mu_{l_1, l_2}^{-\alpha \beta}
\langle \Psi_M X_{t_{i-1}}, e_{l_1,l_2} \rangle 
\langle Y_i^{\Psi_M}, e_{l_1,l_2} \rangle 
}{\displaystyle \frac{1}{N}
\sum_{i=1}^N \sum_{l_1=1}^L \sum_{l_2=1}^L \mu_{l_1, l_2}^{-\alpha \beta}
\langle \Psi_M X_{t_{i-1}}, e_{l_1,l_2} \rangle^2}. 
\end{equation}
For $\beta >-1$, define
\begin{align*}
\overline J_{\beta,\nu} &=
\overline J_{\beta,\nu,N,M,L} =
\frac{1}{N}
\sum_{i=1}^N \sum_{l_1=1}^L \sum_{l_2=1}^L \mu_{l_1, l_2}^{-\alpha \beta}
\langle \Psi_M X_{t_{i-1}}, e_{l_1,l_2} \rangle^2,
\\
\overline K_{\beta,\nu} &=
\overline K_{\beta,\nu,N,M,L} =
\sum_{i=1}^N 
\sum_{l_1=1}^L \sum_{l_2=1}^L \mu_{l_1, l_2}^{-\alpha \beta}
\langle \Psi_M X_{t_{i-1}}, e_{l_1,l_2} \rangle 
\langle Y_i^{\Psi_M}, e_{l_1,l_2} \rangle.
\end{align*}

The following proposition shows the sufficient conditions 
for the asymptotic equivalence of the estimators
$\hat \theta_{0,\beta}^{(\mrm{disc})}$ and $\hat \theta_{0,\beta}$.
\begin{prop}\label{propB3}
Let $\alpha >0$, $-1 < \beta \le \frac{1}{\alpha} -1$ 
and $p > \frac{4}{\alpha(\beta +1)}$. 
Assume that [A1]$_{-1,p}$ and [B1]$_{\beta,p}$ hold.

\begin{enumerate}
\item[(1)]
Under [B2]$_{\beta,p}$, it holds that 
$\hat \theta_{0,\beta}^{(\mrm{disc})} - \hat \theta_{0,\beta} = \oo_{\PP}(1)$
as $\nu \to 0$, $N \to \infty$, $M \to \infty$ and $L \to \infty$.

\item[(2)]
Under [C3]$_{\beta,p}$, it holds that 
$\mcl R_{\beta,\nu} 
(\hat \theta_{0,\beta}^{(\mrm{disc})} - \hat \theta_{0,\beta} ) = \oo_{\PP}(1)$
as $\nu \to 0$, $N \to \infty$, $M \to \infty$ and $L \to \infty$.
\end{enumerate}
\end{prop}

\begin{proof}
Recall $\hat \theta_{0,\beta}^{(\mrm{disc})} = \widetilde K_{\beta,\nu}/\widetilde J_{\beta,\nu}$.
Let $\overline \Delta_{\beta,\nu}^{(1)} 
= |\overline J_{\beta,\nu} - \widetilde J_{\beta,\nu}|$
and $\overline \Delta_{\beta,\nu}^{(2)} 
= |\overline K_{\beta,\nu} - \widetilde K_{\beta,\nu}|$.
Note that
\begin{equation*}
|\hat \theta_{0,\beta}^{(\mrm{disc})} - \hat \theta_{0,\beta}|
\le
\frac{|\hat \theta_{0,\beta}^{(\mrm{disc})}| \overline \Delta_{\beta,\nu}^{(1)}
+\overline \Delta_{\beta,\nu}^{(2)}}
{\widetilde J_{\beta,\nu} -\overline \Delta_{\beta,\nu}^{(1)}}.
\end{equation*}
Let $i=1,2$.
Since it follows from Proposition \ref{propB2}, 
\eqref{Pf_prop8-01} and \eqref{Pf_prop8-02} 
that $\hat \theta_{0,\beta}^{(\mrm{disc})} = \OO_{\PP}(1)$ and
$\varphi_\beta(\nu)/\widetilde J_{\beta,\nu} = \OO_\PP(1)$
under [B1]$_{\beta,p}$, it is enough to show that
\begin{equation}\label{Pf_prop9-01}
\overline \Delta_{\beta,\nu}^{(i)}
= \oo_{\PP}(\varphi_{\beta}(\nu))
\end{equation}
under [B2]$_{\beta,p}$, and 
\begin{equation}\label{Pf_prop9-02}
\overline \Delta_{\beta,\nu}^{(i)}
= \oo_{\PP}(\varphi_{2\beta +1}(\nu)^{1/2})
\end{equation}
under [C3]$_{\beta,p}$.

\textbf{Step 1:} 
We estimate $\overline \Delta_{\beta,\nu}^{(1)}$ 
and $\overline \Delta_{\beta,\nu}^{(2)}$.
It holds that for $0 < \delta < \alpha(\beta +1) -\frac{4}{p}$,
\begin{equation*}
\sum_{l_1^2+l_2^2 > L^2} \mu_{l_1,l_2}^{-\alpha \beta} 
\langle \Psi_M X_t, e_{l_1,l_2} \rangle^2
\lesssim 
\frac{1}{L^{2\delta}}
\sum_{l_1,l_2 \ge 1} \mu_{l_1,l_2}^{\delta -\alpha \beta} 
\langle \Psi_M X_t, e_{l_1,l_2} \rangle^2
= \frac{\| \Psi_M X_t \|_{\mcl L_{\beta -\delta/\alpha}^2}^2}{L^{2\delta}},
\end{equation*}
\begin{equation*}
\sum_{l_1^2+l_2^2 > L^2} \mu_{l_1,l_2}^{-\alpha \beta} 
|\langle \Psi_M X_t, e_{l_1,l_2} \rangle| |\langle Y_i^{\Psi_M}, e_{l_1,l_2} \rangle|
\lesssim 
\frac{\| \Psi_M X_t \|_{\mcl L_{\beta -\delta/\alpha}^2} 
\| Y_i^{\Psi_M} \|_{\mcl L_{\beta -\delta/\alpha}^2} }{L^{2\delta}}.
\end{equation*}
Using Lemmas \ref{lemB3}, \ref{lemB5} and \ref{lemB6}, 
we see that for
$\frac{1}{p} < \gamma_{-\delta} < \frac{\alpha(\beta -\delta +1)}{2} -\frac{1}{p}$,
\begin{align*}
\EE \Bigl[\| \Psi_M X_t \|_{\mcl L_{\beta -\delta/\alpha}^2}^2 \Bigr] 
&\lesssim 
\EE \Bigl[\| \Psi_M X_t -X_t \|_{\mcl L_{\beta -\delta/\alpha}^2}^2 \Bigr] 
+ \EE \Bigl[\| X_t \|_{\mcl L_{\beta -\delta/\alpha}^2}^2 \Bigr]
\\
&\lesssim
\frac{\varphi_{\beta -\delta -2\gamma_{-\delta}/\alpha}(\nu)}
{M^{2(\gamma_{-\delta} -1/p)}} 
+\nu^{\alpha(\beta +1) -\delta -1}
\\
&\lesssim
\nu^{\alpha(\beta +1) -\delta -2\gamma_{\delta} -1},
\\
\EE \Bigl[\| Y_i^{\Psi_M} \|_{\mcl L_{\beta -\delta/\alpha}^2}^2 \Bigr] 
&\lesssim
\frac{\nu^{\alpha(\beta +1) -\delta -2\gamma_{-\delta} -1}}
{N^{\alpha(\beta +1) -\delta -2\gamma_{-\delta}}}
\end{align*}
and thus for $i=1,2$,
\begin{equation*}
\EE \Bigl[\overline \Delta_{\beta,\nu}^{(i)} \Bigr]
\lesssim
\frac{\nu^{\alpha(\beta +1) -\delta -2\gamma_{-\delta} -1}}
{L^{2\delta} N^{(\alpha(\beta +1) -\delta -2\gamma_{-\delta})/2-1}}
=: \mcl V_{\beta,\delta,\gamma_{-\delta}}(\nu, N, L).
\end{equation*}

\textbf{Step 2:} 
In order to show \eqref{Pf_prop9-01} and \eqref{Pf_prop9-02},
we verify that there exist $\delta$ and $\gamma_{-\delta}$ such that
\begin{equation*}
(\nu^{1-c -\alpha(\beta +1)} \mcl V_{\beta,\delta,\gamma_{-\delta}}(\nu, N, L))^{-1}
= \nu^{\delta +2\gamma_{-\delta} +c}
L^{2\delta} N^{\alpha(\beta -\delta +1)/2 -\gamma_{-\delta}-1}
\to \infty
\end{equation*}
for $c=0,1/2$, respectively.
Let $L = N^b$, $b > \frac{\alpha -1}{4}$. 
Recall $0 < \delta < \alpha(\beta +1) -\frac{4}{p}$
and $\frac{1}{p} < \gamma_{-\delta} < \frac{\alpha(\beta -\delta +1)}{2} -\frac{1}{p}$.
Since
\begin{align*}
&\sup_{\delta} \sup_{\gamma_{-\delta}}
\frac{2b \delta +\alpha(\beta -\delta +1)/2 -\gamma_{-\delta} -1}
{\delta +2\gamma_{-\delta} +c}
\\
&= \sup_{\delta}
\biggl(
-\frac{1}{2}
+\frac{(2b +(1-\alpha)/2)\delta +\alpha(\beta +1)/2 +c/2 -1}{\delta +2/p +c}
\biggr)
\\
&= 2b\biggl(1 -\frac{2/p +c}{\alpha(\beta +1) -2/p +c} \biggr)
-\frac{(\alpha -1)\alpha(\beta +1)/2 +(1 -2\alpha)/p +1}{\alpha(\beta +1) -2/p +c}
\\
&=
\begin{cases}
b \tau_{x,p}^{(\mrm{trunc})} -\phi_{x,\alpha,p}^{(\mrm{trunc})}, & c = 0,
\\
b \rho_{x,p}^{(\mrm{trunc})} -\psi_{x,\alpha,p}^{(\mrm{trunc})}, & c = 1/2,
\end{cases}
\end{align*}
we find that \eqref{Pf_prop9-01} and \eqref{Pf_prop9-02} hold
under [B2]$_{\beta,p}$ and [C3]$_{\beta,p}$, respectively.
\end{proof}

From Propositions \ref{propB1}-\ref{propB3}, 
we can prove that $\hat \theta_{0,\beta}$ in \eqref{est_ver2} and 
$\hat \theta_{0,\beta}^{(\mrm{cont})}$ are asymptotically equivalent.
Finally, we conclude the proof of Theorem \ref{th2} by 
showing that $\hat \theta_{0,\beta}$ in \eqref{est_ver2} equals 
that in \eqref{est}.
Since 
\begin{equation*}
\int_b^c \sqrt{2} \sin(\pi l x) \ee^{a x/2} \dd x = h_l(c:a) - h_l(b:a),
\end{equation*}
we have 
\begin{align*}
\langle \Psi_M X_t, e_{l_1,l_2} \rangle
&= \sum_{j=1}^{M_1} \sum_{k=1}^{M_2}
\int_{z_{k-1}}^{z_k} \int_{y_{j-1}}^{y_j}
2 X_t(y_{j-1},z_{k-1}) \sin(\pi l_1 y) \sin(\pi l_2 z) \ee^{(\kappa y + \eta z)/2}
\dd y \dd z
\\
&= \sum_{j=1}^{M_1} \sum_{k=1}^{M_2}
X_t(y_{j-1},z_{k-1}) 
\delta_j^{[y]} h_{l_1}(\kappa) \delta_k^{[z]} h_{l_2}(\eta)
\\
&= [ X_t ]_{M,l_1,l_2}
\end{align*}
and
\begin{align*}
\langle Y_i^{\Psi_M}, e_{l_1,l_2} \rangle
&= \langle \Psi_M X_{t_i}, e_{l_1,l_2} \rangle
- \ee^{\nu \lambda_{l_1,l_2}/N} \langle \Psi_M X_{t_{i-1}}, e_{l_1,l_2} \rangle 
\\
&= [ X_{t_{i}} ]_{M,l_1,l_2} 
- \ee^{\nu \lambda_{l_1,l_2}/N} [ X_{t_{i-1}} ]_{M,l_1,l_2}.
\end{align*}
Therefore, $\hat \theta_{0,\beta}$ in \eqref{est_ver2} 
can be expressed in the form of \eqref{est}.

\subsubsection{Auxiliary results for control of discretization errors}
Here, we provide some lemmas which are useful in the proof of Theorem \ref{th2}.
\begin{lem}\label{lemB1}
Let $\alpha >0$, $q > 2$ and $\beta > \frac{2}{\alpha q}-1$. 
Then, for any $0 \le u \le t \le 1$,
\begin{equation*}
\EE\Bigl[ \| \overline X_{t,u} \|_{\mcl L_{\beta}^q}^2 \Bigr] 
\lesssim (t-u)^{\alpha(\beta+1) \land 1} \varphi_{\beta}(\nu).
\end{equation*}
\end{lem}
\begin{proof}
Choose $\frac{1}{q} < \gamma < \frac{\alpha(\beta+1) \land 1}{2}$ 
and define 
\begin{equation*}
Z_{\nu,\gamma}(t,u,\cdot) 
= \sigma \int_u^t (t-s)^{-\gamma} S_{\nu(t-s)} \dd W_s^Q(\cdot).
\end{equation*}
Since
\begin{align*}
\| G_u(\bs x,\cdot) \|_{\mcl L_{\beta+1}^2}^2
&= \sum_{l_1,l_2 \ge 1}
\ee^{-2\lambda_{l_1,l_2}u} \mu_{l_1,l_2}^{-\alpha(\beta+1)} e_{l_1,l_2}^2(\bs x)
\nonumber
\\
&\lesssim 
\sum_{l_1,l_2 \ge 1} 
\frac{\ee^{-u(l_1^2+l_2^2)}}{(l_1^2+l_2^2)^{\alpha(\beta+1)}}
\sim 
\iint_{x^2+y^2 \ge 1}
\frac{\ee^{-u(x^2+y^2)}}{(x^2+y^2)^{\alpha(\beta+1)}} \dd x \dd y
\\
&\lesssim
\begin{cases}
u^{\alpha(\beta+1)-1}, & \alpha(\beta+1) < 1,
\\
-\log u, & \alpha(\beta+1) =1,
\\
1, & \alpha(\beta+1) > 1
\end{cases}
\end{align*}
for $u \in (0,1]$, we find from \eqref{closed-stoch-integral} that
\begin{align*}
\EE\Bigl[ \bigl(\widetilde Q^{\beta/2} Z_{\nu,\gamma}(t,u, \bs x)\bigr)^2 \Bigr] 
&= \sigma^2
\int_u^t (t-s)^{-2\gamma} 
\| G_{\nu(t-s)}(\bs x,\cdot) \|_{\mcl L_{\beta+1}^2}^2 \dd s
\nonumber
\\
&\lesssim 
\begin{cases}
\displaystyle
\nu^{\alpha(\beta +1) -1}
\int_u^t (t-s)^{\alpha(\beta+1) -1 -2\gamma} \dd s, 
& \alpha(\beta+1) <1,
\\
\displaystyle
-\int_u^t (t-s)^{-2\gamma} \log(\nu(t-s)) \dd s, 
& \alpha(\beta+1) = 1,
\\
\displaystyle
\int_u^t (t-s)^{-2\gamma} \dd s,
& \alpha(\beta+1) >1
\end{cases}
\nonumber
\\
&\sim (t-u)^{\alpha(\beta +1) \land 1 -2\gamma} \times
\begin{cases}
\nu^{\alpha(\beta+1) -1}, & \alpha(\beta+1) < 1,  
\\
-\log (\nu (t-u)), & \alpha(\beta+1) = 1,
\\
1, & \alpha(\beta+1) > 1,
\end{cases}
\end{align*}
where the last estimation follows from
\begin{equation*}
-\int_u^t (t-s)^{-2\gamma} \log(t-s) \dd s 
= \frac{(t-u)^{1 -2\gamma}}{(1 -2\gamma)^2}
(-(1 -2\gamma)\log (t-u) + 1).
\end{equation*}
Therefore, the Gaussianity of 
$\widetilde Q^{\beta/2} Z_{\nu,\gamma}(t,u,\bs x)$ implies that
\begin{align*}
\EE\Bigl[\| Z_{\nu,\gamma}(t,u) \|_{\mcl L_{\beta}^q}^q \Bigr]
& \lesssim 
\int_D \EE \Bigl[ \bigl| \widetilde Q^{\beta/2}
Z_{\nu,\gamma}(t,u,\bs x) \bigr|^q \Bigr] \dd \bs x
\nonumber
\\
&\sim
\int_D \EE\Bigl[ \bigl(\widetilde Q^{\beta/2}
Z_{\nu,\gamma}(t,u,\bs x)\bigr)^2 \Bigr]^{q/2} 
\dd \bs x
\nonumber
\\
&\lesssim
(t-u)^{(\alpha(\beta+1) \land 1 -2\gamma)q/2} \times
\begin{cases}
\nu^{(\alpha(\beta+1)-1)q/2}, & \alpha(\beta+1) < 1,
\\
(-\log (\nu (t-u)))^{q/2}, & \alpha(\beta+1) = 1,
\\
1, & \alpha(\beta+1) > 1.
\end{cases}
\end{align*}
Since it follows from Theorem 5.10 in \cite{DaPrato_Zabczyk2014} that
\begin{equation*}
\overline X_{t,u} = \frac{\sigma \sin(\gamma \pi)}{\pi} 
\int_u^t (t-s)^{\gamma -1} S_{\nu(t-s)} Z_{\nu,\gamma}(s,u) \dd s,
\end{equation*}
we see from \eqref{Bochner-ineq}, the H\"{o}lder inequality and
\begin{equation*}
\int_0^t (-s^a \log (\nu s))^b \dd s \lesssim t^{1+ab} (-\log(\nu t))^b
\end{equation*}
that for $1/q+1/q^*=1$, 
\begin{align*}
\EE \Bigl[ \| \overline X_{t,u} \|_{\mcl L_{\beta}^q}^2 \Bigr]
&\sim
\EE \Biggl[
\biggl\|
\int_u^t (t-s)^{\gamma -1} S_{\nu(t-s)} Z_{\nu,\gamma}(s,u) \dd s
\biggr\|_{\mcl L_{\beta}^q}^2
\Biggr]
\\
&\le
\EE \Biggl[
\biggl(
\int_u^t (t-s)^{\gamma -1} \|Z_{\nu,\gamma}(s,u) \|_{\mcl L_{\beta}^q} \dd s
\biggr)^2
\Biggr]
\\
&\le
(t-u)^{2(q^*(\gamma -1) +1)/q^*}
\EE \Biggl[
\int_u^t \|Z_{\nu,\gamma}(s,u) \|_{\mcl L_{\beta}^q}^q \dd s
\Biggr]^{2/q}
\\
&\lesssim
(t-u)^{2(\gamma -1/q)} (t-u)^{\alpha(\beta+1) \land 1 -2\gamma +2/q}
\\
&\qquad\times
\begin{cases}
\nu^{\alpha(\beta+1)-1}, & \alpha(\beta+1) < 1,
\\
-\log (\nu(t-u)), & \alpha(\beta+1) = 1,
\\
1, & \alpha(\beta+1) > 1
\end{cases}
\\
&\sim (t-u)^{\alpha(\beta+1) \land 1} \varphi_{\beta}(\nu).
\end{align*}
\end{proof}

\begin{lem}\label{lemB2}
Let $\alpha >0$ and $\beta > -1$. Then, for any $0 \le u \le t \le 1$,
\begin{equation*}
\EE\Bigl[ \| \overline X_{t,u} \|_{\mcl L_{\beta}^2}^2 \Bigr] 
\sim \nu^{\alpha(\beta+1)-1} (t-u)^{\alpha(\beta+1)}.
\end{equation*}
\end{lem}
\begin{proof}
The desired result follows from
\begin{align*}
\EE\Bigl[ \| \overline X_{t,u} \|_{\mcl L_{\beta}^2}^2 \Bigr]
&= \sigma^2 
\int_u^t \| S_{\nu(t-s)} \|_{\mrm{HS}(\mcl U_\beta)}^2 \dd s  
\\
&= \sigma^2 \sum_{l_1,l_2 \ge 1} 
\mu_{l_1,l_2}^{-\alpha(\beta+1)}
\times \frac{1-\ee^{-2\lambda_{l_1,l_2} \nu (t-u)}}{2\lambda_{l_1,l_2}\nu} 
\\
&\sim \sum_{l_1,l_2 \ge 1} (l_1^2+l_2^2)^{-\alpha(\beta+1)} 
\biggl\{ \frac{1}{(l_1^2+l_2^2)\nu} \land (t-u) \biggr\}
\\
&= 
\nu^{-1}
\sum_{l_1^2+l_2^2 > (\nu (t-u))^{-1}} 
\frac{1}{(l_1^2+l_2^2)^{1+\alpha(\beta+1)}}
\\
&\qquad 
+ (t-u) \sum_{l_1^2+l_2^2 \le (\nu (t-u))^{-1}} 
(l_1^2+l_2^2)^{-\alpha(\beta+1)}
\\
&\sim
\nu^{-1} \iint_{x^2+y^2 >(\nu (t-u))^{-1}} 
\frac{\dd x \dd y}{(x^2+y^2)^{1+\alpha(\beta+1)}}
\\
&\qquad 
+ (t-u) 
\iint_{1 \le x^2+y^2\le (\nu (t-u))^{-1}} (x^2+y^2)^{-\alpha(\beta+1)} \dd x \dd y
\\
&\sim \nu^{\alpha(\beta+1)-1} (t-u)^{\alpha(\beta+1)}.
\end{align*}
\end{proof}

For $\alpha >0$, $\beta > -1$, define
\begin{equation*}
\widetilde \varphi_{\beta,q}(\nu) = 
\begin{cases}
\nu^{\alpha(\beta+1) \land 1 -1}, & q = 2,
\\
\varphi_{\beta}(\nu), & q > 2.
\end{cases}
\end{equation*}

\begin{lem}\label{lemB3}
Let $\alpha >0$. 
Assume that $\beta$ and $q$ satisfy any one of the following (i) and (ii).
\begin{itemize}
\item[(i)]
$q = 2$, $\beta > -1$,

\item[(ii)]
$q > 2$, $\beta > \frac{2}{\alpha q} -1$.
\end{itemize}
If [A1]$_{\beta,q}$ holds, then
\begin{equation*}
\EE\Bigl[ \| X_t \|_{\mcl L_{\beta}^q}^2 \Bigr] 
\lesssim 1 + \widetilde \varphi_{\beta,q}(\nu),
\quad t \in \mbb T.
\end{equation*}
\end{lem}

\begin{proof}
It follows from \eqref{vcf}, \eqref{Bochner-ineq}, [A1]$_{\beta,q}$ and 
Lemmas \ref{lemB1} and \ref{lemB2} that 
\begin{align*}
\EE \Bigl[\| X_t \|_{\mcl L_{\beta}^q}^2 \Bigr]
&\lesssim
\EE \Bigl[\| S_{\nu t} X_0 \|_{\mcl L_{\beta}^q}^2 \Bigr]
+ \EE\Biggl[
\biggl\|
\int_0^t S_{\nu(t-s)} X_s \dd s
\biggr\|_{\mcl L_{\beta}^q}^2
\Biggr]
+ \EE\Bigl[\| \overline X_t \|_{\mcl L_{\beta}^q}^2 \Bigr]
\\
&\lesssim
\EE \Bigl[\| X_0 \|_{\mcl L_{\beta}^q}^2 \Bigr]
+ \EE\biggl[
\int_0^t  \| S_{\nu (t-s)} X_s \|_{\mcl L_{\beta}^q}^2 \dd s
\biggr]
+ \widetilde \varphi_{\beta,q}(\nu)
\\
&\lesssim
1 + \widetilde \varphi_{\beta,q}(\nu) + \int_0^t 
\EE \Bigl[\| X_s \|_{\mcl L_{\beta}^q}^2 \Bigr] \dd s.
\end{align*}
Using Gronwall's inequality, we have
$\EE [\| X_t \|_{\mcl L_{\beta}^q}^2 ]
\lesssim 1 + \widetilde \varphi_{\beta,q}(\nu)$.
\end{proof}

\begin{lem}\label{lemB4}
Let $\alpha >0$ and $\beta > -1$. 
Assume that [A1]$_{\beta,2}$ holds.
It holds that for $0 \le \gamma < \frac{\alpha(\beta+1)}{2}$ and
$t \in [t_{i-1}, t_i]$, $i=1,\ldots, N$, 
\begin{equation*}
\EE \Bigl[\| X_t-X_{t_{i-1}} \|_{\mcl L_{\beta}^2}^2 \Bigr] 
\lesssim 
\frac{\nu^{\alpha(\beta+1)-1}}{N^{2\gamma \land 1}} + 
\Bigl(\frac{\nu}{N}\Bigr)^{2\gamma} + \frac{1}{N}.
\end{equation*}
\end{lem}

\begin{proof}
Noting that 
\begin{equation*}
\| \mcl A^\gamma u \|^2
= \sum_{l_1,l_2 \ge1} \lambda_{l_1,l_2}^{2\gamma} 
\langle u, e_{l_1,l_2} \rangle^2 
\sim \sum_{l_1,l_2 \ge1} \mu_{l_1,l_2}^{-\alpha(-2\gamma/\alpha)} 
\langle u, e_{l_1,l_2} \rangle^2
= \| \widetilde Q^{-\gamma/\alpha} u \|^2
= \| u \|_{\mcl L_{-2\gamma/\alpha}^2}^2,
\end{equation*}
we see from Proposition 4.41 in \cite{Hairer2009}, 
\eqref{Pf_prop7-01}, \eqref{Bochner-ineq}, Lemmas \ref{lemB2} and \ref{lemB3} that
for $0 \le \gamma < \frac{\alpha(\beta+1)}{2}$,
\begin{equation}\label{Pf_lem7-01}
\| S_t u- u \|_{\mcl L_{\beta}^2} 
\lesssim t^\gamma \| \mcl A^\gamma u \|_{\mcl L_{\beta}^2}
\sim t^\gamma \| u \|_{\mcl L_{\beta-2\gamma/\alpha}^2},
\quad t \in \mbb  T
\end{equation}
and
\begin{align*}
\EE \Bigl[\| X_t-X_{t_{i-1}} \|_{\mcl L_{\beta}^2}^2 \Bigr] 
&\lesssim
\EE \Bigl[
\| S_{\nu(t-t_{i-1})}X_{t_{i-1}} - X_{t_{i-1}} \|_{\mcl L_{\beta}^2}^2
\Bigr]
\\
&\qquad+ \int_{t_{i-1}}^t \EE \Bigl[ 
\| S_{\nu(t-s)} X_s \|_{\mcl L_{\beta}^2}^2 
\Bigr] \dd s
+ \EE \Bigl[\| \overline X_{t,t_{i-1}} \|_{\mcl L_{\beta}^2}^2 \Bigr]
\\
&\lesssim
\Bigl(\frac{\nu}{N}\Bigr)^{2\gamma} 
\EE \Bigl[ \| X_{t_{i-1}} \|_{\mcl L_{\beta-2\gamma/\alpha}^2}^2 \Bigr]
+ \int_{t_{i-1}}^t \EE\Bigl[\| X_s \|_{\mcl L_{\beta}^2}^2 \Bigr] \dd s
+ \EE \Bigl[\| \overline X_{t,t_{i-1}} \|_{\mcl L_{\beta}^2}^2 \Bigr]
\\
&\lesssim
\frac{\nu^{\alpha(\beta+1)-1}+ \nu^{2\gamma}}{N^{2\gamma}}
+ \frac{\nu^{\alpha(\beta+1)-1}+1}{N} 
+ \frac{\nu^{\alpha(\beta+1)-1}}{N^{\alpha(\beta+1)}}
\\
&\lesssim
\frac{\nu^{\alpha(\beta+1)-1}}{N^{2\gamma \land 1}} + 
\Bigl(\frac{\nu}{N}\Bigr)^{2\gamma} + \frac{1}{N}.
\end{align*}
\end{proof}

Let $M = (M_1 \land M_2)^2$. 
\begin{lem}\label{lemB5}
Let $\alpha >0$, $q >2$ and $\beta > \frac{4}{\alpha q} -1$. 
Assume that [A1]$_{\beta,q}$ holds.
It holds that for
$\frac{1}{q} < \gamma < (\frac{\alpha(\beta +1)}{2} - \frac{1}{q}) \land \frac{1}{2}$ ,
\begin{equation*}
\EE \Bigl[\| X_t - \Psi_M X_t \|_{\mcl L_{\beta}^2}^2 \Bigr]
\lesssim
\frac{\varphi_{\beta -2\gamma/\alpha}(\nu)}{M^{2(\gamma -1/q)}}, 
\quad t \in \mbb T.
\end{equation*}
\end{lem}

\begin{proof}
Since $\| u \|_{\mcl L_{-2\gamma/\alpha}^q} 
\sim \| \mcl A^{\gamma} u \|_{\mcl L^q}$ and Theorem 16.15 in \cite{Yagi2010} 
yield $\mcl L_{-2\gamma/\alpha}^q \hookrightarrow \mcl W^{2\gamma,q}$,
we find from the Sobolev embedding theorem (Theorem 8.2 in \cite{Nezza_etal2012}) that
\begin{equation*}
\| u \|_{\mcl C^{2(\gamma -1/q)}}
\lesssim \| \mcl A^{\gamma} u \|_{\mcl L^q} 
\sim \| u \|_{\mcl L_{-2\gamma/\alpha}^q}.
\end{equation*}
Therefore, it holds 
from $\| \widetilde Q^{\beta/2} X_t \|_{\mcl C^{2(\gamma -1/q)}}
\lesssim \| X_t \|_{\mcl L_{\beta -2\gamma/\alpha}^q}$
and Lemma \ref{lemB3} that for $t \in \mbb T$, 
\begin{align*}
\| X_t - \Psi_M X_t \|_{\mcl L_{\beta}^2}^2
&= \sum_{j=1}^{M_1} \sum_{k=1}^{M_2} 
\int_{z_{k-1}}^{z_k} \int_{y_{j-1}}^{y_j}  
(\widetilde Q^{\beta/2} X_t(y,z) -\widetilde Q^{\beta/2} X_t(y_{j-1},z_{k-1}))^2 
\bar v(y,z) \dd y \dd z
\\
&\le \sum_{j=1}^{M_1} \sum_{k=1}^{M_2} 
\int_{z_{k-1}}^{z_k} \int_{y_{j-1}}^{y_j}  
\| \widetilde Q^{\beta/2} X_t \|_{\mcl C^{2(\gamma -1/q)}}^2
\Biggl|
\begin{pmatrix}
y - y_{j-1}
\\
z - z_{k-1}
\end{pmatrix}
\Biggr|^{4(\gamma -1/q)} 
\bar v(y,z) \dd y \dd z
\\
&\lesssim
\Biggl(
\frac{\| X_t \|_{\mcl L_{\beta -2\gamma/\alpha}^q}}{M^{\gamma -1/q}}
\Biggr)^2
\end{align*}
and
\begin{equation*}
\EE \Bigl[\| X_t - \Psi_M X_t \|_{\mcl L_{\beta}^2}^2 \Bigr]
\lesssim
\frac{
\EE\Bigl[\| X_t \|_{\mcl L_{\beta -2\gamma/\alpha}^q}^2\Bigr]}{M^{2(\gamma -1/q)}}
\lesssim
\frac{\varphi_{\beta -2\gamma/\alpha}(\nu)}{M^{2(\gamma -1/q)}}.
\end{equation*}
\end{proof}

\begin{lem}\label{lemB6}
Let $\alpha >0$, $Y_i = X_{t_i} -S_{\nu(t_i-t_{i-1})} X_{t_{i-1}}$
and $Y_i^{\Psi_M} = \Psi_M X_{t_i} - S_{\nu (t_i-t_{i-1})} \Psi_M X_{t_{i-1}}$.

\begin{itemize}
\item[(1)]
For $q$ and $\beta$ satisfying either (i) or (ii) of Lemma \ref{lemB3},
assume that [A1]$_{\beta,q}$ holds. Then,
\begin{equation*}
\EE \Bigl[\| Y_i \|_{\mcl L_{\beta}^q}^2 \Bigr] 
\lesssim \frac{\widetilde \varphi_{\beta,q}(\nu)}{N^{\alpha(\beta +1)}},
\quad i =1,\ldots, N.
\end{equation*}

\item[(2)]
For $q >2$ and $\beta > \frac{4}{\alpha q} -1$, 
assume that [A1]$_{\beta,q}$ holds.
It holds that for
$\frac{1}{q} < \gamma < (\frac{\alpha(\beta +1)}{2} - \frac{1}{q}) \land \frac{1}{2}$,
\begin{equation*}
\EE \Bigl[\| Y_i - Y_i^{\Psi_M} \|_{\mcl L_{\beta}^2}^2 \Bigr]
\lesssim
\frac{\varphi_{\beta -2\gamma/\alpha}(\nu)}
{M^{2(\gamma -1/q)} N^{\alpha(\beta +1) -2\gamma}},
\quad i =1,\ldots, N.
\end{equation*}
\end{itemize}
\end{lem}

\begin{proof}
(1) Using \eqref{Pf_prop7-01}, \eqref{Bochner-ineq} 
and Lemmas \ref{lemB1}-\ref{lemB3}, we get
\begin{align*}
\EE \Bigl[\| Y_i \|_{\mcl L_{\beta}^q}^2 \Bigr]
&\lesssim
\EE \Biggl[ \biggl\|
\int_{t_{i-1}}^{t_i} S_{\nu(t_i-s)} X_s \dd s
\biggr\|_{\mcl L_{\beta}^q}^2 \Biggr]
+ \EE\Bigl[\| \overline X_{t_i,t_{i-1}} \|_{\mcl L_{\beta}^q}^2 \Bigr]
\nonumber
\\
&\lesssim
\EE \biggl[
\int_{t_{i-1}}^{t_i} \| S_{\nu(t_i -s)} X_s \|_{\mcl L_{\beta}^q}^2 \dd s
\biggr]
+ \EE \Bigl[\| \overline X_{t_i,t_{i-1}} \|_{\mcl L_{\beta}^q}^2 \Bigr]
\nonumber
\\
&\lesssim
\frac{1}{N} \int_{t_{i-1}}^{t_i} 
\EE \Bigl[ \| X_s \|_{\mcl L_{\beta}^q}^2 \Bigr] \dd s
+ \EE \Bigl[\| \overline X_{t_i,t_{i-1}} \|_{\mcl L_{\beta}^q}^2 \Bigr]
\nonumber
\\
&\lesssim
\frac{1 + \widetilde \varphi_{\beta,q}(\nu)}{N^2}
+ \frac{\widetilde \varphi_{\beta,q}(\nu)}{N^{\alpha(\beta +1)}}
\nonumber
\\
&\lesssim
\frac{\widetilde \varphi_{\beta,q}(\nu)}{N^{\alpha(\beta +1)}}.
\end{align*}

(2) In the same way as Lemma \ref{lemB5}, we see from (1) that
\begin{equation*}
\EE \Bigl[\| Y_i - Y_i^{\Psi_M} \|_{\mcl L_{\beta}^2}^2 \Bigr]
\lesssim
\frac{\EE 
\Bigl[\| Y_i \|_{\mcl L_{\beta -2\gamma/\alpha}^q}^2 \Bigr]}{M^{2(\gamma -1/q)}}
\lesssim
\frac{\varphi_{\beta -2\gamma/\alpha}(\nu)}
{M^{2(\gamma -1/q)} N^{\alpha(\beta +1) -2\gamma}}.
\end{equation*}
\end{proof}

\subsection{Proofs of Proposition \ref{prop1} and Corollary \ref{cor1}}\label{sec4.3}
In this subsection, we provide proofs of the results in Subsection \ref{sec2.4}.
\begin{proof}[\bf{Proof of Proposition \ref{prop1}}]
Let $\mcl A_{\theta_0,\nu} = \nu \mcl A -\theta_0$,
$T_{\nu,t} = \ee^{-t \mcl A_{\theta_0,\nu}}$
and $\lambda_{l_1,l_2}^{(\theta_0)} = \nu \lambda_{l_1,l_2} -\theta_0$.
Note that $\mcl A_{\theta_0,\nu} e_{l_1,l_2} 
= \lambda_{l_1,l_2}^{(\theta_0)} e_{l_1,l_2}$.
Moreover, let
\begin{equation*}
f_1(x) = 
\begin{cases}
\frac{1 -\ee^{-x}}{x}, & x \neq 0,
\\
1, & x = 0, 
\end{cases}
\quad
f_2(x) = 
\begin{cases}
\frac{x -1 +\ee^{-x}}{x^2}, & x \neq 0,
\\
\frac{1}{2}, & x = 0.
\end{cases}
\end{equation*}
It follows from $X_t = T_{\nu,t}X_0 +\sigma \int_0^t T_{\nu,t-s} \dd W_s^Q$
and the independence of $X_0$ and $W_t^Q$ that
\begin{align*}
\EE[J_{\beta,\nu}] 
&= \int_0^1 \EE\Bigl[ \| T_{\nu,t} X_0 \|_{\mcl L_{\beta}^2}^2 \Bigr] \dd t
+ \sigma^2 \int_0^1 \int_0^t 
\| T_{\nu,t-s} \|_{\mrm{HS}(\mcl U_\beta)}^2 \dd s \dd t
\nonumber
\\
&= \int_0^1 
\sum_{l_1,l_2 \ge 1}
\ee^{-2t \lambda_{l_1,l_2}^{(\theta_0)}} \mu_{l_1,l_2}^{-\alpha \beta}
\EE[\langle X_0, e_{l_1,l_2} \rangle^2] \dd t
\\
&\qquad+\sigma^2 \int_0^1 \int_0^t 
\sum_{l_1,l_2 \ge 1}
\ee^{-2(t-s) \lambda_{l_1,l_2}^{(\theta_0)}} \mu_{l_1,l_2}^{-\alpha(\beta +1)}
\dd s \dd t
\nonumber
\\
&= \sum_{l_1,l_2 \ge 1} 
f_1(2 \lambda_{l_1,l_2}^{(\theta_0)}) \mu_{l_1,l_2}^{-\alpha \beta}
\EE[\langle X_0, e_{l_1,l_2} \rangle^2]
+ \sigma^2 \sum_{l_1,l_2 \ge 1} 
f_2(2 \lambda_{l_1,l_2}^{(\theta_0)}) \mu_{l_1,l_2}^{-\alpha(\beta +1)}
\\
&= \sum_{l_1,l_2 \ge 1} 
f_{l_1,l_2}(2 \lambda_{l_1,l_2}^{(\theta_0)}) \mu_{l_1,l_2}^{-\alpha(\beta +1)},
\end{align*}
where $f_{l_1,l_2}(x) = f_1(x) \mu_{l_1,l_2}^\alpha 
\EE[\langle X_0, e_{l_1,l_2} \rangle^2]+ \sigma^2 f_2(x) > 0$, $x \in \mbb R$. 
Note that $f_{l_1,l_2}(x) \lesssim 1/x$ under [A1]$_{-1,2}$.
Therefore, we obtain
\begin{equation*}
\mcl R_{\beta,\nu} = 
\frac{\EE[J_{\beta,\nu}]}{\sqrt{\EE[J_{2\beta +1,\nu}]}}
= \frac{\displaystyle
\sum_{l_1,l_2 \ge 1} 
f_{l_1,l_2}(2 \lambda_{l_1,l_2}^{(\theta_0)}) \mu_{l_1,l_2}^{-\alpha(\beta +1)}
}{\displaystyle
\sqrt{\sum_{l_1,l_2 \ge 1} 
f_{l_1,l_2}(2 \lambda_{l_1,l_2}^{(\theta_0)}) \mu_{l_1,l_2}^{-2\alpha(\beta +1)}}}.
\end{equation*}
For $x > 0$, let
\begin{equation*}
F(x) = 
\sum_{l_1,l_2 \ge 1} 
f_{l_1,l_2}(2 \lambda_{l_1,l_2}^{(\theta_0)}) \mu_{l_1,l_2}^{-x}, 
\quad 
G(x) = \frac{F(x)^2}{F(2x)}.
\end{equation*}
Since $F(x)$ is differentiable and 
\begin{equation*}
F'(x) = 
-\sum_{l_1,l_2 \ge 1} 
f_{l_1,l_2}(2 \lambda_{l_1,l_2}^{(\theta_0)}) \mu_{l_1,l_2}^{-x} \log \mu_{l_1,l_2},
\end{equation*}
we have
\begin{align*}
&F'(x) F(2x) -F(x) F'(2x) 
\\
&=
\sum_{l_1, \ldots,l_4 \ge 1}
f_{l_1,l_2}(2 \lambda_{l_1,l_2}^{(\theta_0)}) 
f_{l_3,l_4}(2 \lambda_{l_3,l_4}^{(\theta_0)})
\mu_{l_1,l_2}^{-2x} \mu_{l_3,l_4}^{-2x} 
(\mu_{l_3,l_4}^{x} -\mu_{l_1,l_2}^{x}) \log \mu_{l_1,l_2} 
\\
&=
\sum_{\substack{l_1,\ldots,l_4 \ge 1,\\ 
l_1^2 + l_2^2 < l_3^2 +l_4^2}}
f_{l_1,l_2}(2 \lambda_{l_1,l_2}^{(\theta_0)}) 
f_{l_3,l_4}(2 \lambda_{l_3,l_4}^{(\theta_0)})
\mu_{l_1,l_2}^{-2x} \mu_{l_3,l_4}^{-2x} 
(\mu_{l_3,l_4}^{x} -\mu_{l_1,l_2}^{x}) \log \mu_{l_1,l_2} 
\\
&\qquad +
\sum_{\substack{l_1,\ldots,l_4 \ge 1,\\ 
l_1^2 + l_2^2 > l_3^2 +l_4^2}}
f_{l_1,l_2}(2 \lambda_{l_1,l_2}^{(\theta_0)}) 
f_{l_3,l_4}(2 \lambda_{l_3,l_4}^{(\theta_0)})
\mu_{l_1,l_2}^{-2x} \mu_{l_3,l_4}^{-2x} 
(\mu_{l_3,l_4}^{x} -\mu_{l_1,l_2}^{x}) \log \mu_{l_1,l_2} 
\\
&=
\sum_{\substack{l_1,\ldots,l_4 \ge 1,\\ 
l_1^2 + l_2^2 < l_3^2 +l_4^2}}
f_{l_1,l_2}(2 \lambda_{l_1,l_2}^{(\theta_0)}) 
f_{l_3,l_4}(2 \lambda_{l_3,l_4}^{(\theta_0)})
\mu_{l_1,l_2}^{-2x} \mu_{l_3,l_4}^{-2x} 
(\mu_{l_3,l_4}^{x} -\mu_{l_1,l_2}^{x}) 
\log \frac{\mu_{l_1,l_2}}{\mu_{l_3,l_4}}
\\
& < 0 
\end{align*}
and hence
\begin{equation*}
G'(x) = \frac{2F(x)}{F(2x)^2}(F'(x) F(2x) -F(x) F'(2x)) < 0,
\end{equation*}
which implies that the function $\beta \mapsto \mcl R_{\beta,\nu}$ 
is decreasing for $\beta \in (-1,\frac{1}{\alpha} -1]$. 
\end{proof}

\begin{proof}[\bf{Proof of Corollary \ref{cor1}}]
Fix ${\rm m}> {\rm n} >0$.

(1) Since
$\nu (M_1 \land M_2)^{2r} N^{-\phi_{x}^{(\mrm{space})}}
= \nu^{1 -{\rm m}r +{\rm n}\phi_{x}^{(\mrm{space})}}$,
it follows from [B1]$_{\beta,\infty}$ that 
\begin{equation}\label{Pf_prop2-01}
{\rm m} > \frac{1 +{\rm n}\phi_{\alpha(\beta +1)}^{(\mrm{space})}}{r}
> 1 +{\rm n}\phi_{\alpha(\beta +1)}^{(\mrm{space})}.
\end{equation}
Because $\phi_{x}^{(\mrm{space})} \ge \phi_{1}^{(\mrm{space})} = 1$,
there exists $\beta$ such that \eqref{Pf_prop2-01} if ${\rm m} > {\rm n} +1$.
Solving the inequality \eqref{Pf_prop2-01} for $\beta$,
we have $\beta > \beta_{\alpha,{\rm n},{\rm m}}^{(\mrm{cons})}$.
Noting that
$\nu L^s N^{-\phi_{x,\alpha}^{(\mrm{trunc})}}
= \nu^{1 -\ell s +{\rm n}\phi_{x,\alpha}^{(\mrm{trunc})}}$,
we find from [B2]$_{\beta,\infty}$ that
for $\beta > \beta_{\alpha,{\rm n},{\rm m}}^{(\mrm{cons})}$,
\begin{equation*}
\ell > 
\frac{1 +{\rm n}\phi_{\alpha(\beta +1),\alpha}^{(\mrm{trunc})}}{s}
> \frac{1 +{\rm n}\phi_{\alpha(\beta +1),\alpha}^{(\mrm{trunc})}}{2}.
\end{equation*}

(2) Notice that $\nu N^q = \nu^{1 -{\rm n} q}$ 
and $\nu (M_1 \land M_2)^{2r_j} N^{-\psi_{j,x}^{(\mrm{space})}}
= \nu^{1 -{\rm m}r_j +{\rm n}\psi_{j,x}^{(\mrm{space})}}$.
Hence, we see from [C1]$_\beta$ and [C2]$_{\beta,\infty}$ that
\begin{equation*}
{\rm n} > \frac{1}{q} > \frac{1}{\rho_{\alpha(\beta +1)}^{(\mrm{time})}} = 
\begin{cases}
\frac{1}{2 \alpha(\beta +1)}, & \alpha(\beta +1) \le \frac{1}{2},
\\
2(1-\alpha(\beta +1)), & \alpha(\beta +1) > \frac{1}{2},
\end{cases}
\end{equation*}
\begin{equation*}
{\rm m} > \frac{1 +{\rm n}\psi_{1,\alpha(\beta +1)}^{(\mrm{space})}}{r_1}
\lor \frac{1 +{\rm n}\psi_{2,\alpha(\beta +1)}^{(\mrm{space})}}{r_2}
> \frac{1 +{\rm n}\psi_{1,\alpha(\beta +1)}^{(\mrm{space})}}
{\rho_{\alpha(\beta +1)}^{(\mrm{space})}}
\lor 2(1 +{\rm n}\psi_{2,\alpha(\beta +1)}^{(\mrm{space})}),
\end{equation*}
which together with 
$\psi_{1,x}^{(\mrm{space})} \ge \psi_{1,1}^{(\mrm{space})} =2/3$, 
$\psi_{2,x}^{(\mrm{space})} \ge \psi_{2,1}^{(\mrm{space})} =1/2$
and $\rho_{x}^{(\mrm{space})} \le \rho_{1}^{(\mrm{space})} =2/3$ 
imply that
$\beta > \beta_{\alpha,{\rm n},{\rm m}}^{(\mrm{asym})}$
for ${\rm m} > {\rm n} +2$. 
Since $\nu L^s N^{-\psi_{x,\alpha}^{(\mrm{trunc})}}
= \nu^{1 -\ell s +{\rm n}\psi_{x,\alpha}^{(\mrm{trunc})}}$
and [C3]$_{\beta,\infty}$, we find that for 
$\beta > \beta_{\alpha,{\rm n},{\rm m}}^{(\mrm{asym})}$,
\begin{equation*}
\ell > 
\frac{1 +{\rm n}\psi_{\alpha(\beta +1),\alpha}^{(\mrm{trunc})}}{s}
> \frac{1 +{\rm n}\psi_{\alpha(\beta +1),\alpha}^{(\mrm{trunc})}}
{\rho_{\alpha(\beta +1)}^{(\mrm{trunc})}}.
\end{equation*}
\end{proof}

\subsection{Proofs of Propositions \ref{prop3} and \ref{prop4}}\label{sec4.4}
This subsection gives proofs of Propositions \ref{prop3} and \ref{prop4}. 
Since the proofs are analogous to those of Proposition 2.1 and Theorem 2.2 
in \cite{TKU2023arXiv}, we shall focus on the different parts.
For the detailed proofs, refer to \cite{TKU2023arXiv}.

Let $e_l^{(1)}(x) = \sqrt{2}\sin(\pi l x)\ee^{-\kappa x/2}$ and 
$e_l^{(2)}(x) = \sqrt{2}\sin(\pi l x)\ee^{-\eta x/2}$.
The second order moment of the triple increment $T_{i,j,k} X$ is given by
\begin{equation*}
\EE[(T_{i,j,k} X)^2] = \sigma^2 F_{j,k}^{j,k} + R_{i,j,k},
\end{equation*}
where 
\begin{align*}
F_{j',k'}^{j,k} &= 
\sum_{l_1,l_2 \ge 1} \frac{1 -\ee^{\lambda_{l_1,l_2}^{(\theta_0)}\Delta}}
{\lambda_{l_1,l_2}^{(\theta_0)} \mu_{l_1,l_2}^\alpha}
(e_{l_1}^{(1)}(\widetilde y_j) -e_{l_1}^{(1)}(\widetilde y_{j-1}))
(e_{l_2}^{(2)}(\widetilde z_k) -e_{l_2}^{(2)}(\widetilde z_{k-1}))
\\
&\qquad \times
(e_{l_1}^{(1)}(\widetilde y_{j'}) -e_{l_1}^{(1)}(\widetilde y_{j'-1}))
(e_{l_2}^{(2)}(\widetilde z_{k'}) -e_{l_2}^{(2)}(\widetilde z_{k'-1}))
\end{align*}
and $R_{i,j,k}$ satisfies $\sum_{i=1}^N R_{i,j,k} \lesssim F_{j,k}^{j,k}$.

For a function $f:\mbb R^2 \to \mbb R$ and a positive number $\epsilon$, define
$D_{1,\epsilon} f(x,y) = f(x +\epsilon, y) -f(x,y)$
and $D_{2,\epsilon} f(x,y) = f(x, y +\epsilon) -f(x,y)$.
For $\alpha \in (0,2)$ and $y,z \in [0,2)$, let
\begin{align*}
F_{\alpha}(y,z) &= 
\sum_{l_1,l_2 \ge 1} \frac{1 -\ee^{-\lambda_{l_1,l_2}^{(\theta_0)} \Delta}}
{\lambda_{l_1,l_2}^{(\theta_0)} \mu_{l_1,l_2}^\alpha}
\cos(\pi l_1 y) \cos(\pi l_2 z)
\\
&= \nu^\alpha \Delta^{1+\alpha} 
\sum_{l_1,l_2 \ge 1} 
\frac{1 -\ee^{-\lambda_{l_1,l_2}^{(\theta_0)} \Delta}}
{\lambda_{l_1,l_2}^{(\theta_0)} \Delta}
\cdot \frac{1}{(\mu_{l_1,l_2} \nu \Delta)^\alpha}
\cos(\pi l_1 y) \cos(\pi l_2 z).
\end{align*}
Under $\nu \to 0$, the following lemma holds 
instead of Lemmas 4.10 and 4.11 in \cite{TKU2023arXiv}.

\begin{lem}\label{lemD1}
For $\alpha \in (0,2)$ and $\delta = r \sqrt{\nu \Delta}$, it holds that
\begin{align*}
F_{j',k'}^{j,k} &= 
\ee^{(-\kappa(\widetilde y_{j-1} +\widetilde y_{j'})
-\eta(\widetilde z_{k-1} +\widetilde z_{k'}))/2} 
D_{1,\delta}^2 D_{2,\delta}^2 
F_\alpha(\widetilde y_{j'-1} -\widetilde y_{j}, 
\widetilde z_{k'-1} -\widetilde z_{k})
\\
&\qquad +\OO(\nu^{\alpha} \Delta^{1+\alpha})
+\OO\Biggl( \nu^{\alpha -1/2} \Delta^{1/2 +\alpha}
\biggl(
\frac{\ind_{\{ j \neq j'\}}}{|j-j'|+1} +\frac{\ind_{\{ k \neq k'\}}}{|k-k'|+1}
\biggr) \Biggr),
\end{align*}
\begin{equation*}
D_{1,\delta}^2 D_{2,\delta}^2 
F_\alpha(\widetilde y_{j'-1} -\widetilde y_{j}, 
\widetilde z_{k'-1} -\widetilde z_{k}) =
\begin{cases}
\nu^{\alpha -1} \Delta^\alpha \psi_{r,\alpha} 
+ \OO(\nu^{\alpha-1} \Delta^{1+\alpha}),
& (j',k') = (j,k),
\\
\OO\Bigl(
\nu^{\alpha -1} \Delta^\alpha
\Bigl( \nu \Delta + \frac{1}{(|j-j'| +1)(|k-k'| +1)} \Bigr) 
\Bigr),
& (j',k') \neq (j,k).
\end{cases}
\end{equation*}
Furthermore, it follows that under [A2], 
\begin{align*}
\cov[T_{i,j,k} X, T_{i',j',k'} X] &= 
\OO\Biggl(
\frac{\nu^{\alpha -1} \Delta^\alpha}{|i-i'| +1}
\biggl( \nu \Delta + \frac{1}{(|j-j'| +1)(|k-k'| +1)} \biggr) 
\Biggr)
\\
&\qquad
+\OO\Biggl( \frac{\nu^{\alpha -1/2} \Delta^{1/2 +\alpha}}{|i-i'|+1}
\biggl(
\frac{\ind_{\{ j \neq j'\}}}{|j-j'|+1} +\frac{\ind_{\{ k \neq k'\}}}{|k-k'|+1}
\biggr) \Biggr),
\end{align*}
\begin{equation*}
\sum_{i,i'=1}^N \cov[T_{i,j,k}X, T_{i',j,k}X]^2
=\OO (\nu^{2\alpha-2} N \Delta^{2\alpha})
\quad\text{uniformly in } j, k,
\end{equation*}
\begin{equation*}
\sum_{k,k'=1}^{m_2} \sum_{j,j'=1}^{m_1} \sum_{i,i'=1}^N 
\cov[T_{i,j,k}X, T_{i',j',k'}X]^2
=\OO(\nu^{2\alpha-2} m_1 m_2 N \Delta^{2\alpha}),
\end{equation*}
\begin{equation*}
\sum_{i,i'=1}^N \cov[(T_{i,j,k}X)^2, (T_{i',j,k}X)^2 ]
=\OO (\nu^{2\alpha-2} N \Delta^{2\alpha})
\quad \text{uniformly in } j, k,
\end{equation*}
\begin{equation*}
\sum_{k,k'=1}^{m_2} \sum_{j,j'=1}^{m_1} \sum_{i,i'=1}^N 
\cov[(T_{i,j,k}X)^2, (T_{i',j',k'}X)^2]
=\OO (\nu^{2\alpha-2} m_1 m_2 \sqrt{m_1 m_2 N} \Delta^{2\alpha}).
\end{equation*}
\end{lem}

Proposition \ref{prop3} is immediately obtained from Lemma \ref{lemD1}.
By the Taylor expansion, we have
\begin{equation*}
-N K_{\nu,N,m}'((\sigma^*)^2)
= \int_0^1 K_{\nu,N,m}''((\sigma^*)^2 + u(\hat \sigma^2 -(\sigma^*)^2)) \dd u
N(\hat \sigma^2 -(\sigma^*)^2).
\end{equation*}
Since Proposition \ref{prop3} and Lemma \ref{lemD1} imply that
$N K_{\nu,N,m}'((\sigma^*)^2) = \OO_\PP(1)$ and
there exists $V >0$ such that 
$\int_0^1 K_{\nu,N,m}''((\sigma^*)^2 + u(\hat \sigma^2 -(\sigma^*)^2)) \dd u \pto V$,
we get Proposition \ref{prop4}.


\begin{thebibliography}{}
\bibitem{Altmeyer_etal2022}
R. Altmeyer, T. Bretschneider, J. Jan\'{a}k and M. Rei\ss{}. (2022).
\newblock Parameter Estimation in an SPDE Model for Cell Repolarization.
\newblock {\em SIAM/ASA Journal on Uncertainty Quantification}, 10(1):179--199.


\bibitem{Altmeyer_etal2022arXiv}
R. Altmeyer, A. Tiepner and M. Wahl. (2022).
\newblock Optimal parameter estimation for linear SPDEs from multiple measurements.
\newblock{\em arXiv preprint arXiv:2211.02496}.


\bibitem{Bibinger_Bossert2023}
M. Bibinger and P. Bossert. (2023). 
\newblock Efficient parameter estimation for parabolic SPDEs based on 
a log-linear model for realized volatilities. 
\newblock {\em Japanese Journal of Statistics and Data Science}, 6:407--429. 


\bibitem{Bibinger_Trabs2020}
M. Bibinger and M. Trabs. (2020).
\newblock Volatility estimation for stochastic {PDE}s 
using high-frequency observations.
\newblock {\em Stochastic Processes and their Applications}, 130(5):3005--3052.


\bibitem{Bossert2023arXiv}
P.~Bossert. (2023).
\newblock Parameter estimation for second-order SPDEs in multiple space dimensions.
\newblock{ \em arXiv preprint arXiv:2310.17828}.


\bibitem{Chong2020}
C. Chong. (2020).
\newblock High-frequency analysis of parabolic stochastic {PDE}s.
\newblock {\em The Annals of Statistics}, 48(2):1143--1167.


\bibitem{Cialenco2018}
I. Cialenco. (2018).
\newblock Statistical inference for {SPDE}s: an overview.
\newblock {\em Statistical Inference for Stochastic Processes}, 21(2):309--329.


\bibitem{Cialenco_etal2020}
I. Cialenco, F. Delgado-Vences and H.J. Kim. (2020).
\newblock Drift estimation for discretely sampled {SPDE}s.
\newblock {\em Stochastics and Partial Differential Equations: 
Analysis and Computations}, 8:895--920.


\bibitem{Cialenco_Glatt-Holtz2011}
I. Cialenco and N. Glatt-Holtz. (2011).
\newblock Parameter estimation for the stochastically perturbed 
{Navier-Stokes} equations.
\newblock {\em Stochastic Processes and their Applications}, 121(4):701--724.


\bibitem{Dawson1975}
D.A. Dawson. (1975).
\newblock Stochastic evolution equations and related measure processes.
\newblock {\em Journal of Multivariate Analysis}, 5(1):1-52.


\bibitem{Denkowski_etal2003}
Z. Denkowski, S. Mig{\'o}rski and N. S. Papageorgiou. (2003).
\newblock {\em An introduction to nonlinear analysis: theory}.
\newblock Kluwer Academic/Plenum, New York.


\bibitem{Gaudlitz_Reiss2023}
S. Gaudlitz and M. Reiss. (2023).
\newblock Estimation for the reaction term in semi-linear SPDEs
under small diffusivity.
\newblock {\em Bernoulli 29(4):3033-3058}.


\bibitem{Gloter_Sorensen2009}
A. Gloter and M. S{\o}rensen. (2009).
\newblock Estimation for stochastic differential equations with a small 
diffusion coefficient.
\newblock {\em Stochastic Processes and their Applications}, 119:679--699.


\bibitem{Hairer2009}
M. Hairer. (2009).
\newblock An introduction to Stochastic PDEs. 
\newblock {\em arXiv preprint arXiv:0907.4178}.


\bibitem{Hildebrandt_Trabs2021}
F. Hildebrandt and M. Trabs. (2021).
\newblock Parameter estimation for {SPDE}s based on discrete observations 
in time and space.
\newblock {\em Electronic Journal of Statistics}, 15(1):2716--2776.


\bibitem{Hildebrandt_Trabs2023}
F. Hildebrandt and M. Trabs. (2023).
\newblock Nonparametric calibration for stochastic reaction-diffusion equations
based on discrete observations, 
{\em Stochastic Processes and their Applications}, 162:171--217.


\bibitem{Hubner_etal1993}
M. H\"{u}bner, R. Khasminskii and B.L. Rozovskii. (1993).
\newblock {\em Two Examples of Parameter Estimation for Stochastic Partial
  Differential Equations}, 149--160.
\newblock Springer New York.


\bibitem{Hubner_Rozovskii1995}
M. H\"{u}bner and B.L. Rozovskii. (1995).
\newblock On asymptotic properties of maximum likelihood estimators 
for parabolic stochastic PDE's.
\newblock {\em Probability Theory and Related Fields}, 103(2):143--163.




\bibitem{Kaino_Uchida2021a}
Y. Kaino and M. Uchida. (2021).
\newblock Parametric estimation for a parabolic linear {SPDE} model based on
  discrete observations.
\newblock {\em Journal of Statistical Planning and Inference}, 211:190--220.


\bibitem{Kaino_Uchida2021b}
Y. Kaino and M. Uchida. (2021).
\newblock Adaptive estimator for a parabolic linear {SPDE} with a small noise.
\newblock {\em Japanese Journal of Statistics and Data Science}, 4:513--541.


\bibitem{Kessler1997}
M. Kessler. (1997).
\newblock Estimation of an ergodic diffusion from discrete observations.
\newblock {\em Scandinavian Journal of Statistics}, 24(2):211--229.


\bibitem{Kusuoka2000}
S. Kusuoka. (2000).
\newblock Term structure and SPDE,
\newblock In {\em Advances in Mathematical Economics}, 2:67--85, Springer.




\bibitem{Lototsky2003}
S.V. Lototsky. (2003).
\newblock Parameter estimation for stochastic parabolic equations: 
	asymptotic properties of a two-dimensional projection-based estimator.
\newblock {\em Statistical Inference for Stochastic Processes}, 6:65--87.


\bibitem{Markussen2003}
B. Markussen. (2003).
\newblock Likelihood inference for a discretely observed 
stochastic partial differential equation. 
\newblock {\em Bernoulli}, 9:745--762.


\bibitem{Moleriu2009}
R. Moleriu. (2009).
A generalization of the Girsanov theorems on the Hilbert space.
\textit{The Romanian Journal of Technical Sciences. Applied Mechanics}, 54(2), 113--124.


\bibitem{Nezza_etal2012}
E. Di Nezza, G. Palatucci and E. Valdinoci. (2012). 
\newblock Hitchhiker's guide to the fractional Sobolev spaces. 
\newblock{\em Bulletin des Sciences Math\'{e}matiques} 136:521-573.


\bibitem{Nourdin_etal2009}
I. Nourdin, G. Peccati and G. Reinert. (2009).
\newblock Second order Poincar\'{e} inequalities and CLTs on Wiener space.
\newblock{\em Journal of Functional Analysis}, 257:593-603.




\bibitem{Nualart2006}
D. Nualart. (2006). 
\newblock {\em The Malliavin Calculus and Related Topics}. 
\newblock Springer, 2nd edition.


\bibitem{Piterbarg_Ostrovskii1997}
L. Piterbarg and A. Ostrovskii. (1997).
\newblock {\em Advection and diffusion in random media: 
implications for sea surface temperature anomalies}.
\newblock Springer Science \& Business Media.


\bibitem{DaPrato_Zabczyk2014}
G. Da Prato and J. Zabczyk. (2014).
\newblock {\em Stochastic Equations in Infinite Dimensions}, 2nd edition.
\newblock Encyclopedia of Mathematics and its Applications. 
Cambridge University Press.


\bibitem{Sanz-Solo2005}
M. Sanz-Sol\'{o}. (2005). 
\newblock {\em Malliavin calculus 
with applications to stochastic partial differential equations}.
\newblock Fundamental sciences. Mathematics. EPFL Press, Lausanne.




\bibitem{Sorensen_Uchida2003}
M. S{\o}rensen and M. Uchida. (2003).
\newblock Small-diffusion asymptotics for discretely sampled stochastic 
differential equations.
\newblock {\em Bernoulli}, 9(6):1051--1069.


\bibitem{TKU2023a}
Y. Tonaki, Y. Kaino and M. Uchida. (2023).
\newblock Parameter estimation for linear parabolic {SPDEs} in two space
  dimensions based on high frequency data.
\newblock {\em Scandinavian Journal of Statistics}, 50(4):1568-1589.


\bibitem{TKU2023arXiv}
Y. Tonaki, Y. Kaino and M. Uchida. (2023).
\newblock Parametric estimation for linear parabolic SPDEs in two space dimensions 
based on temporal and spatial increments.
\newblock {\em arXiv preprint arXiv:2304.09441}.


\bibitem{TKU2024}
Y. Tonaki, Y. Kaino and M. Uchida. (2024).
\newblock Parameter estimation for a linear parabolic {SPDE} model in two space
  dimensions with a small noise.
\newblock {\em Statistical Inference for Stochastic Processes}, 27(1):123-179.


\bibitem{Tuckwell2013}
H.C. Tuckwell. (2013).
\newblock Stochastic partial differential equations in neurobiology: 
linear and nonlinear models for spiking neurons,
\newblock In {\em Stochastic biomathematical models}, 149--173. Springer.


\bibitem{Uchida_Yoshida2012}
M. Uchida and N. Yoshida. (2012).
\newblock Adaptive estimation of an ergodic diffusion process based on sampled data.
\newblock {\em Stochastic Processes and their Applications}, 122(8):2885--2924.


\bibitem{Walsh1986}
J.B. Walsh. (1986).
\newblock An introduction to stochastic partial differential equations.
In {\em \'{E}cole d'\'{E}t\'{e} de Probabilit\'{e}s de Saint Flour XIV-1984},
265--439, Springer.


\bibitem{Yagi2010}
A. Yagi. (2010).
\newblock {\em Abstract parabolic evolution equations and their applications}.
\newblock Springer monographs in mathematics. Springer. 


\bibitem{Yoshida1980}
K. Yoshida. (1980).
\newblock {\em Functional analysis}, 6th edition.
\newblock Springer, Berlin.


\bibitem{Yoshida2011}
N. Yoshida. (2011).
\newblock Polynomial type large deviation inequalities and quasi-likelihood analysis
for stochastic differential equations.
\newblock {\em Annals of the Institute of Statistical Mathematics}, 63:431--479.

\end{thebibliography}
\end{document}